\documentclass[a4paper,11pt]{article}
\usepackage[margin=24mm]{geometry}
\usepackage[T1]{fontenc}
\usepackage[utf8]{inputenc}
\usepackage{xcolor}

\newcommand{\EEE}{\color{black}}
\usepackage{bm,bef_alex}
\usepackage{epic,eepic,textcomp,tikz,bbm}
\usepackage{xcolor,scrtime,graphicx}
\usepackage{amsmath,amsfonts,amssymb,eucal,mathrsfs,wasysym,upgreek}
\usepackage{enumerate}
\numberwithin{equation}{section}

\newcommand{\mathbcal}[1]{{\bm#1}}
\newcommand{\YX}{\mathcal X}

\def \M{\mathcal M}

\newcommand{\LET}{\calL\!\calE\!\calT}
\renewcommand{\bfH}{H}
\renewcommand{\d}{\mathrm{d}}
\renewcommand{\bfLambda}{\boldsymbol{\Lambda}}
\renewcommand{\dd}{\mathbcal{\ell}}
\newcommand{\xx}{\boldsymbol{x}}
\newcommand{\zz}{\boldsymbol{z}}
\newcommand{\rr}{\boldsymbol{r}}
\newcommand{\ztzt}{\boldsymbol{\zeta}}
\newcommand{\mm}{\boldsymbol{\upmu}}
\newcommand{\lala}{\boldsymbol{\uplambda}}

\newcommand{\sisi}{\bfbeta}

\newcommand{\x}{\times}

\newcommand{\kc}{\kappa}

\makeatletter
\newcommand{\argmin}{\mathop{\operator@font argmin}}
\makeatother
\makeatletter
\newcommand{\astop}{\mathop{\operator@font \ast}}
\makeatother
\makeatletter
\newcommand{\DOP}{\mathop{\operator@font \Delta}}
\makeatother

\newcommand{\mfP}{\mathfrak P}
\newcommand{\ACCPOINTS}[1]{\bm{\calA\calP}(#1)}
\newcommand{\EQOVER}[2]{\overset{\text{\eqref{#1}}}{#2}}
\def\FG{\bm}
\def\mafo{\mathrm}
\def\M{\mathfrak M}

\newcommand{\sfW}{\mathsf W}
\newcommand{\HK}{\mathsf H \hspace{-0.09em} \mathsf K}
\newcommand{\SHK}{\mathsf S \hspace{-0.09em} \mathsf{H} \hspace{-0.09em} \mathsf{K}}
\newcommand{\SHKwo}{\mathsf{S\hspace{-0.18em} H\hspace{-0.25em} K}}
\newcommand{\HKwo}{\mathsf{H\hspace{-0.25em} K}}

\def\div{\mathop{\mafo{div}}}

\newcommand{\sfd}{\mathsf{d}}  
\newcommand{\Geod}{\operatorname {Geod}}
\newcommand{\curv}{\operatorname {curv}}
\newcommand{\dist}{\mathsf{dist}}
\newcommand{\supp}{\operatorname{supp}}
\newcommand{\cP}{\calP}

\newcommand{\dil}[1]{\operatorname{dil}_{#1}}
\newcommand{\prd}[1]{\operatorname{prd}_{#1}}

\def\mfo{\mathfrak o}

\def\mfC{\mathfrak C}

\def\mfP{\mathfrak P}

\newcommand{\res}[1]{\mathop{\hbox{\vrule height 7pt width .5pt depth 0pt
			\vrule height .5pt width 6pt depth 0pt}} #1}

\makeatletter
\renewcommand{\a}{\alpha}
\renewcommand{\b}{\beta}
\renewcommand{\M}{\mathcal{M}}

\makeatother

\begin{document}

\title{Geometric properties of cones with applications on the Hellinger-Kantorovich
  space, and a new distance on the space of probability measures}

\author{Vaios Laschos$^a$ and Alexander Mielke$^{a,b}$\\
{\footnotesize\begin{tabular}{@{}l}
  $^a$ Weierstra\ss{}-Institut f\"ur Angewandte Analysis und Stochastik,
   Berlin \\
  $^b$ Institut f\"ur Mathematik, Humboldt-Universit\"at zu Berlin
 \end{tabular}}
}

\date{May 15, 2018}
\maketitle

\begin{abstract}
  We study general geometric properties of cone spaces, and we apply them on the Hellinger--Kantorovich space
  $(\M(X),\HKwo_{\alpha,\beta}).$ We exploit a two-parameter scaling property of the Hellinger-Kantorovich metric $ \HKwo_{\alpha,\beta},$ and we prove the existence of a distance $\SHKwo_{\alpha,\beta}$ on the space of
  Probability measures that turns the Hellinger--Kantorovich space
  $(\M(X),\HKwo_{\alpha,\beta})$ into a cone space
  over the space of probabilities measures
  $(\cP(X),\SHKwo_{\alpha,\beta}).$ We provide a two parameter rescaling of geodesics in   $(\M(X),\HKwo_{\alpha,\beta}),$ and for $(\cP(X),\SHKwo_{\alpha,\beta})$ we obtain a full
  characterization of the geodesics.  We finally prove finer geometric
  properties, including local-angle condition and
  partial $K$-semiconcavity of the squared distances, that will be used in a future paper
  to prove existence of gradient flows on both spaces.
\end{abstract}

\section{Introduction}
\label{se:Intro}

In \cite{LiMiSa16OTCR,LiMiSa14?OETP}, and independently in
\cite{KoMoVo16NOTD} and \cite{CPSV15?UOTG,CPSV15?IDOT}, a new family
of distances $\HKwo_{\alpha,\beta}$ on the space $\M(X)$ of arbitrary
nonnegative and finite measures  was introduced, where
$(X,\sfd_X)$ is a geodesic, Polish space. This new family of
Hellinger--Kantorovich distances generalize both the
Kantorovich--Wasserstein distance (for $\alpha=1$ and $\beta= 0$)
and the Hellinger-Kakutani distance (for $\alpha=0$ and
$\beta=1$), allowing for both transportation and
creation/annihilation of mass, which is organized in a jointly
optimal fashion depending on the ratio of the parameters $\alpha$ and
$\beta$.  

The origin of our work stems from the observation in
\cite[Prop.\,19]{LiMiSa16OTCR} that the total mass $m(s)=\int_X
1\mathrm d \mu(s)$ of a constant-speed geodesic $[0,1]\ni s \mapsto
\mu(s) \in \M(X)$ is a quadratic function in $s$, viz.\ 
\begin{equation}\label{massgeodesic}
m(s) = (1{-}s)m(0) + s m(1) -s(1{-}s) \frac4\beta
\HKwo^2_{\alpha,\beta}(\mu(0),\mu(1)).
\end{equation}

We will show here that this formula is already a consequence of a simpler scaling property, 
that  \textbf{fully characterizes cone spaces}, which in the case of $\HKwo^2_{\alpha,\beta},$ takes the form 
\begin{equation}
  \label{eq:HK.scale}
  \HKwo^2_{\alpha,\beta}(r_0^2 \mu_0,r_1^2\mu_1)=r_0r_1 
  \HKwo^2_{\alpha,\beta}( \mu_0,\mu_1) +(r_0^2{-}r_0r_1)\frac4\beta
  \mu_0(X)+  (r_1^2{-}r_0r_1)\frac4\beta \mu_1(X).
\end{equation}
The property is proved independently in Theorem \ref{th:ScalHK} based on the
characterization of $\HKwo^2_{\alpha,\beta}$ via the
logarithmic-entropy functional $\LET_{\dd}$, cf.\ Theorem
\ref{thm:LET}. 

This suggests to write arbitrary measures $\mu\in
\M(X)\setminus\{0\}$ as 
\begin{equation}\label{cone structure of HK}
 \mu = r^2 \nu \quad \text{with} \quad [\nu,r] \in \cP(X)\x
(0,\infty), \quad \text{where  } 
r=\sqrt{\mu(X)}, \ 
\nu=\frac1{r^2} \mu ,
\end{equation}
and $\cP(X)$ denotes the probability measures. Thus, the set $\M(X)$ can be
interpreted as a cone over $\cP(X)$ in the sense of Section
\ref{se:Cones}, and the Hellinger--Kantorovich distance has the form 
\[
\HKwo_{\alpha,\beta}^2(r_0^2\nu_0 , r_1^2 \nu_1)= \frac{4}\beta 
  \Big( r_0^2 + r_1^2 - 2r_0r_1 \cos 
        \big(\SHK_{\alpha,\beta}(\nu_0,\nu_1)\big)\Big),   
\]
where the so-called \emph{spherical Hellinger--Kantorovich distance}
on $\cP(X)$ is simply defined by 
\[
\SHKwo_{\alpha,\beta}(\nu_0,\nu_1) = \arccos\left( 1- \frac{\frac\beta4
  \HKwo_{\alpha,\beta}^2(\nu_0, \nu_1)}2\right) .
\]
One main result is that $\SHKwo_{\alpha,\beta}$ is indeed a distance
on the space of probability measures, such that
the Hellinger--Kantorovich space $(\M(X),\HKwo_{\a,\b})$ is indeed a
cone space over the space of probability measures, namely  
$(\cP(X),\SHKwo_{\alpha,\beta})$. This distance is a generalization of
the spherical Hellinger distance, also called ``Fisher-Rao distance''
or ``Bhattacharya distance 1'' in
\cite[Sec.\,7.2+Sec.\,14.2]{DezDez09ED}, in a similar way that the
Hellinger-Kantorovich distance is a generalization of the Hellinger
distance. 

The fact that $\SHKwo_{\alpha,\beta}$ satisfies the triangle
inequality will be derived  in the abstract Section \ref{se:Cones}
for general distances $\sfd_\calC$ satisfying a 
scaling property as in \eqref{eq:HK.scale}.  We work on the 
cone $(\calC,\sfd_{\calC})$ over a general space $(\YX,\sfd_{\YX})$, and the  sole additional
assumption we need is that the distance $\sfd_\calC$ is bounded on the
set $\{\,[x,1]\,:\,  x\in \YX  \,\}\subset \calC$ by the
constant 2, see Theorem \ref{coneinverse}. The latter bound  follows easily
for the Hellinger-Kantorovich distance from 
\[
\frac\beta4 \HKwo^2_{\alpha,\beta}(\nu_0,\nu_1) \leq 
\frac\beta4\left( \frac4\beta \nu_0(X)+\frac4\beta \nu_1(X)\right) 
=  2 \leq 4. 
\] 

 In Sections \ref{su:GeodC} to \ref{su:Projecting} we consider the case
that $(\YX,\sfd_\YX)$ is a geodesic space and that $\sfd_\calC$ is given by 
\begin{equation}
  \label{conedistance}
\sfd^{2}_{\calC}([x_0, r_{0}],[x_1, r_{1}]) 
= r^{2}_{0}+r^{2}_{1}-2r_{0}r_{1}\cos_{\pi}(\sfd_{\YX}(x_{0},x_{1})),
\end{equation}
where $\cos_{a}(b)=\cos(\min\{a,b\})$. In Sections \ref{su:Lifting}
and \ref{su:Projecting} we show how geodesics in $(\calC,\sfd_\calC) $
between $[x_0, r_{0}] $ and $[x_1, r_{1}]$ can be obtained from those
between $x_0$ and $x_1$ in $(\YX,\sfd_\YX)$. 
Based on this, we discuss how comparison angles and local
angles behave when we move between the spherical space $(\YX,\sfd_\YX)$ and the
cone $(\calC,\sfd_\calC)$. In particular, we discuss the \emph{local
  angle condition} $m$-LAC,  see Definition \ref{def:m-LAC} and
\cite{Savare2007,OlPaVi14OTTA} for the usefulness of this in the
theory of metric gradient flows. The main observation is that if $\sfd_{\YX}(x_{0},x_{i})<\pi,$ $\xx_{0i}$ are constant-speed geodesics in $\YX$ connecting $x_0$ with
$x_i$, and if $\zz_{0i}$ are the corresponding geodesics in $\calC$
connecting $z_0=[x_0,r_0] $ and $z_i=[x_i,r_i]$ with $r_0,r_i>0$, then the
upper angles satisfy the relation 
\begin{align*}
\sfd_\calC(z_0,z_i)\sfd_\calC(z_0,z_j)  \cos \big(\varangle_{\rmu\rmp} (\zz_{0i},\zz_{0j})\big)&=
  (r_0{-}r_i\cos(\sfd_\YX(x_0,x_i)))(r_0{-}r_j\cos(\sfd_\YX(x_0,x_j)))\\&+
r_i r_j \sin(\sfd_\YX(x_0,x_i))\sin(\sfd_\YX(x_0,x_j)) \cos\big( \varangle_{\rmu\rmp}
(\xx_{0i},\xx_{0j})\big).
\end{align*}

Based on this, Theorem \ref{Cones and LAC} establishes that the $m$-LAC
condition transfers between 
$(\YX,\sfd_\YX)$ and $(\calC\setminus\{\bm0\}, \sfd_\calC)$. We conclude the second section by proving some
$K$-semiconcavity results. More specifically for any three points $x_{0},x_{1},x_{2}$ contained in a ball of radius $\mathfrak{D}<\frac{\pi}{2}$ we prove the following. if $\xx_{01}$ satisfies $K$-semiconcavity with respect to the \emph{observer} $x_{2},$ then for any $z_{0}=[x_{0},r_{0}],z_{1}=[x_{1},r_{1}],z_{2}=[x_{2},r_{2}],$ we have that $\zz_{01},$ satisfies $K'$-semiconcavity with respect to the \emph{observer} $z_{2},$ where $K'$ depends only on $K,r_{0},r_{1},r_{2},\mathfrak{D}.$ Conversely, if $r_{0}=r_{1}$ and $\zz_{01}$ satisfies $K$-semiconcavity with respect to the \emph{observer} $z_{2},$ then $\xx_{01}$ satisfies $K'$-semiconcavity with respect to the \emph{observer} $x_{2},$ where $K'$ depends only on $K,r_{0},r_{2},\mathfrak{D}.$

Section \ref{se:HK} shows that the abstract results apply in the
specific case of the Hellinger-Kantorovich space
$(\M(X),\HKwo_{\alpha,\beta})$, which takes the role of
$(\calC,\sfd_\calC)$, which then leads to the spherical space 
$(\cP(X),\SHKwo_{\alpha,\beta})$. A direct characterization in the sense
of \cite[Sec.\,8.6]{LiMiSa14?OETP} of the
geodesic curves using a continuity and a Hamilton-Jacobi equation in
the latter space is given in Theorem \ref{thm:SHK.Geod}.

In Section \ref{se:Angles} we provide additional geometric properties
that hold for both spaces. Among them, is the
local-angle condition, and some partial semiconcavity. In
\cite{LiMiSa16OTCR}, it was proved that $K$-semiconcavity, a property,
which is associated among other things with the existence of gradient
flows, does not hold in general. In this article, we prove that on the
subsets of measures that have bounded density (both from
below and above) with respect to some finite, locally doubling measure
$\mathcal{L}$, this property holds for sufficient large $K$ depending
only on the bounds and $\mathcal{L}$. This result will be used in a
consecutive paper to prove the existence of gradient flows. For this we 
provide a sharp estimate of the total mass of the calibration measure
associated with the optimal entropy-transport problem. This estimate
is used in our proofs, but it is also helpful for the numerical
approximations of the Hellinger-Kantorovich distance.

To simplify the subsequent notations we use the simple relation
$\HKwo^2_{\alpha,\beta}=\frac{1}{\beta}\HKwo^2_{\alpha/\beta,1}$,
which shows that it suffices to work with a one-parameter family. We
set $\HK_{\dd}^2= \HKwo^2_{1/\dd^2,4}$, which allows us to recover
$\HKwo_{\alpha,\beta}$ via $\HKwo^2_{\alpha,\beta}=\frac4\beta
\HK^2_\dd$ with $\dd^2=\beta/(4\alpha)$.

\section{Cones over metric spaces}
\label{se:Cones}

\subsection{Background and scaling property}
\label{su:ConesScaling}

In \cite{Berestovskii1983} (see also \cite{Aleksandrov1986},
\cite{BriHaf99MSNP}, and \cite{Burago2001}), the concept of the cone
$\calC$ over a metric space $(\YX,\sfd_{\YX}),$ is introduced.  The
cone is the quotient of the product $\YX\times[0,\infty),$ obtained by
identifying together all points in $\YX\times\{0\}$ with a point $\bm
0$, called the apex or tip of the cone. The cone $\calC$ is equipped
with the distance $\sfd_{\calC}$ given in
\eqref{conedistance}. In \cite{Burago2001}, one can find a proof
that $\sfd_{\calC}$ is a metric distance.  The following results
exhibits the scaling properties of such cone distances.

\begin{lemma}[Cone distances have scaling properties]\label{le:ConeScales}
The cone distance $\sfd_{\calC}$ in \eqref{conedistance} satisfies the scaling
property 
\begin{equation}
\label{asasas}
\forall\, [x_0,r_0],[x_1,r_1]\in \calC:\quad \sfd^{2}_{\calC}([x_0,
r_{0}],[x_1, r_{1}])=r_{0}r_{1} \sfd^{2}_{\calC}([x_0,1],[x_1, 1]) 
 + (r_{0}{-}r_{1})^2.
\end{equation}
Moreover, any distance $\sfd_{\calC}$ satisfying \eqref{asasas}
(i.e.\ without assuming \eqref{conedistance} a priori)  satisfies
the more general scaling property
\begin{equation}\label{eq:asasa}
\sfd^{2}_{\calC}([x_0, r_{0}\widetilde{r}_{0}],[x_1,
r_{1}\widetilde{r}_{1}])=r_{0}r_{1}\sfd^{2}_{\calC}([x_0,{r}_{0}],[x_1,
{r}_{1}]) +(\widetilde r^{2}_{0} {-} \widetilde r_{0}
\widetilde r_{1}){r}^{2}_{0}
+(\widetilde r^{2}_{1} {-} \widetilde r_{0} \widetilde r_{1}){r}^{2}_{1}
\end{equation}
for all $\widetilde{r}_{0}$ and $\widetilde{r}_{1}$.
\end{lemma} 
\begin{proof}
Statement \eqref{asasas} follows by using \eqref{conedistance} twice,
once as it is given, and once with $r_0=r_1=1$, and then eliminating
$\cos_{\pi}(\sfd_{\YX}(x_{0},x_{1}))$. 

Statement \eqref{eq:asasa} follows by using \eqref{asasas} twice,
once as it is given, and once with $r_0$ $r_1$ replaced by
$r_{0}\widetilde{r}_{0}$ and $r_{1}\widetilde{r}_{1}$,
respectively. After eliminating $\sfd^{2}_{\calC}([x_0,1],[x_1,1])$ the assertion follows.
\end{proof}

While we were studying the Hellinger-Kantorovich space, we noticed
that the scaling property \eqref{asasas} actually fully characterizes a cone
space. We have the following general theorem, which allows us to
derive the cone distance from the scaling property. 

\begin{theorem}[Scaling implies cone distance]\label{coneinverse}
  For a metric space $(\calC,\sfd_{\calC}),$ let assume that it exists a
  set $\YX,$ that could possibly be identified with a subset of $\calC,$
  and a surjective function
  $[\cdot,\cdot]:\YX\times[0,\infty) \to  \calC,$ such that the
  distance $\sfd_{\calC}$ satisfies \eqref{asasas} and 
\begin{equation}
  \label{eq:d.bound}
 \forall x_{0}\neq x_{1}\in \YX:\quad
 0<\sfd^{2}_{\calC}([x_{0},1],[x_{1},1])\leq 4 ;
\end{equation} 
then $\sfd_{\YX}:\YX\times \YX \to [0,\infty)$ given by $\sfd_{\YX}
(x_{0},x_{1}) = \arccos\left(1-\frac{\sfd^{2}_{\calC}([x_0,1] ,[x_1,
    1])}{2}\right) \in [0,\pi]$ is a metric distance on $\YX$, and
$(\calC,\sfd_\calC)$ is a metric cone over $(\YX,\sfd_\YX)$, i.e.\
\eqref{conedistance} holds. 
\end{theorem}
\begin{proof} Clearly, $\sfd_\YX$ as defind in the assertion is
  symmetric and positive. Hence, it remains to establish the triangle
  inequality. Given $x_{0},x_{1},x_{2}\in \YX,$ we
  set 
\[
D_{ij}=\sfd_{\calC}([x_i, 1],[x_j,
  1])\hspace{6pt}\text{and}\hspace{6pt}
  \phi_{ij}= \arccos\left(1-\frac{D^{2}_{ij}}{2}\right),\hspace{8pt}\text{for}\hspace{6pt}
  i\neq j\in\{0,1,2\}.
\] 
Hence, we have to show $\sfd_\YX(x_0,x_2)=\phi_{02}\leq \phi_{01} +
\phi_{12}=\sfd_\YX(x_0,x_1)+\sfd_\YX(x_1,x_2)$. If $\phi_{01} +
\phi_{12}\geq\pi$ then there is nothing to show. Without loss of
generality, we will have $\phi_{01} = \min\{\phi_{01},\phi_{12}\} <
\frac{\pi}{2},$ and $\phi_{01}+\phi_{12}<\pi$. We consider a
comparison triangle in $\mathbb{R}^{2}$, as is depicted in Figure
\ref{fig:Scale2}. In particular, $A_j$ are chosen on the unit circle
such that $\phi_{i,i+1}$ and $D_{i,i+1}$ are the angle (arclength on
the unit circle) and the Euclidean distance, respectively, between
$A_i$ and $A_{i+1}$. Now, $A_{*}$ is chosen as the intersection
of $\overline{OA_{1}}$ with the segment $\overline{A_{0}A_{2}}$, see
Figure \ref{fig:Scale2}.

\begin{figure}[t]
\centerline{
\begin{tikzpicture}[scale=1]
        \draw[thick,-] (0,0)--(10: 5);
	\draw[thick, -] (0,0)--(50: 5);
	\draw[thick, -] (0,0)--(155: 5);
	\draw[ultra thick, color=orange] (10 : 4) arc (10 : 155 : 4) ;
	\node[orange, right]  at (30:4.1) {$\phi_{12}= 
                    \arccos\big(1{-} \tfrac12 D_{12}^2\big)$} ; 
	\node[orange, above]  at (90:4.1) {$\phi_{01}= 
		    \arccos\big(1{-} \tfrac12 D_{01}^2\big)$} ; 
	\node (Z0) at (155 : 4) {};
	\fill (Z0) circle (0.15) node [below]  {$A_0\quad $};
	\node (Z2) at (10 : 4) {};
	\fill (Z2) circle (0.15) node [below]  {$\quad A_2$};
	\node (Z1) at (50:4) {};
	\fill (Z1) circle (0.15) node[ right] {$\ \  A_1$};
	\draw[violet!70, very thick] (Z0) -- (Z2); 
	\draw[violet!70, very thick] (Z0) -- (Z1) node[pos=0.6, above]
                {$D_{01}$} ;
	\draw[violet!70, very thick] (Z1) -- (Z2) node[pos=0.4, below]
		{$D_{12}\quad $} ;
	\fill (0,0) circle (0.15) node[below] {$(0,0)$}; 
	\node (Z star) at (50 : 1.45) {};
        \fill[blue] (Z star) circle (0.15) node [above] 
                {\raisebox{0.4em}{$A_*\quad$}};
        \draw [ultra thick, blue] (0,0) -- (Z star) 
                node[pos=0.5, right] {$\:r_*$} ;
\end{tikzpicture}}
\caption{Construction of the optimal radius $r_*$. The points $A_j$
  have distance $r_j=1$ from the origin and thus correspond to
  $z_j=[x_j,1]$, which gives
  $D_{1j}=|\overline{A_1A_j}|=\sfd_{\calC}(z_1,z_j)$ for $j=0$ and
  $2$. The point $A_*$, which corresponds to $z_*=[x_1,r_*]$, is
  chosen such that $|\overline{A_0A_*}| + |\overline{A_*A_2}| =
  |\overline{ A_0A_2}|$.}
\label{fig:Scale2}
\end{figure}
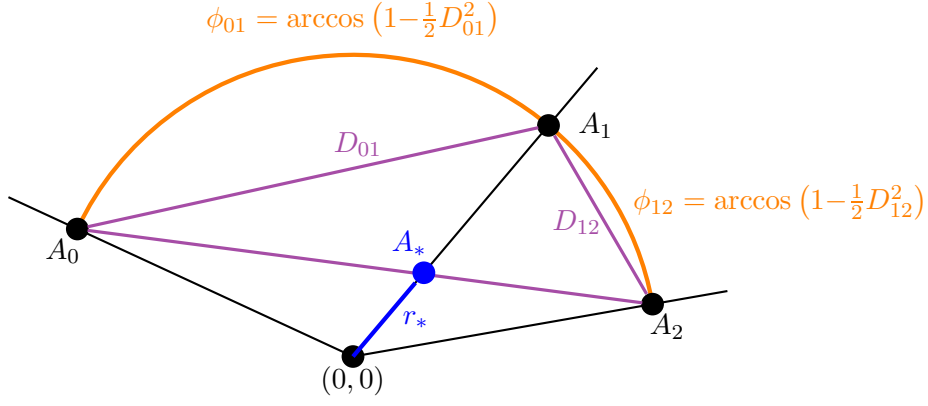 

With this choice of $r_*$ we retur to the cone $(\calC,\sfd_\calC)$
and let $r_{*}=|\overline{OA_{*}}|$ and $z_{0}=[x_{0},1],
z_{1}=[x_{1},r_{*}], z_{2}=[x_{2},1]\in \calC$.  The scaling 
property \eqref{asasas} for $\sfd_{\calC},$ gives
\begin{align*}
 \sfd^{2}_{\calC}([x_{0},1],[x_{1},r_{*}])&= 1 + r_{*}^{2}
  -2r_{*}\cos\phi_{01}=|\overline{A_{0}A_{*}}|^{2} 
\text{ and}\\
\sfd^{2}_{\calC}([x_{1\,},r_{*}],[x_{2},1])& =1 + r_{*}^{2}
-2r_{*}\cos\phi_{12}=|\overline{A_{2}A_{*}}|^{2}.
\end{align*}
Using the triangle inequality for $\sfd_{\calC},$ we arrive at
\begin{equation}
\begin{split}
D^{2}_{02}&=\sfd^{2}_{\calC}([x_{0},1],[x_{2},1])\leq
(\sfd_{\calC}([x_{0},1],[x_{1},r_{*}])
+\sfd_{\calC}([x_{1},r_{*}],[x_{2},1]))^{2}
\\
&= (|\overline{A_{0}A_{*}}|
+|\overline{A_{*}A_{2}}|)^{2} =|\overline{A_{0}A_{2}}|^{2}
=1+1-2\cos(\phi_{01} {+}\phi_{12}).  
\end{split}
\end{equation}
Thus, we conclude that
$\phi_{02}=\arccos\left(1-\frac{D_{02}^{2}}{2}\right)\leq \phi_{01} +
\phi_{12}$, which is the desired triangle inequality for $\sfd_\YX$,
namely $\sfd_\YX(x_0,x_2)\leq \sfd_\YX(x_0,x_1)+\sfd_\YX(x_1,x_2)$. Thus, inserting $\sfd_\calC^2([x_0,1],[x_1,1]) = 2- 2 \cos (
\sfd_\YX(x_0,x_1))$ into \eqref{asasas}, we have established
\eqref{conedistance}, and consequently \eqref{eq:asasa} follows as
well. 
\end{proof}

As a first consequence we obtain the following result.

\begin{corollary}
  \label{co:MetrCone}
  Let $\YX$ a set, and $\calC$ the quotient of the product
  $\YX\times[0,\infty),$ obtained by identifying together all points
  in $\YX\times{0}.$ If $\sfd_{\calC}:\calC\times \calC \to [0,\infty)$
  given by \eqref{conedistance}, for some
  $\sfd_{\YX}:\YX\times \YX \to [0,\infty)$ is a metric distance on
  $\calC$, then $\sfd_{\YX}\wedge\pi$ is a metric distance on $\YX.$
\end{corollary}
\begin{proof}
  By setting $z_{0}=[x_{0},1],$ and $z_{1}=[x_{1},1],$ we can recover
  both the positivity and symmetry property. For the
  proof of the triangle inequality, we just notice that $\sfd_{\calC}$
  satisfies the scaling property, and then the result is an
  application of Theorem \ref{coneinverse}.
\end{proof}

From the perspective of $(\YX,\sfd_{\YX}),$ we call
$(\calC,\sfd_{\calC})$ the cone space over $\YX$; from the perspective
of $(\calC,\sfd_{\calC}),$ we call $(\YX,\sfd_{\YX}\wedge\pi)$ the spherical
space in $\calC.$

\subsection{Geodesics curves}
\label{su:GeodC}

We first recall the standard definition and hence introduce our
notations.

\begin{definition}\label{def:length}
  Let $(\YX,\sfd_{\YX})$ be a metric space, and
  $\xx:[0,\tau]\rightarrow \YX,$ a continuous mapping. Furthermore,
  let $\mathcal{T}$ be the set of all partitions $T=\{0=\tau_{0}\leq
  \dots\leq\tau_{n_{T}}=\tau\}$ of $[0,\tau].$ Then, the \emph{length of the
    curve} $\xx$ is given by $ \mathrm{Len}(\xx):= \sup_{T\in\mathcal{T}}
  \sum_{i=1}^{n_{T}} \sfd_{\YX}(\xx(\tau_{i}),\xx(\tau_{i-1}))$.
\end{definition}

\begin{definition}
  Let $(\YX,\sfd_{\YX})$ be a metric space. We will call
  $(\YX,\sfd_{\YX})$ \emph{geodesic}, if and only if for every two points
  $x_{0},x_{1}$ there exists a continuous mapping
  $\xx_{01}:[0,\tau]\rightarrow \YX$ such that
\[
\xx_{01}(0)=x_{0},\quad \xx_{01}(\tau)=x_{1}, \quad
\text{and}\hspace{12pt}\sfd_{\YX}(x_{0},x_{1})=\mathrm{Len}(\xx_{01}).
\]
A function like that will be called a \emph{geodesic curve} or simply
a \emph{geodesic}. A geodesic satisfying
\[\sfd_{\YX}(\xx_{01}(t_{1}),\xx_{01}(t_{2}))=C|t_{2}{-}t_{1}|\]
for some constant $C>0,$ will be called a \emph{constant-speed geodesic}. If
$C=1,$ then the geodesic is called a \emph{unit-speed geodesic}. Finally for
$x_{0},x_{1}\in\YX,$ any geodesic $\xx_{01}:[0,1]\rightarrow \YX,$
with $\xx_{01}(0)=x_{0},\, \xx_{01}(1)=x_{1}$ is called a \emph{geodesic
joining} $x_{0}$ to $x_{1}.$ We will denote the set of all such geodesics with $\mathrm{Geod}(x_0,x_1)$, i.e.  
\begin{equation}\mathrm{Geod}(x_0,x_1):=\set{\xx:[0,1]\to \YX}{\xx(0)=x_0,\
    \xx(1)=x_1,\, \xx \text{ is constant-speed geodesic }\!\!}.\end{equation}
\end{definition}
In \cite[Chap.\,I, Prop.\,5.10]{BriHaf99MSNP}, the following Theorem  is proved. 

\begin{theorem}\label{antetwra}
 Let $(\YX,\sfd_{\YX})$ be a geodesic space. Let also $z_{0}=[x_{0},r_{0}]$ and  $z_{1}=[x_{1},r_{1}]$ be elements of $\calC.$
\begin{enumerate}
	\item If $r_{0},r_{1}\in(0,\infty)$ and $\sfd_{\YX}(x_{0},x_{1})<\pi,$ then there is a bijection between $\mathrm{Geod}(x_0,x_1),$ 
	and $\mathrm{Geod}(z_0,z_1).$ \item In all other cases, $\mathrm{Geod}(z_0,z_1)$ has a unique element.
\end{enumerate}	
\end{theorem}

As a corollary, we get that $\calC$ is geodesic, if and only if $\YX$
is geodesic for points of distance less than $\pi$. In the
following two Subsections
\ref{su:Lifting} and \ref{su:Projecting} we give explicit
correspondences in the sense of part 1.\ of the above theorem for the
case of constant-speed geodesics.

\subsection{Lifting from $\YX$ into the cone}
\label{su:Lifting}

  In
\cite{LiMiSa16OTCR}, it is proved that the constant-speed geodesics $\zz_{01}(t)$
connecting $z_{0}=[x_{0},r_{0}]$ to $z_{1}=[x_{1},r_{1}],$ with
$0<\sfd_{\YX}(x_{0}, x_{1})<\pi,$ have the following parametrization
\begin{equation}
\zz_{01}(t)=[\xx_{01}(\ztzt_{01}(t)),\rr_{01}(t)],
\end{equation}
where $\xx_{01}(t)$ is a constant-speed geodesic joining $x_{0}$ to
$x_{1}$ and where
$\ztzt_{01}(t)$ and $\rr_{01}(t)$ are given by
\begin{equation}\label{quadraticmass}
\begin{split}
\rr_{01}^{2}(t)&=(1{-}t)^{2}r^{2}_{0}+ t^{2}r^{2}_{1}+2t (1{-}t) r_{0}r_{1}\cos(\sfd_{\YX}(x_{0},x_{1})),\\
\ztzt_{01}(t)& = \frac1{\sfd_{\YX}(x_{0},x_{1}) } \arcsin \left(
\frac{t r_1 \sin\big(\sfd_{\YX}(x_{0},x_{1})\big)}{\rr_{01}(t)} \right) \\&=\frac1{\sfd_{\YX}(x_{0},x_{1}) } 
\arccos\left(\frac{(1{-}t)r_{0}+t\,r_{1} 
   \cos(\sfd_{\YX}(x_{0},x_{1}))}{\rr_{01}(t)}\right)\\
&  = \frac1{\sfd_{\YX}(x_{0},x_{1}) } \arctan \left(
\frac{t r_1 \sin\big(\sfd_{\YX}(x_{0},x_{1})\big)}{(1{-}t)r_0 + t\,r_1
\cos\big(\sfd_{\YX}(x_{0},x_{1})\big)} \right) .
\end{split}
\end{equation}
Alternatively if we want the parametrization with respect to $\sfd_{\calC},$
\eqref{quadraticmass} becomes 
\begin{equation}\label{quadraticmass2}
\begin{aligned}
\rr_{01}^{2}(t)&=\left((1{-}t) r_{0}+ tr_{1}\right)^{2} -
r_{0}r_{1}t(1{-}t)\sfd^{2}_{\calC}([x_0,1] ,[x_1, 1])\\ 
\ztzt_{01}(t)&= \frac1{\sfd_{\YX}(x_{0},x_{1}) } \arccos\left(\frac{(1{-} t)
      r_{0}+t\,r_{1}\left(1{-}\frac{\sfd^{2}_{\calC}([x_0,1] ,[x_1,
          1])}{2}\right)}{\rr_{01}(t)}\right). 
\end{aligned}
\end{equation}

If we differentiate twice the first equation in \eqref{quadraticmass}, we get 
$$(\rr_{01}^{2})''(t)=r^{2}_{0}+r^{2}_{1}-2r_{0}r_{1} \cos(\sfd_{\YX}(x_{0},x_{1}))=\sfd^{2}_{\calC}(z_{0},z_{1}),$$ 
from which we also recover the following formula 
\begin{equation}\label{rr prop}
\rr_{01}^{2}(t)=(1{-}t)r^{2}_{0}+tr^{2}_{1}-t (1{-}t) \sfd^{2}_{\calC}(z_{0},z_{1}), 
\end{equation}
which later applied to $\HKwo_{\a,\b}$ will give \eqref{massgeodesic}. Furthermore \eqref{rr prop}, trivially gives 
convexity of $\rr^{2}_{01},$ i.e. 
\begin{equation}\label{rr prop2}
\rr_{01}^{2}(t)\leq(1{-}t)r^{2}_{0}+tr^{2}_{1}. 
\end{equation}
Finally for the case where $\sfd_{\YX}(x_{0},x_{1})\leq\frac{\pi}{2},$ we get
\begin{equation}\label{rr prop3}
\rr_{01}^{2}(t)\geq(1{-}t)^{2}r^{2}_{0}+t^{2}r^{2}_{1}\geq \frac{1}{2}\min\{r^{2}_{0},r^{2}_{1}\}. 
\end{equation}

\subsection{Projecting from cone to $\YX$}
\label{su:Projecting}

We are now going to provide the inverse parametrization of the
geodesics in $(\YX,\sfd_{\YX}),$ with respect to the geodesics in $(\calC,\sfd_{\calC}).$

\begin{theorem}
\label{th:Project}
For $x_0, x_1\in \YX,$ with $0<\!\sfd_{\YX}(x_{0},x_{1})<\!\pi,$ and
$r_0,r_1>0$ consider $\zz_{01}\in \mathrm{Geod}(z_0,z_1),$ where $z_{0}=[x_{0},r_{0}], z_{1}=[x_{1},r_{1}].$ Then,
\begin{equation}
  \label{eq:eta.muB}
t\mapsto\xx_{01}(t) =\overline{\xx}_{01}\big(\sisi_{01}(t)\big)
  \text{ with } 
 \sisi_{01} (t)=\frac{r_0 \sin\big(t \sfd_{\YX}(x_0,x_1)\big) }{r_1\sin\!\big((1{-}t)\sfd_{\YX}(x_0,x_1)\big) {+} 
 r_0\sin\!\big(t\sfd_{\YX}(x_0,x_1)\big) }
\end{equation}
is an element of $\mathrm{Geod}(x_0,x_1).$ Furthermore 
\begin{equation}\label{teleutaiateleutaia}
\rr_{01}(\sisi_{01}(t))=\frac{r_{0}r_{1}\sin(\sfd_{\YX}(x_0,x_1))}{r_{1}\sin( (1{-}t) \sfd_{\YX}(x_0,x_1))+r_{0}\sin(t\sfd_{\YX}(x_0,x_1))}.
\end{equation}
\end{theorem} 

\begin{proof}
Since, by  the proof of Theorem \ref{antetwra}, $\zz_{01}$ is a geodesic in $(\calC,\sfd_{\calC}),$ if and only if  $\xx_{01}$ is a geodesic in
$(\YX,\sfd_{\YX})$ and $\zz_{01}(t)=[\xx_{01}(\ztzt_{01}(t)),\rr_{01}(t)],$ we just have to calculate the inverse of $\ztzt_{01}.$ 

By using the third representation in \eqref{quadraticmass}, we get \begin{equation}\label{blabla}
 \tan\left(\ztzt_{01}(t)\sfd_{\YX}(x_0,x_1)\right)=\frac{tr_{1}\sin(\sfd_{\YX}(x_0,x_1))}{ (1{-}t) r_{0}+tr_{1}\cos(\sfd_{\YX}(x_0,x_1))}.
\end{equation}
Let $\sisi_{01}$ be the inverse of $\ztzt_{01}.$ By composing every elemet of \eqref{blabla} with $\sisi_{01},$ we get
\begin{equation*}
\tan\left(t\sfd_{\YX}(x_0,x_1)\right)=\frac{\sisi_{01}(t)r_{1}\sin\left(\sfd_{\YX}(x_0,x_1)\right)}
{(1-\sisi_{01}(t))r_{0}+\sisi_{01}(t)r_{1}\cos(\sfd_{\YX}(x_0,x_1))},
\end{equation*}
which gives
\begin{equation*}
 \sisi_{01}(t)=\frac{r_{0}\tan\left(t\sfd_{\YX}(x_0,x_1)\right)}{r_{1}\sin(\sfd_{\YX}(x_0,x_1))+r_{0}\tan(t\sfd_{\YX}(x_0,x_1))-r_{1}
 \tan(t\sfd_{\YX}(x_0,x_1))\cos(\sfd_{\YX}(x_0,x_1))}.
\end{equation*}
Multiplying both the nominator and denominator with $\cos(t\sfd_{\YX}(x_0,x_1)),$ we get
\begin{equation*}
 \sisi_{01}(t)=\frac{r_{0}\sin\left(t\sfd_{\YX}(x_0,x_1)\right)}{r_{1}\sin(\sfd_{\YX}(x_0,x_1))\cos(t\sfd_{\YX}(x_0,x_1))+\sin(t\sfd_{\YX}(x_0,x_1))(r_{0}-r_{1}\cos(\sfd_{\YX}(x_0,x_1)))}
\end{equation*}
and by an application of $\sin(a)\cos(b)-\cos(a)\sin(b)=\sin(a-b),$ we get \eqref{eq:eta.muB}.

Now by using the first representation of \eqref{quadraticmass}, we get
\begin{equation*}
\sin(t\sfd_{\YX}(x_0,x_1))=\frac{\sisi_{01}(t)r_{1}\sin\left(\sfd_{\YX}(x_0,x_1)\right)}{\rr_{01}(\sisi_{01}(t))},
\end{equation*}
and combining with \eqref{eq:eta.muB} we get  \eqref{teleutaiateleutaia}. 
\end{proof}


Finally, we are now interested in the scaling properties of
constant-speed geodesics on $\calC$ is we simple change the radius of
$z_j=[x_j,r_j]$ into $r_j\widetilde r_j$. We will show 
that the constant-speed geodesic curves behave nicely under the
two-parameter rescaling. In the sequel, for $z=[x,r]\in\calC,$ and
$\widetilde{r}>0,$ we denote with $\widetilde{r}z,$ the element
$[x,r\widetilde{r}]\in \calC.$

\begin{proposition}
  \label{pr:He.Geod.Sca} For $z_0=[x_{0},r_{0}], z_1=[x_{1},r_{1}]\in
  \calC$ and $\widetilde{r}_0,\widetilde{r}_1\geq 0,$ we have that if
  $\zz_{01}(\cdot)=[\overline{\xx}_{01}(\cdot),\rr_{01}(\cdot)]$ belongs in $\mathrm{Geod}(z_0,z_1),$ then
  $\widetilde{\zz}_{01}(\cdot)= A_{01}(\cdot)\zz_{01}(B_{01}(\cdot)),$ with
\begin{equation}\label{eq:He.Geod.Sca}
  A_{01}(t)=\widetilde{r}_0+\left(\widetilde{r}_1{-}\widetilde{r}_0\right)t
  \hspace{4pt}\text{ and}\hspace{4pt}
  B_{01}(t)=\frac{\widetilde{r}_1t}{A_{01}(t)} ,
\end{equation}
is an element of $\mathrm{Geod}(\widetilde{r}_{0}z_{0},\widetilde{r}_{1}z_{1})$.
\end{proposition}
\begin{proof} We first observe
  $\widetilde{\zz}_{01}(0)=\widetilde{r}_0z_0$ and
  $\widetilde{\zz}_{01}(1)=\widetilde{r}_1 z_1$, because
  $A_{01}(0)=\widetilde{r}_0$ and $A_{01}(1)=\widetilde{r}_1$.  Thus,
  to check that $t\mapsto \ol{\zz}_{01}(t) $ is a geodesic it suffices
  to show
	\[
	\sfd_{\calC}(\ol{\zz}_{01}(0),\ol{\zz}_{01}(t))= t\,
	\sfd_{\calC}(\ol{\zz}_{01}(0),\ol{\zz}_{01}(1))=t\,\sfd_{\calC}(r_0z_0,r_1z_1),
	\]
	i.e.\ $\ol{\zz}_{01}$ is a constant-speed geodesic. However, using \eqref{rr prop}, we first observe 
	\begin{equation}
	\label{eq:Gam.Mass}
	\rr_{01}^{2}(B_{01}(t))= (1{-}B_{01}(t)) r^{2}_0 + 
	B_{01}(t) r^{2}_1 - B_{01}(t)(1{-}B_{01}(t))\sfd^{2}_{\calC}(z_0,z_1).
	\end{equation}
 With this, the abbreviation
	$a_t=A_{01}(t)$, and the relations $B_{01}(t)=\frac{\widetilde{r}_1t}{a_t}$ and
	$1{-}B_{01}(t)=\frac{\widetilde{r}_0(1{-}t)}{a_t}$ we obtain
	\begin{align*}
	&\sfd^{2}_{\calC}(\ol{\zz}_{01}(0),\ol{\zz}_{01}(t))= \sfd^{2}_{\calC}\big(\widetilde{r}_0z_0,a_t
	\zz_{01}(B_{01}(t))\big)   \\
&\quad\EQOVER{asasas}= \widetilde{r}_0 a_t
	\sfd^{2}_{\calC}(z_0, \zz_{01}(B_{01}(t))) + \widetilde{r}_0
        (\widetilde{r}_0{-}a_t)r^{2}_0 + 
	a_t(a_t{-}\widetilde{r}_0)\rr_{01}^{2}(B_{01}(t)) \\
&\underset{\text{\eqref{eq:Gam.Mass}}}{\overset{\zz_{01}
        \text{\,is geod.}}= }
	\widetilde{r}_0a_t \tfrac{\widetilde{r}^{2}_1t^2}{a_t^2} 
	\sfd^{2}_{\calC}(z_0,z_1) +\widetilde{r}_0
        (\widetilde{r}_0{-}a_t)r^{2}_0 
       + a_t(a_t{-}\widetilde{r}_0)\big(
        \tfrac{\widetilde{r}_0(1{-}t)}{a_t} r^{2}_{0} {+} 
	\tfrac{\widetilde{r}_1 t}{a_t} r^{2}_1
	{-}\tfrac{\widetilde{r}_0 \widetilde{r}_1
          t(1{-}t)}{a_t^2}\sfd^{2}_{\calC}(z_0,z_1) \big) \\ 
&\qquad \overset{*}= \widetilde{r}_0\widetilde{r}_1 t^2 \sfd^{2}_{\calC}(z_0,z_1)  + (\widetilde{r}_0^2{-}\widetilde{r}_0\widetilde{r}_1)t^2 r^{2}_0 +
	(\widetilde{r}_1^2{-}\widetilde{r}_0\widetilde{r}_1) t^2 r^{2}_1
	\EQOVER{asasas}= t^2
	\sfd^{2}_{\calC}(\widetilde{r}_0z_0,\widetilde{r}_1z_1)= t^2 
	\sfd^{2}_{\calC}(\ol{z}_{0},\ol{z}_1),  
	\end{align*} 
	where in $\overset{*}=$ we simply used the definition of
	$a_s=A_{01}(s)$. Thus, the assertion is shown.
\end{proof}

\subsection{Comparison and local angles}
\label{se:Comp.Angles}

We now introduce comparison angles, see e.g.\ \cite{Stur99MSLB, Burago2001, S.AlexanderV.Kapovitch}, that
are used to study notions of curvature and their properties, and subsequentially be utilized to generate gradient flows 
on metric spaces, cf.\ \cite{Ohta2009a, S.AlexanderV.Kapovitch, Savare2007,OlPaVi14OTTA}.  
Since we relate the space $(\YX,\sfd_\YX)$ with the cone $(\calC,\sfd_\calC),$ we will see  in the
next subsection (cf.\ the proof of Theorem \ref{Cones and LAC})) that it is natural to use comparison angles
$\widetilde{\varangle}_{\kc}$ for  different $\kc$ on these two spaces.

\begin{definition}[Comparison angles]\label{comp}
Let $(\YX,\sfd_\YX)$ be a metric space and  $x_{0},x_{1},x_{2}\in
\YX$ with $x_{0}\not\in\{ x_{1},x_{2}\}$. For $\kc \in \R$ we define
$a_{\kc}$ via   
\begin{equation*}
	a_{\kc}(x_{0};x_{1},x_{2}):=
\begin{cases}
\displaystyle
\frac{\sfd_{\YX}^{2}(x_{0},x_{1})+\sfd_{\YX}^{2}(x_{0},x_{2})-\sfd_{\YX}^{2}(x_{1},x_{2})}
 {2\sfd_{\YX}(x_{0},x_{1})\sfd_{\YX}(x_{0},x_{2})}
&\text{for } \kc=0,
\\[1em]
\displaystyle	\frac{\cos(\sqrt{\kc}\,\sfd_{\YX}(x_{1},x_{2})) -
  \cos(\sqrt{\kc}\, \sfd_{\YX}(x_{0},x_{1})) 
  \cos(\sqrt{\kc}\,\sfd_{\YX}(x_{0},x_{2}))} 
 {\sin(\sqrt{\kc}\,\sfd_{\YX}(x_{0},x_{1})) 
  \;\sin( \sqrt{\kc}\,\sfd_{\YX}(x_{0},x_{2}))}&
\text{for }\kc>0,
\\[1em]
\displaystyle	
\frac{ \cosh( \mathsf{k}\, \sfd_{\YX}(x_{0},x_{1}))  
   \cosh(\mathsf{k}\, \sfd_{\YX}(x_{0},x_{2})) 
 -\cosh(\mathsf{k}\,\sfd_{\YX}(x_{1},x_{2})) }
{\sinh( \mathsf{k}\, \sfd_{\YX}(x_{0},x_{1}))\;
 \sinh( \mathsf{k}\, \sfd_{\YX}(x_{0},x_{2}))}&\text{for } \kc<0,
	\end{cases}
\end{equation*}
where $\mathsf{k}=\sqrt{-\kc}$. 
	The \emph{$\kc$-comparison angle $\widetilde{\varangle}_{\kc}(x_{0};x_{1},x_{2})\in[0,\pi]$ with vertex $x_{0}$} is defined by the formula
\begin{equation*}
	\widetilde{\varangle}_{\kc}(x_{0};x_{1},x_{2})=\arccos(a_{\kc}(x_{0};x_{1},x_{2})). 
\end{equation*}
\end{definition}

From now on, the value of $\kc$ in the previous definition will be
refereed as the choice of model space $\mathbb{M}^{2}(\kappa).$ This terminology is borrowed
from the study of Alexandrov spaces, where the sphere ($\kc>0$), the
plane ($\kc=0$), and the hyberbolic plane ($\kc<0$) are used as
reference, cf.\ \cite{Stur99MSLB, Burago2001, S.AlexanderV.Kapovitch}. Later, our main choice will be
$\kc=1$ on the spherical space (hence the name) $(\YX,\sfd_\YX)$ and
$\kc=0$ on the cone $(\calC,\sfd_\calC)$.

Let $\xx_{01}$ and $\xx_{02},$ be
two geodesics in 
$(\YX,\sfd_{\YX}),$ emanating from the same initial point
$x_{0}:=\xx_{01}(0)=\xx_{02}(0).$ The following theorem guarantees
that the set 
\begin{equation}\label{eq:AccPoints}
\ACCPOINTS{\xx_{01},\xx_{02}}:=\set{c \in [-1,1]}{\exists\, 0<s_k,t_k\to 0:\ a_\kc(x_0;\xx_{01}(t_k),\xx_{02}(s_k)) \to c} 
\end{equation}
of accumulation points of $a_{\kc} (x_{0};
\xx_{01}(t),\xx_{02}(s))$ as $t,s \to 0$ is independent of $\kc.$

\begin{proposition}\label{ThKsSZ} 
  Let  $(\YX,\sfd_{\YX})$ be a metric space and $\xx_{01}:[0,\tau_{1}] \to
  \YX$, $\xx_{02}:[0,\tau_{2}] \to \YX$ be two unit-speed geodesics,
  issuing from $x_{0}\in \YX$. Then, for $\kc\in\mathbb{R}$ we have
\begin{equation}\label{ZoRi}
a_{0}(x_{0};\xx_{01}(t),\xx_{02}(s))-a_{\kc}(x_{0};\xx_{01}(t),\xx_{02}(s)) \to 0 \quad \text{for }t,s\to 0. 
\end{equation}
\end{proposition}
We will provide an analytical proof here. For the reader with a more geometrically oriented 
mind we suggest the proof in \cite[Page 52, Lemma 6.3.1]{S.AlexanderV.Kapovitch}, which became known to us after the completion
of the article.

\begin{proof}  We give here details for the case $\kc=1$. The 
  other cases work exactly the same. 
For $(t,s)\in (0,\tau] \x (0,\tau]$ with $\tau<\min\{1/2,
\tau_1,\tau_2\}$ we set
$c_{t,s}:=\sfd(\xx_{01}(t),\xx_{02}(s))$. Using $t =
\sfd(x_0,\xx_{01}(t))$ and $s= \sfd(x_0,\xx_{02}(s))$, the triangle
inequality  gives $|t{-}s| \leq c_{s,t}\leq t{+}s$. This is equivalent
to 
\[
\exists\, \theta \in [-1,1]:  \quad c^2_{t,s} = s^2+t^2 -2st\theta,
\] 
where $\theta$ equals $a_{0}(x_{0};\xx_{01}(t),\xx_{02}(s))$. Now,
defining the function
\[
G(s,t;\theta) = \theta- \frac{\cos\sqrt{s^2{+}t^2{-}2st\theta} -
  \cos(s)\cos(t)}{\sin(s) \;\sin(t)} , 
\]
we see that \eqref{ZoRi} is established if we show $\|
G(s,t;\cdot)\|_\infty\to 0 $ for $s,t\to 0$, where $\|\cdot\|_\infty$
means the supremum over $\theta\in [-1,1]$. To establish the uniform convergence of $G(s,t;\cdot)$ we decompose
$G$ in three parts, namely
\begin{align*}
G(s,t;\theta)&=G_1(s,t;\theta)+G_2(s,t;\theta)+G_3(s,t;\theta) \quad
\text{with} 
\\
G_1(s,t;\theta)&:= \theta - \frac{\sin(st\theta)}{\sin(s)\:\sin(t)} =
\Big( 1-
\frac{F(s)F(t)}{F(st\theta)}\Big)\frac{\sin(st\theta)}{\sin(s)\sin(t)},\\
 G_2(s,t;\theta)&:=\frac{\sin(st\theta)-\cos\sqrt{s^2{+}t^2{-}2st\theta} 
 +\cos\sqrt{s^2{+}t^2}}{\sin(s)\:\sin(t)},\\
 G_3(s,t;\theta)&:=\frac{\cos(s)\cos(t)-\cos\sqrt{s^2{+}t^2}}{\sin(s)\:\sin(t)},
\end{align*}
where the function $F(r)=\frac1r \sin r$ can be analytically
extended by $F(0)=1$. 

Using $s,t\leq 1/2$ and $|\theta|\leq 1$ we easily obtain 
\[
|G_1(s,t;\theta)|\leq 6\big( s+ t\big) \frac{s t}{(s/2)\:(t/2)}\leq 24
(s{+}t) \ \to\  0 \text{ for }s,t\to 0.
\]
For $G_3$ we use that $K(r)=1-\cos(\sqrt r) $ is an analytic function
with $K(0)= 0$. Thus, with $\sigma=s^2$ and $\tau=t^2$ we have 
\begin{align*}
&\big|\cos(s)\cos(t)-\cos\sqrt{s^2{+}t^2}\big|=
\big|(1{-}K(\sigma))(1{-}K(\tau)) - 1+K(\sigma{+}\tau) \big|  \\
&\leq \big| K(\sigma) +K(\tau) -K(\sigma{+}\tau)- K(0) \big| +
K(\sigma)K(\tau) \\
&\leq \big| \int_0^\sigma \int_0^\tau K''(\hat\sigma{+}\hat\tau)
\:\mathrm d \hat\tau \,\mathrm d \hat\sigma \big| + C^2_1 \sigma\tau
\leq \big( C_2{+}C^2_1) \sigma \tau\  =\  \big( C_2{+}C^2_1)s^2 t^2,
\end{align*}
where $C_1$ and $C_2$ are bounds for $|K'(r)| $ and $|K''(r)|$ with
$r\in [0,1/2]$, repesctively. Inserting this into the definition of
$G_3$ we find
\[
|G_3(s,t;\theta)|\leq \frac{\big( C_2{+}C^2_1)s^2 t^2}{(s/2)\:(t/2)} \
\leq 4 \big( C_2{+}C^2_1)s t \ \to \ 0 \text{ for }s,t\to 0.
\]

The estimate for $G_2$ we use $K$ again and
rewrite the nominator as
\[
\sin(st\theta)+K(s^2{+}t^2)-K(s2{+}t^2{-}2st\theta)=
\sin(st\theta)-st\theta +
\int_0^1\!\!\!\big(1{-}2K'(s^2{+}t^2{-}2st\theta\eta)\big)\rmd \eta
\,st\theta.
\]
Using $1= 2K'(0)$ we can estimate the integral by the bound $C_2$ on $K''$
and obtain
\[
|G_2(s,t;\theta)| \leq \frac{|st\theta|^3/6 + 2C_2(t{+}s)^2
  st|\theta|}{(s/2)\:(t/2)} \leq \frac46 s^2t^2 + 8C_2 (t{+}s)^2 \ \to
\ 0 \text{ for }s,t\to 0.
\]
With this, the desired uniform convergence $G(s,t;\cdot) \to 0$ is
established, and the proof is complete. 
\end{proof}

We are now going to introduce the notion of local angles. 

\begin{definition}[Local Angles]
  Let $\xx_{01}$and $\xx_{02}$ be two geodesics in $\YX$ emanating
  from the same initial point $x_{0}:=\xx_{01}(0)=\xx_{02}(0).$ The
  \emph{upper angle}
  $\varangle_{\rmu\rmp}(\xx_{01},\xx_{02})\in[0,\pi]$ and the
  \emph{lower angle}
  $\varangle_{\rml\rmo}(\xx_{01},\xx_{02})\in[0,\pi],$ between
  $\xx_{01}$ and $\xx_{02}$ are defined by
\begin{subequations}
 \label{eq:LocAng}
\begin{align}
	\varangle_{\rmu\rmp}(\xx_{01},\xx_{02})&:=\limsup_{s, t \downarrow 0}
        \widetilde{\varangle}_{0}(x_{0},\xx_{01}(s),\xx_{02}(t)) 
  { = \arccos\big(\inf \ACCPOINTS{\xx_{01},\xx_{02}}\big), }
\\
\varangle_{\rml\rmo}(\xx_{01},\xx_{02})&:=\liminf_{s, t \downarrow 0}
  \widetilde{\varangle}_{0}(x_{0};\xx_{01}(s),\xx_{02}(t)) 
{ = \arccos\big(\sup \ACCPOINTS{\xx_{01},\xx_{02}}\big) } . 
\end{align}
\end{subequations}   
When $\varangle_{\rmu\rmp}(\xx_{01},\xx_{02})=
\varangle_{\rml\rmo}(\xx_{01},\xx_{02}),$ we say that the
\emph{(local) angle} exists in
the strict sense  and write $\varangle(\xx_{01},\xx_{02})$. 
\end{definition}

In the previous definition, we could use any model space $\mathbb{M}^{2}(\kappa),$ since as we have seen in Proposition \ref{ThKsSZ} the set of
limit points of $a_{\kc}(x_{0};\xx_{01}(t),\xx_{02}(s))$ as $t,s \to
0,$ is independent of $\kc.$ It is also trivial that the above limits
are invariant under re-parametrization, and that is why we are mostly
going to use constant-speed geodesics for joining points.

\subsection{Curvature and Local Angle Condition}
\label{su:CurvLAC}

Curvature is one of the most fundamental geometric properties in geodesic metric spaces, and it has applications in gradient
flows (see \cite{Ohta2009a, S.AlexanderV.Kapovitch,  Savare2007}).  There are many equivalent characterizations, 
see \cite{S.AlexanderV.Kapovitch, Burago2001, Berestovskii1983} for definitions and exposition. We are going to
provide the one that is closer to our results, which was introduced in \cite{Stur99MSLB}.

\begin{definition}\label{WhoAmI} 
We will say that a geodesic metric space $(\YX,\sfd_{\YX})$ has \emph{curvature not less than $\kc$} at a point $x$, 
if there is a neighborhood $U$ of $x$, such that 
\begin{eqnarray}\label{YmnosTouPerel}
    \sum_{i,j=1}^{m}b_{i}b_{j}a_{\kc}(x_{0};x_{i},x_{j})\geq 0
\end{eqnarray} 
for every  $m\in\mathbb{N}$, $x_{0},x_{1},\dots,x_{m}$ in $U,$ and $b_{1},\dots,b_{m}\in[0,\infty).$  We say
that $(\YX,\sfd_{\YX})$ has \emph{curvature not less than $\kc$ ``in the large''}, if we can take $U=\YX.$
We shortly write $\underline{\curv}_{\YX}(x)>\kc,$ if the space $(\YX,\sfd_{\YX})$ has curvature not less than $\kc,$ at $x.$ We finally write
$\underline{\curv}_{\YX}\geq\kc$  if the space $(\YX,\sfd_{\YX})$ has curvature not less than $\kc,$ in the large.
\end{definition}

We would like to note at this point that  $\underline{\curv}_{\YX}(x)>\kc$ for every $x\in\YX,$ does not a-priori imply
that $\underline{\curv}_{\YX}>\kc,$ since the second will require for \eqref{YmnosTouPerel} to hold for arbitrarily big triangles. However 
we recall the following beautiful theorem (see  \cite[Th. 10.3.1]{Burago2001}), which we will use at a later point.

\begin{theorem}[Toponogov's Theorem]\label{Toponogov}
 If a complete geodesic metric space $(\YX,\sfd_{\YX})$ has curvature not less than $\kappa$ at every point, 
 then it has curvature not less than $\kappa$ in the large, i.e.
 \[\left(\forall x\in\YX:\,\, \underline{\curv}_{\YX}(x)>\kc\right) \Leftrightarrow \underline{\curv}_{\YX}>\kc\]
\end{theorem}
 Concerning the curvatures of a cone $\calC$ and its spherical
space $\YX$, the following result is well-known. 

\begin{theorem}\cite[Thm.\,4.7.1]{Burago2001}\label{curvofcones}
Let $(\calC,\sfd_\calC)$  be a cone over a geodesic metric space  $(\YX,\sfd_{\YX})$ , and $\boldsymbol{0}$ its
apex. Then, the following holds:

\begin{itemize}
\item[{\rm{(a)}}]$\left(\forall z\in \calC\setminus \{\boldsymbol{0}\}: \underline{\curv}_{\calC}(z)>0\right),$  
if and only if  $\underline{\curv}_{\YX}\geq 1.$ 
 
\item[{\rm{(b)}}] $\underline{\curv}_{\calC}\geq 0,$ if and only
  if $\underline{\curv}_{\YX}\geq 1$ and no triangle in $\YX$ has perimeter greater than $2\pi$
  (i.e. for any pairwise different $x_{1},x_{2},x_{3},$ we have
  $\sfd_{\YX}(x_{1},x_{2})+\sfd_{\YX}(x_{2},x_{3})+\sfd_{\YX}(x_{3},x_{1})\leq
  2\pi$).\end{itemize}
\end{theorem}

The notion of curvature is not very stable when we take the cone
$(\calC,\sfd_{\calC})$ over a space $(\YX,\sfd_{\YX})$ or when
constructing the Wasserstein space $(\cP_2(\YX),\sfW)$ over $(\YX,\sfd_{\YX}).$ For the first
statement, we recall the previous theorem and see that we need $\underline{\curv}_{\YX}\geq 1$ to achieve $\underline{\curv}_{\mfC}\geq 0,$
while any other ``lower curvature bound'' $\kc<1$ for $(\YX,\sfd_{\YX})$ is not enough to guarantee any ``lower curvature bound'' for $(\calC,\sfd_{\calC})$.
For the second statement, we refer to \cite{AmGiSa05GFMS}, where it is shown that we
\textbf{need} $\underline{\curv}_{\YX}\geq 0$ to deduce $\underline{\curv}_{\mathcal{P}_{2}(\YX)}\geq 0$.  

Hence, we are going to investigate a significantlly weaker but much more stable notion than lower curvature, which along with
some other geometric
properties, is enough enough to prove existence of gradient flows, cf.\ 
\cite[Part 1, Ch.\,6]{OlPaVi14OTTA}. 
The property  that we are going to examine is the 
\emph{Local Angle Condition} (LAC). As it will be shown, LAC is a property
that is transferable from $(\YX,\sfd_{\YX})$ to
$(\calC\setminus\{\boldsymbol{0}\},\sfd_{\calC}),$ but is also stable when we
move to the Wasserstein and the Hellinger-Kantorovich space
$(\M(\YX),\HK_{\dd})$ over $(\YX,\sfd_{\YX}).$

\begin{definition}
\label{def:m-LAC}
For $m\in\mathbb{N},$ a geodesic metric space  $(\YX,\sfd_{\YX})$ satisfies \emph{$m$-LAC at a point $x_{0},$} if for every
choice of $m$ non-trivial geodesics $\xx_{0i}$ starting at $x_{0}$ and positive real numbers $b_{i},\, i\in\{1,\dots m\},$ we have
\begin{equation}\label{eq:m-LAC}
	\sum_{i,j=1}^{m}b_{i}b_{j}\cos(\varangle_{\rmu\rmp}(\xx_{0i},\xx_{0j}))\geq 0.
\end{equation}
If  $ (\YX,\sfd_{\YX})$ satisfies \emph{$m$-LAC} at all points, we say that the space satisfies $\emph{$m$-LAC}.$
\end{definition}

We note that $(\YX,\sfd_{\YX})$ satisfying $m$-LAC at a point $x_{*}$ is a fundamentally weaker notion than having 
$\underline{\curv}_{\YX}(x_{*})\geq \kc$ for some $\kc\in \mathbb{R}.$ For $m$-LAC, one has to look only 
at infinitesimal triangles with common vertex $x_{*},$ while for curvature bounds, one has 
to look at all triangles in a neighborhood of $x_{*}.$ Furthermore, since the triangles used in the definition of $m$-LAC are arbitrarily
small, by application of Proposition \ref{ThKsSZ} the dependence on any specific $\kc$ disappears. Using loose terminology, 
one can say that curvature is a second order, while $m$-LAC is a first order property. Furthermore one could say that $m$-LAC captures, in a rough sense, the  infinitesimally Euclidean nature around $x_{*}$ of the ``geodesically convex hulls'' generated by $m$ geodesics.
By using geodesics in $\eqref{YmnosTouPerel},$ taking limits, and recalling the fact that  angles exist in spaces with 
curvature not less than a real number (see \cite{Burago2001}), one can easily retrieve the following theorem.

\begin{theorem}
Let $(\YX,\sfd_{\YX})$ a geodesic metric space and $x$ a point in it. If $\underline{\curv}_{\YX}(x)\geq\kc$ for 
some $\kc\in\mathbb{R},$ then $(\YX,\sfd_{\YX})$ satisfies 
$m$-LAC at every $x_{0}$ in a neighborhood $U$ of $x$ and for all $m\in \N.$
\end{theorem}

For $m=1$ and $2$ the condition is trivially satisfied. For
$m=3$, which is the case needed for construction solutions for
gradient flows,  we have the following equivalent, more geometric
characterization.  

\begin{theorem}[\cite{Savare2007,OlPaVi14OTTA}]\label{altdef}
A geodesic metric space $(\YX,\sfd_{\YX})$ satisfies  $3$-LAC at $x_{0}$, if and only if for all triples of geodesics $\xx_{01},\xx_{02}, \xx_{03}$
emanating from $x_{0},$ we have
\[\varangle_{\rmu\rmp}(\xx_{01},\xx_{02}) + \varangle_{\rmu\rmp}(\xx_{02},\xx_{03}) + \varangle_{\rmu\rmp}(\xx_{03},\xx_{01}) \leq 2\pi. \]
\end{theorem}

We now provide one of our major abstract results. We will
show that  $m$-LAC is  stable on lifting to  cones and  projecting to
the spherical space inside a cone.

\begin{theorem}\label{Cones and LAC pointwise}
  Let $(\calC, \sfd_{\calC})$ be the cone over a geodesic metric space
  $(\YX,\sfd_{\YX}).$ Then we have
\begin{itemize}
\item[{\rm{(a)}}] If $(\calC, \sfd_{\calC})$ satisfies $m$-LAC at $z_{0}=[x_{0},r_{0}]$ for some 
$x_{0}\in\YX$ and $r_{0}>0,$ then $(\YX,\sfd_{\YX})$  satisfies  $m$-LAC at $x_{0}.$  
\item[{\rm{(b)}}] Conversely if  $(\YX,\sfd_{\YX})$ satisfies $m$-LAC at $x_{0}$,
  then   $z_{0}=(x_{0},r_{0})\in (\calC,\sfd_{\calC})$ also satisfies it for every $r_{0}>0.$
\item[{\rm{(c)}}] $(\calC, \sfd_{\calC})$ satisfies $3$-LAC at the apex \,$\boldsymbol{0}$ if and
  only if $(\YX,\sfd_{\YX})$ has perimeter less
  than $2\pi$.
\item[{\rm{(d)}}] If $(\YX,\sfd_{\YX})$ has diameter less or
  equal to $\pi/2$, then  $(\calC, \sfd_{\calC})$ satisfies $m$-LAC at $\boldsymbol{0}$ for all $m\in\mathbb{N}$. 
\end{itemize}
\end{theorem}

Before we prove this theorem, we  provide some
auxiliary lemmas. For notational economy, we again set
$\phi_{ij}=\sfd_{\YX}(x_{i},x_{j})$ and
$D_{ij}=\sfd_{\calC}(z_{i},z_{j}).$ We will use planar comparison
angles (i.e. $\kc=0$) for the cone $\calC,$ and spherical comparison
angles ($\kc=1$) for the underlying space $\YX$ (recall Definition
\ref{comp}).

\begin{lemma}
\label{le:PreForm}
Let $z_{0}=[x_{0},r_{0}] \in \calC\setminus\{\boldsymbol{0}\}$,
$z_{1}=[x_{1},r_{1}]$, $z_{2}=[x_{2},r_{2}]\in\calC$, and
 $0<\sfd_{\YX}(x_{0},x_{i})<\pi,$ $i\in\{1,2\}.$   Let $\xx_{0i}\in \mathrm{Geod}(x_0,x_i),$ for $i=1,2.$ Let also
$\zz_{0i}=[\overline{\xx}_{0i},\rr_{0i}]$ be the corresponding
constant-speed geodesics in $\calC.$ Then,
$\mathcal{A}_{0,\calC}(t,s):=a_{0}(z_{0};\zz_{01}(t),\zz_{02}(s))$ and
$\overline{\mathcal{A}}_{1,\YX}(t,s):=a_{1}(x_{0};
\overline{\xx}_{01}(t),\overline{\xx}_{02}(s))$  are connected by
the relation 
\begin{equation}\label{preformula}
\begin{split}
  &\mathcal{A}_{0,\calC}(t,s)=\frac{(r_{1}\cos(\phi_{01})-r_{0})
    (r_{2}\cos(\phi_{02})-r_{0})}{D_{01}D_{02}}
  \\&+\frac{\sin(\sfd_{\YX}(x_{0},\overline{\xx}_{01}(t)))
    \sin(\sfd_{\YX}(x_{0},\overline{\xx}_{02}(s)))}
  {\sfd_{\YX}(x_{0},\overline{\xx}_{01}(t))\sfd_{\YX}(x_{0},
    \overline{\xx}_{02}(s))}\frac{\rr_{01}(t) \rr_{02}(s)\ztzt_{01}(t)
    \ztzt_{02}(s)\phi_{01}\phi_{02}}
  {tsD_{01}D_{02}}\overline{\mathcal{A}}_{1,\YX}(t,s).
\end{split}
\end{equation}
\end{lemma}
\begin{proof}
By the reparametrization rule\eqref{quadraticmass} we have
$\overline{\xx}_{0i}(t)=\xx_{0i}(\ztzt_{0i}(t)),$ where  
\begin{equation}
\label{eq:repar}
\ztzt_{0i}(t)=\frac1{\phi_{0i}
}\arccos\left(\frac{ (1{-}t) r_{0}+tr_{i}\cos(\phi_{0i})}{\rr_{0i}(t)}\right),
\end{equation}
from which we obtain
\begin{equation}\label{aha}
\begin{split}
\rr_{0i}(t)\cos(\sfd_{\YX}(x_{0},\overline{\xx}_{0i}(t)))
&=\rr_{0i}(t)\cos(\ztzt_{0i}(t)\phi_{0i})= (1{-}t) r_{0}+tr_{i}\cos(\phi_{0i})\\
&=r_{0} + t(r_{i}\cos(\phi_{0i})-r_{0}) . 
\end{split}
\end{equation}
On the one hand the definition of the comparison angles $a_{1}$ on
$(\YX,\sfd_\YX)$ yields
\begin{equation}
\begin{split}
 \cos(\sfd_{\YX}(\overline{\xx}_{01}(t),\overline{\xx}_{02}(s)))
 &=\cos(\sfd_{\YX}(x_{0},\overline{\xx}_{01}(t))) 
 \cos(\sfd_{\YX}(x_{0},\overline{\xx}_{02}(s))) 
\\ &
 \quad + \overline{\mathcal{A}}_{1,\YX}(t,s) \sin(\sfd_{\YX}(x_{0}, 
 \overline{\xx}_{01}(t)))\sin(\sfd_{\YX}(x_{0},\overline{\xx}_{02}(s))) .
\end{split}
\end{equation}
On the other hand, the definition of $a_0$ on $(\calC,\sfd_\calC)$ and
$\sfd_\calC(z_0,\zz_{0j}(t)) = t D_{0j}$ lead to  	
\begin{equation}
\label{prepreformula}
\mathcal{A}_{0,\calC}(t,s)= \frac{\sfd_{\calC}^{2}(z_{0},\zz_{01}(t)) 
+\sfd_{\calC}^{2}(z_{0},\zz_{02}(s))-\sfd_{\calC}^{2}(\zz_{01}(t),\zz_{02}(s))}
{2tsD_{01}D_{02}}. 
\end{equation}
The nominator of the right-hand side is equal to
\begin{equation}
\begin{aligned}
&r_{0}^{2}+\rr_{01}(t)^{2}-2r_{0}\rr_{01}(t) 
        \cos(\sfd_{\YX}(x_{0},\overline{\xx}_{01}(t)))
\\
&\quad + r_{0}^{2}+\rr_{02}(s)^{2}-2r_{0}\rr_{02}(s) 
 \cos(\sfd_{\YX}(x_{0},\overline{\xx}_{02}(s)))\\&\quad
-\rr_{01}(t)^{2} - \rr_{02}(s)^{2} + 2\rr_{01}(t)\rr_{02}(s)
\cos(\sfd_{\YX}(\overline{\xx}_{01}(t),\overline{\xx}_{02}(s)))\\
&=\underline{2r_{0}^{2}-2r_{0}\rr_{01}(t)
  \cos(\sfd_{\YX}(x_{0},\overline{\xx}_{01}(t))) 
  -2r_{0}\rr_{02}(s)\cos(\sfd_{\YX}(x_{0},\overline{\xx}_{02}(s)))}\\
&\quad + \underline{ 2\rr_{01}(t)\rr_{02}(s)
  \cos(\sfd_{\YX}(x_{0},\overline{\xx}_{01}(t)))
  \cos(\sfd_{\YX}(x_{0},\overline{\xx}_{02}(s)))}\\
&\quad  +2\rr_{01}(t)\rr_{02}(s) \overline{\mathcal{A}}_{1,\YX}(t,s)
  \sin(\sfd_{\YX}(x_{0},\overline{\xx}_{01}(t))) 
 \sin(\sfd_{\YX}(x_{0},\overline{\xx}_{02}(s))). 
\end{aligned}
\end{equation}
Using \eqref{aha} on the underlined terms on the last sum, we obtain
\begin{equation*}
\begin{split}
&2r_{0}^{2}- 2r_{0}\left(r_{0} + t(r_{1}\cos(\phi_{01})-r_{0})\right)
- 2r_{0}\left(r_{0}+ s(r_{2}\cos(\phi_{02})-r_{0})\right)\\
&\quad+ 2\left(r_{0} + t(r_{1}\cos(\phi_{01})-r_{0})\right)\left(r_{0}+s(r_{2}\cos(\phi_{02})-r_{0})\right)\\&=
2ts\,(r_{1}\cos(\phi_{01})-r_{0})\:(r_{2}\cos(\phi_{02})-r_{0}).
\end{split}
\end{equation*}
So \eqref{prepreformula} takes the form
\[
\begin{split}
  \mathcal{A}_{0,\calC}(t,s)&=\frac{(r_{1}\cos(\phi_{01})-r_{0})(r_{2}\cos(\phi_{02})-r_{0})}{D_{01}D_{02}}
  \\&\quad + \frac{
    \rr_{01}(t)\rr_{02}(s) \sin(\sfd_{\YX}(x_{0},\overline{\xx}_{01}(t))) 
 \sin(\sfd_{\YX}(x_{0},\overline{\xx}_{02}(s)))}{tsD_{01}D_{02}}
  \overline{\mathcal{A}}_{1,\YX}(t,s)
  \\
  &= \frac{(r_{1}\cos(\phi_{01})-r_{0})(r_{2}\cos(\phi_{02})-r_{0})}{D_{01}D_{02}}
  \\&\quad+\frac{\sin(\sfd_{\YX}(x_{0},\overline{\xx}_{01}(t)))
 \sin(\sfd_{\YX}(x_{0},\overline{\xx}_{02}(s)))} 
  {\sfd_{\YX}(x_{0},\overline{\xx}_{01}(t))\sfd_{\YX}(x_{0},\overline{\xx}_{02}(s))}
  \frac{\rr_{01}(t)\rr_{02}(s)\ztzt_{01}(t)\ztzt_{02}(s)
    \phi_{01}\phi_{02}} {tsD_{01}D_{02}}\overline{\mathcal{A}}_{1,\YX}(t,s),
\end{split}
\]
 which is the desired result \eqref{preformula}. 
\end{proof}

Since local angles do not depend on the choice of model space
$\mathbb{M}^{2}(\kappa),$ the previous lemma provides a direct
connection between the local angles of geodesics in
$(\calC,\sfd_{\calC})$ and the the local angles of the corresponding
geodesics in $(\YX,\sfd_{\YX})$.

\begin{proposition}
  \label{prop:angle.Angle} Let $z_{0}=[x_{0},r_{0}] \in
  \calC\setminus\{\boldsymbol{0}\},\,z_{1}=[x_{1},r_{1}],
  z_{2}=[x_{2},r_{2}]\in\calC\setminus\{z_{0}\}$ and  $0<\sfd_{\YX}(x_{0},x_{i})<\pi$ for $i\in\{1,2\}.$  Let
  $\xx_{0i}\in \mathrm{Geod}(x_0,x_i)$ for $i=1,2.$ Let also
  $\zz_{0i}=[\overline{\xx}_{0i},\rr_{0i}]$ the corresponding
  geodesics in $\calC$. Then, $\ACCPOINTS{\xx_{01},\xx_{02}}$ and $
  \ACCPOINTS{\zz_{01},\zz_{02}}$ (see \eqref{eq:AccPoints} for definition) 
  satisfy the relation  
\begin{equation}
 \label{eq:AP.AP}
    \ACCPOINTS{\zz_{01},\zz_{02}} =
    \frac{(r_0 {-} r_{1}\cos \phi_{01}) (r_0 {-} r_{2}\cos \phi_{02} )}
          {\sfd_\calC(z_0,z_1) \sfd_\calC(z_0,z_2)}  
      +\frac{r_{1}r_{2}\sin(\phi_{01} 
      )\sin(\phi_{02})} {\sfd_\calC(z_0,z_1) \sfd_\calC(z_0,z_2)}
  \ACCPOINTS{\xx_{01},\xx_{02}} ,
\end{equation}
where $\phi_{0j}=\sfd_\YX(x_0,x_j)$ and 
where the operations between set and real numbers are per
element. More specifically we have  
\begin{equation}\label{Arka}
\cos\!\left(\!\varangle_{\rmu\rmp}(\zz_{01},\zz_{02})\right)\!\! = \!\!\frac{(r_0 {-} r_{1}\cos
  \phi_{01} )(r_0 {-} r_{2}\cos \phi_{02} ){+}
        r_{1}r_{2}\sin(\phi_{01} 
      )\sin(\phi_{02})\cos\!\left(\varangle_{\rmu\rmp}(\xx_{01},\xx_{02})\right)} {\sfd_\calC(z_0,z_1) \sfd_\calC(z_0,z_2)} , 
\end{equation}
{\centering and}
\begin{equation}\label{Kal}
\cos\left(\varangle_{\rmu\rmp}(\xx_{01},\xx_{02})\right)= 
 \frac{\sfd_\calC(z_0,z_1)
  \sfd_\calC(z_0,z_2)\cos\left(\varangle_{\rmu\rmp}(\zz_{01},\zz_{02})\right)} 
   {r_{1}r_{2}\sin(\phi_{01})\sin(\phi_{02})} - \frac{
  (r_0{-}r_{1}\cos\phi_{01}) (r_0 {-}r_{2}\cos\phi_{02})}
   {r_{1}r_{2}\sin(\phi_{01})\sin(\phi_{02})}.
\end{equation}
Furthermore, when $x_{0}=x_{1}$ or $x_{0}=x_{2},$ formula \eqref{Arka} holds trivially with the right-hand side of the sum being equal to zero.
\end{proposition}
\begin{proof}
By reparametrization \eqref{eq:repar} we have 
$\overline{\mathcal{A}}_{0,\YX}(t,s) = \mathcal{A}_{0,\YX}
(\ztzt_{01}(t), \ztzt_{02}(s))$, 
therefore $\overline{\mathcal{A}}_{0,\YX}(t,s)$ and $ \mathcal{A}_{0,\YX}(t,s)$
have the same accumulation points. Furthermore,
Proposition \ref{ThKsSZ} guarantees that
$\overline{\mathcal{A}}_{0,\YX}(t,s)$ and
$\overline{\mathcal{A}}_{1,\YX}(t,s)=a_{1}(x_{0};
\overline{\xx}_{01}(t),\overline{\xx}_{02}(s))$ have the same accumulation
points. 

Let $\ell$ an accumulation point for
$\overline{\mathcal{A}}_{1,\YX}(t,s)$ and $t_{n},s_{n}$ sequences
that achieve that the limit $\ell$. By using formula \eqref{preformula} in
Lemma \ref{le:PreForm} and $\lim_{\tau \to 0} \frac{\sin(\tau)}{\tau}
= 1$, we have
\begin{equation*}
\begin{split}
\lim_{n \to \infty}\mathcal{A}_{0,\calC}(t_{n},s_{n})&=\frac{(r_{1}\cos(\phi_{01})-r_{0})(r_{2}\cos(\phi_{02})-r_{0})}{D_{01}D_{02}}
\\&\quad +\lim_{n \to \infty}\frac{r^{01}(t_{n})r^{02}(s_{n})\ztzt_{01}(t_{n})\ztzt_{02}(s_{n})\phi_{01}\phi_{02}}
{t_{n}s_{n}D_{01}D_{02}}\lim_{n \to \infty}\overline{\mathcal{A}}_{1,\YX}(t_{n},s_{n}).
\end{split}
\end{equation*}
Using formula \eqref{quadraticmass}, we have $\lim_{\epsilon \to 0}
\frac{\ztzt_{0i}(\epsilon)}{\epsilon} = \frac{r_{i}\sin(\phi_{0i})}
{r_{0}\phi_{0i}}$ and $\lim_{\epsilon \to 0}\rr_{0i}(\epsilon)=r_{0},$
and find 
\begin{equation}
\begin{split}
\lim_{n \to
  \infty}\mathcal{A}_{0,\calC}(t_{n},s_{n})&=\frac{(r_{1}\cos(\phi_{01}){-}
  r_{0})(r_{2}\cos(\phi_{02}){-}r_{0}) + r_{1}r_{2}\sin(\phi_{01})\sin(\phi_{02})\ell}
{D_{01}D_{02}}.
\end{split}
\end{equation}
Doing the same for all accumulation points of
$\mathcal{A}_{0,\calC}(t,s),$ we recover the desired formula
\eqref{eq:AP.AP}. 

The formulas for the upper local angle follow simply the taking the
infimum of the sets of accumulation points, see \eqref{eq:LocAng}. 
\end{proof}

 We are now ready to establish the main result giving the connection
between the local angle condition in $(\calC,\sfd_\calC)$ and
$(\YX,\sfd_\YX)$, respectively.\bigskip 

\noindent 
\begin{proof}[Theorem \ref{Cones and LAC}]

Since the local angle between geodesics depends only on their behavior in neighborhoods around point $x_{0}$ or $z_{0}$ respectively,  
for this proof we will assume, without any loss of generality, that $\sfd_{\YX}(x_{0},x_{i})<\pi.$

\underline{Part (a):}  Let now assume that
$z_{0}=[x_{0},r_{0}]\in(\calC\setminus\{\boldsymbol{0}\})$ satisfies $m$-LAC for some $m\in\mathbb{N}$. For $x_{0}\in \YX,$ consider
$m$ non-trivial constant-speed geodesics $\xx_{0i},$ 
connecting $x_{0}$ to $x_{1},\dots,x_{m}$, respectively. Let
$\xx^{\epsilon}_{0i}(t)=\xx_{0i}(\epsilon t)$ be defined on $[0,1]$ and
consider the geodesics $\zz^{\epsilon}_{0i}$ in $\calC$ that
corresponds to $\xx^{\epsilon}_{0i}$ and
$\rr^{\epsilon}_{0i}(0)=\rr^{\epsilon}_{0i}(1)=r_{0}$. Let finally
$b_{1},\dots,b_{m} \geq 0.$ Using $
\varangle_{\rmu\rmp}(\xx_{0i}^\epsilon,\xx_{0j}^\epsilon)  =
\varangle_{\rmu\rmp}(\xx_{0i},\xx_{0j})$ for all $\epsilon\in (0,1)$,  
applying \eqref{Kal} with $r_i=r_{0}$, and using the simple limits $
\lim_{\tau \to  0}\frac{\sqrt{2-2\cos(\tau)}}{\sin(\tau)}=1$ and $  \lim_{\tau
	\to  0}\frac{1-\cos(\tau)}{\sin(\tau)}=0$,
we have
\begin{equation*}
\begin{split}
&\sum_{i,j=1}^{m}b_{i}b_{j}\cos(\varangle_{\rmu\rmp}(\xx_{0i},\xx_{0j}))
=\lim_{\epsilon \to 0} \sum_{i,j=1}^{m}
b_{i}b_{j}\cos(\varangle_{\rmu\rmp} (\xx^{\epsilon}_{0i},\xx^{\epsilon}_{0j}))\\
&=\lim_{\epsilon \to 0}\sum_{i,j=1}^{m}b_{i}b_{j}\bigg(
\frac{\sqrt{2{-}2\cos\sfd_{\YX} (x_{0},\xx_{0i}(\epsilon))}
	\sqrt{2{-}2\cos\sfd_{\YX}(x_{0},\xx_{0j}(\epsilon))}
	\cos(\varangle_{\rmu\rmp}(\zz^{\epsilon}_{0i},\zz^{\epsilon}_{0j}))}
{\sin(\sfd_{\YX}(x_{0},\xx_{0i}(\epsilon)))
	\sin(\sfd_{\YX}(x_{0},\xx_{0j}(\epsilon)))}\\
&\hspace*{8em} -
\frac{(\cos(\sfd_{\YX}(x_{0},\xx_{0i}(\epsilon))){-}1)
	(\cos(\sfd_{\YX}(x_{0},\xx_{0j}(\epsilon))){-} 1)}
{\sin(\sfd_{\YX}(x_{0},\xx_{0i}(\epsilon)))
	\sin(\sfd_{\YX}(x_{0},\xx_{0j}(\epsilon)))}\bigg) \\&=
\lim_{\epsilon \to 0} \sum_{i,j=1}^{m} b_{i}b_{j} \cos(\varangle_{\rmu\rmp}(
\zz^{\epsilon}_{0i},\zz^{\epsilon}_{0j}))\geq0.
\end{split}
\end{equation*}

   \underline{Part (b):} We start by assuming  that  $x_{0} \in \YX$
  satisfies $m$-LAC for some $m\in\mathbb{N}$.  Let $z_{0}=[x_{0},r_{0}]\in \calC\setminus\{\boldsymbol{0}\}$ and
  $\zz_{01 },\dots,\zz_{0m},$ $m$ non-trivial constant-speed geodesics
  connecting $z_{0}$ to some $z_{1},\dots,z_{m}\in\calC.$ 
 By applying \eqref{Arka}, for all $b_i^\calC \geq 0$ we have
\begin{equation*}
\begin{split}
&\sum_{i,j=1}^{m}b^\calC_{i}b^\calC_{j}\cos(\varangle_{\rmu\rmp}(\zz_{0i},\zz_{0j}))\\
&=\sum_{i,j=1}^{m}b^\calC_{i}b^\calC_{j}\frac{(r_{i}\cos(\phi_{0i}){-}r_{0})
  (r_{j}\cos(\phi_{0j}){-}r_{0}) + r_{i}r_{j}\sin(\phi_{0i}) 
 \sin(\phi_{0j})\cos(\varangle_{\rmu\rmp}(\xx_{0i},\xx_{0j}))}{D_{0i}D_{0j}}\\
&=\left(\sum_{i}^{m}b^\calC_{i}\frac{(r_{i}\cos(\phi_{0i}){-}r_{0})}{D_{0i}}\right)^{2}+\sum_{i,j=1}^{m}b^\calC_{i}b^\calC_{j}\frac{r_{i}r_{j}\sin(\phi_{0i})\sin(\phi_{0j})}
{D_{0i}D_{0j}}\cos(\varangle_{\rmu\rmp}(\xx_{0i},\xx_{0j})). 
\end{split}
\end{equation*} 
Since $x_{0}$ satisfies $m$-LAC, the last term is
non-negative as we may choose $b_j^X:=b_j^\calC r_/D_{0j}\geq 0$ as
testvector. As the first term is a square we conclude that
$z_{0}\in(\calC,\sfd_\calC)$ satisfies $m$-LAC as well.\medskip

For parts (c) and (d) we have to study the geodesics $\zz_{0i}$
starting at the apex $\bm0$. For this  we just notice that for
such geodesics $\zz_{01},\zz_{02}$  ending at some
$z_{1}=[x_{1},r_{1}], z_{2}=[x_{2},r_{2}]$ the angle is equal to
$\sfd_{\YX}(x_{1},x_{2})\wedge\pi.$ Therefore by using Definition
\ref{altdef}, we see that $3$-LAC is satisfied if and only if for
every choice of pairwise different points $x_{1},x_{2},x_{3},$ we have
$\sfd_{\YX}(x_{1},x_{2})\wedge\pi+ \sfd_{\YX}(x_{2},x_{3})\wedge\pi+
\sfd_{\YX}(x_{3},x_{1})\wedge\pi\leq 2\pi,$ which by applying the
triangule inequality is easy to see that it holds if and only if for every
choice of pairwise different points $x_{1},x_{2},x_{3},$ we have
$\sfd_{\YX}(x_{1},x_{2}) + \sfd_{\YX}(x_{2},x_{3}) +
\sfd_{\YX}(x_{3},x_{1})\leq 2\pi.$ This shows part (c).  

When the diameter is less than
$\pi/2,$ then all cosines are positive and therefore \eqref{eq:m-LAC}
is satisfied trivially for all $m\in\mathbb{N}$.  Hence, part (d)
is shown as well. 
\end{proof}	

We can now recover the following immediate result. 

\begin{corollary}\label{Cones and LAC}
Let $(\calC, \sfd_{\calC})$ be the cone over a geodesic metric space
$(\YX,\sfd_{\YX}).$ 
\begin{itemize}
\item[{\rm{(a)}}] If $(\calC, \sfd_{\calC})$ satisfies $m$-LAC, then $(\YX,\sfd_{\YX})$ does too.
\item[{\rm{(b)}}] Conversely if $(\YX,\sfd_{\YX})$ satisfies $m$-LAC,
then $(\calC, \sfd_{\calC})$ satisfies it at every point in $\calC\setminus\{\boldsymbol{0}\}$.
\item[{\rm{(c)}}] $(\calC, \sfd_{\calC})$ satisfies $3$-LAC, if and
only if $(\YX,\sfd_{\YX})$ satisfies $3$-LAC and has perimeter less or equal to $2\pi$.
\item[{\rm{(d)}}] If $(\YX,\sfd_{\YX})$ has diameter less or
equal to $\pi/2$ and satisfies $m$-LAC for some $m$, then $(\calC,\sfd_{\calC})$  satisfies $m$-LAC. 
\end{itemize}
\end{corollary}

\subsection{$K$-semiconcavity}
\label{su:Ksemiconcave}

Another notion that we are going to introduce is the one of
$K$-semiconcavity of a metric space $(\YX,\sfd_{\YX}),$ on a set
$A\subset \YX.$ Before we do that, we are going to give the definition
of $K$-semiconcave functions, and some lemmas that are going to be
used in the proofs.
 
\begin{definition}
A function $f:[0,1] \to \mathbb{R}$ is called \emph{K-semiconcave},
if and only if  the mapping $t \mapsto f{-} K t^{2}$ is concave.  
\end{definition}
Of course, for smooth functions $f$ this means $f''(t) \leq
2K$. The following result deals with semiconcave functions under
composition. For a smooth $K$-semiconve function $f[0,1]\to [a,b]$ and another
smooth function $g:[a,b]\to \R$  the composition satisfies $(g\circ
f)''(t)= g'(f(t))f''_1(t)+ g''(f_1(t))\big(f'_1(t)\big)^2$. In the
lemma below we will use the concave function $g_B:[0,\pi^2] \ni v \mapsto
1-\cos \sqrt v $ in Part B, where the term
involving $g''$ can be estimated by $0$, while in Part A we use
$g_A:{[0,2[} \ni w \mapsto \big(\arccos( 1{-}w)\big)^2$, where $g''$
needs to be estimated on the range of $f$. A major part in the proof
involves the proper treatment of the non-smooth situation where $f''$
is merely a measure.

\begin{lemma}\label{apotoenastoallo}
  Let $(\YX,\sfd_{\YX})$ be a geodesic metric space, $x\in\YX,$ and
  $\mathcal{D}<\frac{\pi}{2}.$ Let also $x_{0},x_{1},x_{2}\in
  B(x,\mathcal{D})\subset\YX,$ and $\xx_{01}\in\Geod(x_{0},x_{1}),$
  with $x_{2}\notin \xx_{01}([0,1]).$ Let finally
  $f_{1},f_{2}:[0,1]\to\mathbb{R}$ given by
  $f_{1}(t)=1-\cos(\sfd_{\YX}(x_{2},\xx_{01}(t)))$ and
  $f_{2}(t)=\sfd^{2}_{\YX}(x_{2},\xx_{01}(t))$ respectively.  We have:
\begin{itemize}
\item[(A)\!] If $ f_{1}$ is $\sfd^{2}_{\YX}(x_{0},x_{1})K$-semiconcave
  for some $K\!>\!0,$ then $f_{2}$ is
  $\!\left(\!1{+}\frac{1+K}{\pi-2\mathcal{D}}\!\right) \!\sfd^{2}_{\YX}
  (\!x_{0},x_{1}\!)$-semiconcave.  
\item[(B)\!] If $ f_{2}$ is $\sfd^{2}_{\YX}(x_{0},x_{1})K$-semiconcave
  for some $K\!>\!0,$ then $f_{1}$ is $\left(1{+}
    K\right)\sfd^{2}_{\YX}(x_{0},x_{1})$-semiconcave. 
\end{itemize}
\end{lemma}
\begin{proof}
  We are going to prove the result by taking second derivatives. Since
  the classical second derivatives are not enough to characterize
  convexity/concavity, we are going to make use of distributional
  derivatives (for definition see \cite{Rudi91FA}). More
  specifically, by \cite[Thm.\,6.8]{EvaGar15MTFP}, we have that a
  continuous function $g:[0,1]\to\mathbb{R}$ is convex if and only if  its
  derivative is of bounded variation (for definition and properties,
  see \cite[Ch.\,5]{EvaGar15MTFP}) and its second derivative is a
  finite positive measure. This means that a finite $g$ is concave
  if and only if  it exists a negative, finite measure $\mu_{g},$ such that for
  every $f\in C^{\infty,+}_{c}((0,1))=\{f:(0,1)\to\mathbb{R}:
  f\hspace{4pt} \text{is positive and smooth}\},$ we have
\begin{equation}
\int_{0}^{1}g(t)f''(t)\,\d t =\int_{0}^{1}f(t)\mu_{g}(\d t)\leq 0.
\end{equation}

 So, we just have to prove that if

\begin{equation}\label{disdis} \int_{0}^{1}\left(
    f_{1}(t)-K_{1}t^{2}\right)f''(t)\d t \leq 0 \ \text{ for all } f\in 
\mathrm C^{\infty,+}_{\mathrm c}((0,1))
\end{equation}
for some $K_{1}>0,$ then 
\begin{equation}\label{disdis2}
  \int_{0}^{1}\left(f_{2}(t)-K_{2}t^{2}\right)f''(t) \d t \leq 0 \  
\text{ for all } f\in C^{\infty,+}_{c}((0,1)),
\end{equation}
for some $K_{2}>0,$ and vice-versa, where the relationship between $K_{1}$ and $K_{2},$ will be specified later.

For abbreviation, we set $v(t)=\sfd_{\YX}(x_{2},\xx_{01}(t)).$ By applying the triangular inequality, we have \begin{equation}\label{vest}
|v(t)-v(s)|= |\sfd_{\YX}(x_{2},\xx_{01}(t))-\sfd_{\YX}(x_{2},\xx_{01}(s))|\leq\sfd_{\YX}(\xx_{01}(t),\xx_{01}(s))\leq |t-s|\sfd_{\YX}(x_{0},x_{1}),
\end{equation}
from which we get that $v(\cdot)$ is Lipschits and $|v'(t)|\leq\sfd_{\YX}(x_{0},x_{1}),$ almost everywhere.
From \eqref{vest} we can deduce that $f_{1},f_{2}$ are also Lipschitz,  therefore the first classical derivative coincides with the first distributional one, and is given by:

\begin{equation}
 f'_{1}(t)=\sin\left(v(t)\right)v'(t), \hspace{16pt} f'_{2}(t)=v(t)v'(t).
\end{equation}
If either of the assumptions are satisfied, which implies concavity we
get that the derivative is of bounded variation. Now, since $v$ is
Lipschitz and bounded away from zero, we get that
$v'=\frac{f'_{1}}{\sin(v)}, \frac{f'_{2}}{v}$ is of bounded variation,
therefore its distributional derivative is a locally finite measure
$\mu_{v}$ (\cite[Th,\,5.1]{EvaGar15MTFP}), and even more, it is
straightforward to see the product rule for the second derivative
holds true, i.e. we have
\begin{equation}\label{AAA}
\int_{0}^{1}\!\! (f_{1}(t)-K_{1}t^{2})f''(t)dt \leq 0 \hspace{2pt}\Leftrightarrow \hspace{2pt} 
\int_{0}^{1}\!\!f(t)\left(\cos(v(t))(v'(t))^{2}dt + \sin(v(t))\mu_{v}(dt)-K_{1}dt\right)\! \leq \!0.\newline 
\end{equation}
Similarly we get
\begin{equation}\label{BBB}
\int_{0}^{1} (f_{2}(t)-K_{2}t^{2})f''(t)dt \leq 0 \hspace{8pt}\Leftrightarrow \hspace{8pt} 
\int_{0}^{1}f(t)\left((v'(t))^{2}dt + v(t)\mu_{v}(dt)-K_{2}dt\right) \leq 0. 
\end{equation}

Part (A). Let assume that \eqref{disdis} is true. By
\eqref{AAA}, we have:
\begin{equation}
\begin{split}
&\int_{0}^{1}f(t)\left(\cos(v(t))(v'(t))^{2}dt + \sin(v(t))\mu_{v}(dt)-K_{1}dt\right)\leq 0 \Rightarrow\\&
\left(\min_{y\in[0,2\mathcal{D}]}\frac{\sin y}{y}\right)\!\! \int_{0}^{1}\!f(t)\!\left(\!\frac{v(t)}{\sin(v(t))}\cos(v(t))(v'(t))^{2}dt + v(t)\mu_{v}(dt)-K_{1}\frac{v(t)}{\sin(v(t))}dt\!\right)\!\leq\! 0 \Rightarrow\\&
\int_{0}^{1}f(t)\left((v'(t))^{2}dt + v(t)\mu_{v}(dt)-K_{2}dt\right)\leq\\&
\int_{0}^{1}f(t)\left(-K_{2} + (v'(t))^{2} -\frac{v(t)}{\sin(v(t))}\cos(v(t))(v'(t))^{2} +K_{1}\frac{v(t)}{\sin(v(t))}\right)dt,
\end{split}
\end{equation}
where we retrieve the last inequality by adding and subtracting. If we
choose $K_{2}$ such that the second term is negative for every
positive test function $f,$ the we are done. We recall that
$|v'(t)|\leq\sfd_{\YX}(x_{0},x_{1}), v(t)\leq 2\mathcal{D}<\pi.$ Now
if $K_{1}= \sfd^{2}_{\YX}(x_{0},x_{1})K,$ we can choose $K_{2}$ to be
equal to
$\left(1+\frac{1+K} {\pi-2\mathcal{D}}\right)\sfd^{2}_{\YX}(x_{0},x_{1})$
and retrieve \eqref{BBB}, independently of the choice of $f\in
C^{\infty,+}_{c}((0,1)).$\newline

Part (B). Let assume that \eqref{disdis2} is true. By
\eqref{BBB}, we have:
\begin{equation}
\begin{split}
&\int_{0}^{1}f(t)\left((v'(t))^{2}dt + v(t)\mu_{v}(dt)-K_{2}dt\right)\leq 0 \hspace{12pt}\Rightarrow \\&
\left(\min_{y\in[0,2\mathcal{D}]}\frac{y}{\sin y}\right)\! \int_{0}^{1}\!f(t)\!\left(\frac{\sin(v(t))}{v(t)}(v'(t))^{2}dt + \sin(v(t))\mu_{v}(dt)-K_{2}\frac{\sin(v(t))}{v(t)}dt\right)\leq 0  \hspace{12pt}\Rightarrow \\&
\int_{0}^{1}f(t)\left(\cos(v(t))(v'(t))^{2}dt + \sin(v(t))\mu_{v}(dt)-K_{1}dt\right)\leq\\&
\int_{0}^{1}f(t)\left(-K_{1} + \cos(v(t))(v'(t))^{2} -\frac{\sin(v(t))}{v(t)}(v'(t))^{2}+K_{2}\frac{\sin(v(t))}{v(t)}\right)dt
\end{split}
\end{equation}
Now, if $K_{2}= \sfd^{2}_{\YX}(x_{0},x_{1})K,$ we can take $K_{1}=\left(1+ K\right)\sfd^{2}_{\YX}(x_{0},x_{1})$ and retrieve \eqref{AAA}, independently of the choice of $f\in C^{\infty,+}_{c}((0,1)).$
\end{proof}

We will use the result of Lemma \ref{apotoenastoallo} in the
following rescaled form that allows to characterize $K$-semiconcavity
by comparing the function with approximating parabolae. 

\begin{corollary} \label{asasa} Let $x_{0},x_{1},x_{2}\in \YX$ 
  and choose  $\xx_{01}\in \Geod(x_{0},x_{1})$.  Let $f_1$
  and $f_2$ be as in Lemma
  \ref{apotoenastoallo}.  For $t_{1},t_{2}\in[0,1]$ we set
 \[
\tilde{x}^{[t_{1},t_{2}]}_{0}=\xx_{01}(t_{1}), \qquad
  \tilde{x}^{[t_{1},t_{2}]}_{1}=\xx_{01}(t_{2}),\qquad \text{and }\quad
  \tilde{f}^{[t_{1},t_{2}]}_{i}(t)=f_{i}\big( t_1 {+} t(t_2 {-}t_1) \big).
\] 
 Then, the following three conditions are equivalent: 
\begin{itemize}
\item[(i)]
   $f_{i}$ is $K\sfd^{2}_{\YX}(x_{0},x_{1})$-semiconcave if,
\item[(ii)]
  for every $t_{1},t_{2}$ the mapping $\tilde{f}^{[t_{1},t_{2}]}_{i}$ is
  $K\sfd^{2}_{\YX}(\tilde{x}^{[t_{1},t_{2}]}_{0},
  \tilde{x}^{[t_{1},t_{2}]}_{1} )-$semiconcave, 
\item[(iii)] for every $t_{1},t_{2},t\in[0,1]$ we have
\begin{equation}\label{equivksemi}
\tilde{f}^{[t_{1},t_{2}]}_{i}(t)+Kt (1{-}t)
\sfd^{2}_{\YX}(\tilde{x}^{[t_{1},t_{2}]}_{0} ,
\tilde{x}^{[t_{1},t_{2}]}_{1}) \geq  (1{-}t)
\tilde{f}^{[t_{1},t_{2}]}_{i}(0) +t\tilde{f}^{[t_{1},t_{2}]}_{i}(1). 
\end{equation}
\end{itemize}
\end{corollary}

 The next elementary lemma will be crucial to estimate the
semiconcavities, where we crucially extract the
factor $t (1{-}t)$ that multiplies $K$ on the right-hand side in
\eqref{equivksemi}.

\begin{lemma}\label{similarsins}
  It exists $C>0,$ such that $|\sin(xt)-t\sin(x)|\leq Ct (1{-}t) x^{3},$
  for all $t\leq1, x\leq\pi.$
\end{lemma}
\begin{proof}
Using the Taylor series
$\sin(y)=\sum_{n=0}^{\infty}\frac{(-1)^{n}(y)^{2n+1}}{(2n+1)!}$ we
obtain   
	\begin{equation*}
	\begin{split}
	\sin(xt)-t\sin(x)&=\sum_{n=0}^{\infty}(-1)^{n}\left(\frac{(tx)^{2n+1}}{(2n+1)!}-t\frac{x^{2n+1}}{(2n+1)!}\right)
	=\sum_{n=1}^{\infty}(-1)^{n}tx^{3} \frac{(t^{2n} {-}1) x^{2n-2}}{(2n+1)!}.
	\end{split}
	\end{equation*}
Using $t \in [0,1]$ and $x\in [0,\pi]$ we find 
	\begin{equation*}
	|\sin(xt)-t\sin(x)|\leq tx^{3}\sum_{n=1}^{\infty}\frac{(1-t^{2n})x^{2n-2}}{(2n+1)!}\leq tx^{3} (1{-}t) \sum_{n=1}^{\infty}\frac{4n\pi^{2n-2}}{(2n+1)!}.
	\end{equation*}
Setting $C:=\sum_{n=1}^{\infty}\frac{4n\pi^{2n-2}}{(2n+1)!} <
\infty$ we arrive at the claimed estimate.
\end{proof}

Next we define notions of local semiconcavity on a space
$(\YX,\sfd_{\YX})$, we give a precise meaning of
\emph{$K$-semiconcavity on a subset of $\YX$}. 

\begin{definition}
  We say that $(\YX,\sfd_{\YX})$ satisfies \emph{$K$-semiconcavity
    along  $\xx_{01}\in \Geod(x_{0},x_{1})$ for some
    $x_{0},x_{1}\in \YX$ with respect to the ``observer'' $x_{2},$} if\/
  $[0,1]\ni t \mapsto  f(t)=\sfd^{2}_{\YX}(x_{2},\xx_{01}(t))$ is
  $K\sfd^{2}_{\YX}(x_{0},x_{1})$-semiconcave. Furthermore, we say that
  $(\YX,\sfd_{\YX})$ satisfies \emph{$K$-semiconcavity on
    $\mathcal{A}\subset \YX$ with respect to observers from
    $\mathcal{B}\subset\YX,$} if it satisfies \emph{$K$-semiconcavity}
  along some geodesic $\xx_{01}\in \Geod(x_{0},x_{1})$ for every
  $x_{0},x_{1}\in A,$ and with respect to every observer $x_{2}\in B.$
  In the case $\mathcal{A}=\mathcal{B},$ we shortly say that $(\YX,\sfd_{\YX})$
  satisfies \emph{$K$-semiconcavity on $\mathcal{A}\subset \YX.$}
  Finally we say that $(\YX,\sfd_{\YX})$ satisfies
  \emph{$K$-semiconcavity}, if $\mathcal{A}=\YX.$
\end{definition}

We would like to remark that in the previous definition, $\xx_{01}(t)$
for $t\in(0,1)$ doesn't have to belong to $A.$ Now, we are going to
prove some results, that are going to be used in the last subsection
to prove $K'$-semiconcavity on ``important'' subsets of
$(\M(X),\HK_{\dd})$  or  $(\mathcal{P}(X),\SHK_{\dd}),$ when
$(X,\sfd_{X})$ satisfies $K$-semiconcavity for some $K>0.$

\begin{proposition}\label{mainKseminconcavitylemma}
  Let $(\YX,\sfd_{\YX})$ be a geodesic metric space, and $(\calC,
  \sfd_{\calC})$ the cone over $(\YX,\sfd_{\YX}).$ For three
  points $z_{0}=[x_{0},r_{0}]$, $z_{1}=[x_{1},r_{1}]$,
  $z_{2}=[x_{2},r_{2}]\in \calC$, consider a geodesic 
  $\xx_{01}\in\Geod(x_{0},x_{1})$, and the corresponding
  geodesic $\zz_{01}$ in $(\calC, \sfd_{\calC})$. Finally let assume that for
  $x\in\YX$ and $\mathcal{D}<\frac{\pi}{2},$ we have $x_{0},x_{1},x_{2}\in
  B\left(x,\mathcal{D}\right).$ 
\begin{itemize}
\item[(A)] If $(\YX,\sfd_{\YX})$ satisfies
  $K$-semiconcavity along $\xx_{01}(t),$ with respect to $x_{2},$ then
  $(\calC, \sfd_{\calC})$ satisfies $K'$-semiconcavity along
  $\zz_{01}(t)$ with respect to $z_{2},$ where $K'$ can be chosen to
  depend continuously only on $K,r_{0},r_{1},r_{2},\mathcal{D}.$
\item[(B)]  If $x_{0}\neq x_{1}$, $r_0=r_1$, and  $(\calC, \sfd_{\calC})$
  satisfies $K$-semiconcavity along $\zz_{01}(t)$ with respect to
  $z_{2}$, then $(\YX,\sfd_{\YX})$ satisfies
  $K'$-semiconcavity along $\xx_{01}(t)$ with respect to $x_{2},$
  where $K'$ can be chosen to depend continuously only on $ K,
  r_{0},r_{2},\mathcal{D}.$
\end{itemize}
\end{proposition}
\begin{proof}
From Theorem \ref{th:Project} we recall that for a geodesic
$\zz_{01}\left(t\right)=[\overline{\xx}_{01}\left(t\right),\rr_{01}\left(t\right)]$
is the corresponding geodesic in $\left(\YX,\sfd_{\YX}\right)$ is given by  
\begin{equation}
t\mapsto\xx_{01}\left(t\right) =\overline{\xx}_{01}\left(\sisi_{01}\left(t\right)\right)
\text{ with } 
\sisi_{01} \left(t\right)=\frac{r_0 \sin\left(t  \phi_{01} \right)
}{r_1\sin\!\left(\left(1{-}t\right) \phi_{01} \right) {+}  
	r_0\sin\!\left(t \phi_{01} \right) } \ \text{ with }
      \phi_{01}:= \sfd_{\YX}(x_0,x_1). 
\end{equation}
Later, we are going to use the fact that when $r_{0}=r_{1},$ we have
\begin{equation}\label{kou}
\begin{split}
&\max_{t,s\in [0,1]}
\frac{d_{\cal{C}}\left(\zz_{01}\left(t\right),\zz_{01}\left(s\right)\right)}
  {d_{\YX}\left(\overline{\xx}_{01}\left(t\right),
  \overline{\xx}_{01}\left(s\right)\right)}
= 
\max_{t,s\in [0,1]}
\frac{d_{\cal{C}}\left(\zz_{01}\left(\beta\left(t\right)\right), 
  \zz_{01}\left(\beta\left(s\right)\right)\right)}
{d_{\YX}\left(\xx_{01}\left(t\right),\xx_{01}\left(s\right)\right)}
= \max_{t,s\in [0,1]}
\frac{\sisi\left(t\right)-\sisi\left(s\right)}{t-s} \\ &\leq
\max_{t}\sisi'\left(t\right)  
\leq \max_{t\in[0,1]}\frac{
  \phi_{01} }{2\sin\left(\frac{ \phi_{01}
    }{2}\right)}\frac{\cos\left(t \phi_{01}
  \right)\cos\left(\frac{1-2t}{2} \phi_{01} \right)+\sin\left(
    \phi_{01} \right)\sin\left(\left(\frac{1-2t}{2} \phi_{01}
    \right)\right)}{\cos^{2}\left(\frac{1-2t}{2} \phi_{01} \right)}\\
&\leq  \frac{ \phi_{01} }{2\sin\left(\frac{ \phi_{01} }{2}\right)}\frac{1}{\cos^{2}\left(\mathcal{D}\right)}.
\end{split}
\end{equation}

Our proof relies on utilizing Corollary \ref{asasa} for
arbitrary $t_{1}$ and  $t_{2}$ and  noticing  that the new
$\tilde{r}_{0}^{[t_{1},t_{2}]}=\rr_{01}\left(t_{1}\right)$ and
$\tilde{r}^{[t_{1},t_{2}]}_{1}=\rr_{01}\left(t_{2}\right)$
are bounded from below by some $r_{\min}>0$  that depend only
on $r_{0},r_{1},\mathcal{D}.$ For notational convenience, we will
drop the dependence
on $t_{1},t_{2}$, but we will use tilde $\tilde{\text{\ }}$ for the new
functions. 

To compare the ``concavity'' magnitude of $\sfd^{2}_\calC$ with to the one of 
$\sfd^{2}_\YX$ along the respective geodesics and observers,  we set
\begin{equation}
  \tilde{A}\left(t\right) = \frac{\left(1{-}t\right)
    \sfd^{2}_{\mathcal{C}}\left(z_{2},\tilde{z}_{0}\right) + t\,
    \sfd^{2}_{\mathcal{C}}\left(z_{2},\tilde{z}_{1}\right) -
    \sfd^{2}_\calC \left(z_{2},\tilde{\zz}
      \left(t,\tilde{z}_{0},\tilde{z}_{1}\right)\right)} 
  {t\left(1{-}t\right)\sfd^{2}_\calC \left(\tilde{z}_{0},\tilde{z}_{1}\right)}.
\end{equation}
Using the formula for the cone distance $\sfd_\calC$  we get
\begin{equation}\begin{split}
\!\!\!\tilde{A}\left(t\right)&\!=\frac{\left(1-t\right)\!\left[\!r_{2}^{2}+\tilde{r}_{0}^{2}-2\tilde{r}_{0}r_{2}\cos\left(\sfd_{\YX}\left(x_{2},\tilde{x}_{0}\right)\right)\right]{+}t\!\left[r_{2}^{2}+\tilde{r}_{1}^{2}-2\tilde{r}_{1}r_{2}\cos\left(\sfd_{\YX}\left(x_{2},\tilde{x}_{1}\right)\right)\right]}{t\left(1{-}t\right)\sfd^{2}_{\mathcal{C}
  }\left(\tilde{z}_{0}, \tilde{z}_{1}\right)} \\&\! 
 \quad -\frac{\left[r_{2}^{2}+\tilde{\rr}_{01}\left(t\right)^{2}-2r_{2}\tilde{\rr}_{01}\left(t\right)
	\cos\left( \sfd_{\YX}\left(x_{2},\bar{\tilde{\xx}}_{01}\left(t\right)\right)\right)\right]}{t\left(1{-}t\right)\sfd^{2}_{\mathcal{C}}\left(\tilde{z}_{0},\tilde{z}_{1}\right)}
\\&\!=1+r_{2}\frac{\tilde{\rr}_{01}\left(t\right)\cos\left(
    \sfd_\YX 
    \left(x_{2},\bar{\tilde{\xx}}_{01}\left(t\right)\right)\right)-\left(1-t\right)\tilde{r}_{0}\cos\left(
     \sfd_\YX \left(x_{2},\tilde{x}_{0}\right)\right)-t\tilde{r}_{1}
	\cos\left( \sfd_\YX  \left(x_{2},\tilde{x}_{1}\right)\right)}{t\left(1{-}t\right)\sfd^{2}_{\mathcal{C}}\left(\tilde{z}_{0},\tilde{z}_{1}\right)}
\end{split}
\end{equation}
 We compose $\tilde{A}\left(t\right)$ with
 $\tilde{\sisi}\left(t\right)$ and find
\begin{align*}
&\frac{\tilde{A}\big(\tilde{\sisi}\left(t\right)\big){-}1}{r_{2}}\\
& =
\frac{\tilde{\rr}_{01}\big(\tilde{\sisi}\left(t\right)\big)
  \cos\left( \sfd_\YX 
    \left(x_{2},\tilde{\xx}_{01}\left(t\right)\right)\right) {-} \big(
  1{-}\tilde{\sisi}\left(t\right)\big) \tilde{r}_{0} \cos\left( 
    \sfd_\YX \left(x_{2} , \tilde{x}_{0}\right)\right)
   {-} x\tilde{\sisi}\left(t\right)\tilde{r}_{1} \cos \left( \sfd_\YX
   \left(x_{2},\tilde{x}_{1}\right) \right)}
{\tilde{\sisi}\left(t\right) \big(1{-}\tilde{\sisi}\left(t\right)\big)
  \sfd^{2}_{\mathcal{C}}\left(\tilde{z}_{0},\tilde{z}_{1}\right)}.
\end{align*}
Recalling that 
$\tilde{\rr}_{01} \big(\tilde{\sisi}_{01}\left(t\right)\big) =
\frac{\tilde{r}_{0}\tilde{r}_{1}\sin\left(\sfd_{\YX} \left(\tilde{x}_0
      , \tilde{x}_1\right)\right)}{\tilde{r}_{1} \sin\left(\left(1-t\right)\sfd_{\YX}\left(\tilde{x}_0,\tilde{x}_1\right)\right)+\tilde{r}_{0}\sin\left(t\sfd_{\YX}\left(\tilde{x}_0,\tilde{x}_1\right)\right)}$ 
we find  
\[
\frac{\tilde{\rr}_{01}\left(\tilde{\sisi}_{01}\left(t\right)\right)}{\tilde{\sisi}_{01}\left(t\right)}=\frac{\tilde{r}_{1}\sin\left(\sfd_{\YX}\left(\tilde{x}_0,\tilde{x}_1\right)\right)}{
  \sin\left(t \sfd_{\YX}\left(\tilde{x}_0,\tilde{x}_1\right)\right)}  \
\text{ and } \ 
\frac{1-\tilde{\sisi}_{01}\left(t\right)}{\tilde{\sisi}_{01}\left(t\right)}=\frac{\tilde{r}_1\sin\!\left(\left(1{-}t\right)\sfd_{\YX}\left(\tilde{x}_0,\tilde{x}_1\right)\right)}{\tilde{r}_0
  \sin\left(t \sfd_{\YX}\left(\tilde{x}_0,\tilde{x}_1\right)\right)}.
\]
 Using the abbreviations  $\tilde{\phi}_{ij}=
\sfd_{\YX}\left(\tilde{x}_i,\tilde{x}_j\right)$, $\tilde{\phi}_{2t}=
\sfd_{\YX}\left(x_2,\tilde{\xx}_{01}\left(t\right)\right)$ for
$i,j\in \{0,1,2\}$ and $t\in[0,1]$ we can write 
$\big(\tilde{A}\big(\!\tilde{\sisi}\left(\!t\!\right)\!\big)-1\big)/r_2$
as a product to estimate the terms individually: \EEE 
\begin{equation}\label{a-1}
\begin{split}
&\frac{\tilde{A}\big(\!\tilde{\sisi}\left(\!t\!\right)\!\big)-1}{r_{2}}=\frac{\tilde{r}_{1}\sin \big(\!\tilde{\phi}_{01}\! \big) \cos \big(\!\tilde{\phi}_{2t}\! \big) 
  -\tilde{r}_1\sin\! \big(\!\left(\!1{-}t\!\right)\tilde{\phi}_{01}\! \big) \cos \big(\!\tilde{\phi}_{20}\! \big) -\tilde{r}_{1}\sin \big(\!t
    \tilde{\phi}_{01}\! \big) \cos \big(\!\tilde{\phi}_{21}\! \big) }{\sin \big(\!
    t\tilde{\phi}_{01}\! \big)  \big(\!1{-}\tilde{\sisi}\left(\!t\!\right)\! \big) \sfd^{2}_{\mathcal{C}
  }\left(\!\tilde{z}_{0},\tilde{z}_{1}\!\right)}
\\&
= \frac{\sin \big(\!\tilde{\phi}_{01} \! \big) \cos
  \big(\!\tilde{\phi}_{2t}\! \big) -\sin\!
  \big(\!\left(\!1{-}t\!\right)\tilde{\phi}_{01}\! \big) \cos
  \big(\!\tilde{\phi}_{20}\! \big) -\sin \big(\! t\tilde{\phi}_{01}\!
  \big) \cos \big(\!\tilde{\phi}_{21}\! \big) }{\sin \big(\!
  t\tilde{\phi}_{01}\! \big) \sin\!
  \big(\!\left(\!1{-}t\!\right)\tilde{\phi}_{01}\! \big) }\times
\frac{\tilde{r}_{1}\sin \big(\!\left(\! 1{-}t
    \!\right)\tilde{\phi}_{01}\! \big) +\tilde{r}_{0}\sin
  \big(\!t\tilde{\phi}_{01}\! \big)
}{\sfd^{2}_{\mathcal{C}}\left(\!\tilde{z}_{0},\tilde{z}_{1}\!\right)} 
\end{split}
\end{equation}

 Part (A). Let's first assume that
$\left(\YX,\sfd_{\YX}\right)$ satisfies $K$-semiconcavity along
$\xx_{01}\left(t\right),$ with respect to $x_{2}$.
If the left term of the last line in \eqref{a-1} is negative then we directly get a bound for $\tilde{A}\big(\!\tilde{\sisi}\left(\!t\!\right)\!\big)$ by $1.$ If the aforementioned term is positive, we proceed as follows. 
Using $\sfd^2_\calC( \tilde z_0,
\tilde z_1 ) \geq 4 \tilde r_0 \tilde r_1 \sin^2 \!\big(
\tilde\phi_{01}/2 \big)$, we \EEE can bound the last term in \eqref{a-1} by
\begin{equation}\label{smallterm}
\begin{split}
&\frac{\max\{\tilde{r}_{0},\tilde{r}_{1}\}\left[\sin \big(\left( 1{-}t
    \right)\tilde{\phi}_{01} \big) +\sin \big(t\tilde{\phi}_{01} \big)
  \right]}{4\tilde{r}_{0}\tilde{r}_{1}\sin^{2}
  \big(\tilde{\phi}_{01}/2 \big)
}=\frac{\max\{\tilde{r}_{0},\tilde{r}_{1}\}\sin
  \big(\tilde{\phi}_{01}/2 \big) \cos
  \big(\frac{\left(1-2t\right)}{2}\tilde{\phi}_{01} \big)
}{4\tilde{r}_{0}\tilde{r}_{1}\sin^{2} \big(\tilde{\phi}_{01}/2 \big) }
\\
& = \frac{\cos \big(\frac{\left(1-2t\right)}{2}\tilde{\phi}_{01} \big) }{4\min\{\tilde{r}_{0},\tilde{r}_{1}\}\sin \big(\tilde{\phi}_{01}/2 \big) }\leq \frac{1}{4r_{\min}\sin \big(\tilde{\phi}_{01}/2 \big) }.
\end{split}
\end{equation}

Now the $K$-semiconvexity of $(\YX,\sfd_\YX)$ and Lemma
\ref{apotoenastoallo} provide us with 
\begin{equation}\label{coscos}
\cos \!\big(\sfd_{\YX}\!\left(\tilde{\xx}_{01}\!\left(t\right),x_{2}\!\right)
\!\big)\! \leq \!\left( 1{-}t \right)\cos
\big(\!\sfd_{\YX}\left(x_{2},\tilde{x}_{0}\right) \!\big)  +  t\cos
\big(\!\sfd_{\YX}\left(x_{2},\tilde{x}_{1}\right)\! \big)  + t(
1{-}t)\left(K {+}1\right)\sfd^{2}_{\YX}\left(\tilde{x}_{0},\tilde{x}_{1}\right). 
\end{equation}
Hence, by combining \eqref{a-1}, \eqref{smallterm}, \eqref{coscos} we get

\begin{equation*}
\begin{split}
\frac{4r_{\min} \big(\!\tilde{A}
  \big(\!\tilde{\sisi}\left(\!t\!\right)\! \big) {-}1\! \big)
}{r_{2}}&\leq \frac{\cos \big(\!\tilde{\phi}_{20}\! \big)
  \big(\!\left(\! 1{-}t \!\right)\sin \big(\!\tilde{\phi}_{01}\! \big)
  {-}\sin \big(\!\left(\! 1{-}t \!\right)\tilde{\phi}_{01}\! \big) \!
  \big) {+}\cos \big(\!\tilde{\phi}_{21}\! \big)  \big(\!t\sin
  \big(\!\tilde{\phi}_{01}\! \big) {-} \sin \big(\!t
  \tilde{\phi}_{01}\! \big) \! \big) }{\sin \big(\!
  \frac{\tilde{\phi}_{01}}{2}\! \big) \sin \big(\!
  t\tilde{\phi}_{01}\! \big) \sin\!
  \big(\!\left(\!1{-}t\!\right)\tilde{\phi}_{01}\! \big) } 
\\
&\quad + (K{+}1) \frac{t\left(\! 1{-}t \!\right)\sin \big(\! \tilde{\phi}_{01}\! \big) \tilde{\phi}^{2}_{01}}{\sin \big(\! \frac{\tilde{\phi}_{01}}{2}\! \big) \sin \big(\! t\tilde{\phi}_{01}\! \big) \sin\! \big(\!\left(\!1{-}t\!\right)\tilde{\phi}_{01}\! \big) }.
\end{split}
\end{equation*}
 Exploiting  Lemma \ref{similarsins} and using
$\sin\left(2y\right)=\sin\left(y\right) \cos\left(y\right)$ we arrive
at 
\begin{equation}
\begin{split}
	\frac{4r_{\min} \big(\tilde{A}
          \big(\tilde{\sisi}\left(t\right) \big) {-}1 \big)
        }{r_{2}}&\leq \frac{\cos \big(\tilde{\phi}_{20} \big)
          \left(Ct\left( 1{-}t
            \right)\tilde{\phi}^{3}_{01}\right){+}\cos
          \big(\tilde{\phi}_{21} \big)  \big(Ct\left( 1{-}t
          \right)\phi^{3}_{01} \big) }{\sin \big(
          \frac{\tilde{\phi}_{01}}{2} \big) \sin \big(
          t\tilde{\phi}_{01} \big) \sin\!
          \big(\left(1{-}t\right)\tilde{\phi}_{01} \big) }
\\
& \quad + (K{+}1) \frac{2\cos \big(\frac{\tilde{\phi}_{01}}{2} \big) \sin \big(\frac{\tilde{\phi}_{01}}{2} \big) t\tilde{\phi}_{01}\left( 1{-}t \right)\tilde{\phi}_{01}}{\sin \big( \frac{\tilde{\phi}_{01}}{2} \big) \sin \big( t\tilde{\phi}_{01} \big) \sin\! \big(\left(1{-}t\right)\tilde{\phi}_{01} \big) }.
\end{split}
\end{equation}

Finally we set
$M=\max_{y\in[0,2\mathcal{D}]}\frac{y}{\sin\left(y\right)}$ and
use $\tilde \phi_{ij} \in [0,2 \mathcal D]$ to obtain 
\[
\frac{4r_{\min} \big(\tilde{A} \big(\tilde{\sisi}\left(t\right) \big)
  -1 \big) }{r_{2}}\leq 2CM^{3}+2CM^{3}+(K{+}1 )M^{2}.
\]
In particular this implies 
\[
\tilde{A}\left(t\right) \leq   K' \quad \text{with} \quad 
K':=  r_{2}\frac{
  2CM^{3}+2CM^{3}+\left(K{+}1\right)M^{2}}{4r_{\min}}+1  .
\]
 Thus, Part (A) is shown in view of Corollary \ref{asasa}. 

 Part (B). To derive the opposite conclusion we again start from
 \eqref{a-1} and obtain
\begin{equation*}
\begin{split}
 &\cos \big(\tilde{\phi}_{2t} \big) -\frac{\sin\!
   \big(\left(1{-}t\right)\tilde{\phi}_{01} \big) }{\sin
   \big(\tilde{\phi}_{01} \big) }\cos \big(\tilde{\phi}_{20} \big)
 -\frac{\sin \big( t \tilde{\phi}_{01} \big) }{\sin
   \big(\tilde{\phi}_{01} \big) }\cos \big(\tilde{\phi}_{21} \big)
\\&
= \frac{\tilde{A} \big(\tilde{\sisi}\left(t\right) \big)
  -1}{r_{2}}\frac{\sfd^{2}_{\mathcal{C}
  }\left(\tilde{z}_{0},\tilde{z}_{1}\right)}{\tilde{r}_{1}\sin
  \big(\left( 1{-}t \right)\tilde{\phi}_{01} \big) +\tilde{r}_{0}\sin
  \big(t\tilde{\phi}_{01} \big) }\frac{\sin \big( t\tilde{\phi}_{01}
  \big) \sin\big(\left(1{-}t\right)\tilde{\phi}_{01}\big)}{\sin
  \big(\tilde{\phi}_{01} \big) } \, . 
\end{split}
\end{equation*}
 Using the $K$-semiconcavity in $(\calC,\sfd_\calC)$ we can use 
$\tilde{A}\left(t\right)\leq K$, and with Lemma \ref{similarsins} we get
\begin{equation*}
\begin{split}
&\cos\left( \sfd_\YX 
  \left(x_{2},\xx_{01}\left(t\right)\right)\right)-\left(1{-}t\right)\cos\left(
  \sfd_\YX \left(x_{2},\tilde{x}_{0}\right)\right)-t\cos\left(
   \sfd_\YX \left(x_{2},\tilde{x}_{1}\right)\right)
\\&
= \left(\left(1{-}t\right)-\frac{\sin\big(\left(1{-}t\right)\tilde{\phi}_{01}\big)}{\sin\big(\tilde{\phi}_{01}\big)}\right)\cos\left(\tilde{\phi}_{20}\right)+
\left(t-\frac{\sin\big(
      t\tilde{\phi}_{01}\big)}{\sin\big(\tilde{\phi}_{01}\big)}\right)\cos\left(\tilde{\phi}_{21}\right)
\\&
\quad
+\frac{\tilde{A}\big(\tilde{\sisi}\left(t\right)\big)-1}{r_{2}}\, 
 \frac{ \sfd^{2}_\calC \left(\tilde{z}_{0},\tilde{z}_{1}\right)}
{ \tilde{r}_{1}\sin\big(\left(1{-}t\right)\tilde{\phi}_{01}\big) 
  +\tilde{r}_{0}\sin\big(t\tilde{\phi}_{01}\big)}\,
  \frac{\sin\big( t\tilde{\phi}_{01}\big) \sin\big(\left(1{-}t\right)\tilde{\phi}_{01}\big)}{\sin\big(\tilde{\phi}_{01}\big)}
\\&
 \leq \left(\left(1{-}t\right)-
  \frac{\sin\big(\left(1{-}t\right)\tilde{\phi}_{01}\big)}
        {\sin\big(\tilde{\phi}_{01}\big)}\right)
      \cos\big(\tilde{\phi}_{20}\big)  +  \left(t-
  \frac{\sin\big( t\tilde{\phi}_{01}\big)}
  {\sin\big(\tilde{\phi}_{01}\big)}
  \right) \cos\big(\tilde{\phi}_{21}\big)\\
& \quad +  
\frac{K{-}1}{r_{2}}\,\sfd^{2}_{\mathcal{C}}\left(\tilde{z}_{0},\tilde{z}_{1}\right)
 \,
 \frac{1}{r_{\min}\cos\big(\frac{\left(1-2t\right)}{2}\tilde{\phi}_{01}\big)}
 \,\frac{1}{\sin\big(\frac{\tilde{\phi}_{01}}{2}\big) 
  \sin\big(\tilde{\phi}_{01}\big)}\: t\left(1{-}t\right)\tilde{\phi}^{2}_{01}
\\&
\leq \left[
  2\frac{C\tilde{\phi}_{01}}{\sin\big(\tilde{\phi}_{01}\big)}
      +\frac{K{-}1}{r_{2}} \,
   \frac{1}{r_{\min}\cos\big(\frac{\tilde{\phi}_{01}}{2}\big) 
   \cos\big(\frac{\left(1-2t\right)}{2}\tilde{\phi}_{01}\big)}
  \,
   \frac{\sfd^{2}_{\mathcal{C}}\left(\tilde{z}_{0},\tilde{z}_{1}\right)}
   {\sin^{2}\big(\frac{\tilde{\phi}_{01}}{2}\big)} \right]
 t \left(1{-}t\right)\tilde{\phi}^{2}_{01}\: .
\end{split}
\end{equation*}
Using $M=\max_{y\in[0,2\mathcal{D}]}\frac{y}{\sin\left(y\right)}$
as above and recalling  \eqref{kou} we arrive at the desired result 
\begin{equation*}
\cos\left( \sfd_\YX 
  \left(x_{2},\tilde{\xx}_{01}\left(t\right)\right)\right)-\left(1-t\right)\cos\left(
   \sfd_\YX 
  \left(x_{2},\tilde{x}_{0}\right)\right)-t\cos\left( \sfd_\YX
   \left(x_{2},\tilde{x}_{1}\right)\right)\leq
K't\left(1{-}t\right)
\sfd^2_\YX  \left(x_{2},\tilde{x}_{1}\right) 
\end{equation*}
with $K'
=2CM+\frac{K-1}{r_{\min}r_{2}\cos^{4}\left(\mathcal{D}\right)}M^{2}$. Applying
Corollary \ref{asasa} once again, the proof of Proposition
\ref{mainKseminconcavitylemma} is complete. 
\end{proof}

Now we directly recover the following Corollary.

\begin{corollary}\label{Ksemicor}
	Let 	$(\YX,\sfd_{\YX})$ be a geodesic metric space, and $(\calC, \sfd_{\calC})$ the cone over $(\YX,\sfd_{\YX}).$ If $(\YX,\sfd_{\YX})$  satisfies $K\!$-semiconcavity, for some $K\!\in\!\mathbb{R}$ on a set $\mathcal{A},$ then  $(\calC, \sfd_{\calC})\!$ satisfies  $K'\!$-semiconcavity on $\left(\mathcal{A}\cap B(x,\mathcal{D})\right)\times [R_{min},R_{max}],$ for every $x\in\YX, \mathcal{D}<\frac{\pi}{2}, R_{min}>0,R_{max}>0$ where $K'$ can be chosen to depend only on $K, R_{min}, R_{max}, d.$ On the other hand, if $(\calC, \sfd_{\calC})$ satisfies  $K$-semiconcavity, for some $K\in\mathbb{R},$ on a set of the form $\left(\mathcal{A}\cap B(x,\mathcal{D})\right)\times \{1\},$ then $(\YX,\sfd_{\YX})$  satisfies $K'$-semiconcavity on $\left(\mathcal{A}\cap B(x,\mathcal{D})\right),$ where $K'$ can be chosen to depend only on $K,\mathcal{D}.$  
\end{corollary}

\section{Hellinger--Kantorovich space $(\M(X),\HK_\dd)$}
\label{se:HK}

In the sequel we are going to work on spaces of measures over
some underlying (geodesic) metric space  $(X,\sf\sfd_X)$ and denote the
associated cone by $(\mfC,\sf\sfd_{\mfC})$. A typical example will
be $X=\Omega \subset \R^d$, where $\Omega$  convex, compact and
equipped with the Euclidean metric $\sfd_X(x,y)=|x{-}y|$. All the
abstract theory from above applies to these couples; however, our main
interest lies in the case where $(\calC,\sfd_\calC)$ is identified with
$(\M(X),\HK_{\dd})$ while the spherical space $(\YX,\sfd_\YX)$ will be
given in terms of the probability measures $\cP(X)$ equipped with the
metric $\SHK_{\dd}$, which is still to be constructed. 

\subsection{Notation and preliminaries}
\label{subsec:notation}

For the sequel, let $(X,\sfd_X)$ be a geodesic, Polish space. We will
denote by $\M(X)$ the space of all 
nonnegative and finite Borel measures on $X$ endowed with the 
weak topology induced by the duality with the continuous and bounded
functions of $\rmC_b(X)$. The subset of measures with 
finite quadratic moment will be denoted by $\M_2(X)$. The spaces $\cP(X)$ and $\cP_2(X)$ are the corresponding subsets of probability
measures.

If $\mu\in \M(X)$ and $T:X\to Y$ is a Borel map, $T_\sharp \mu$ will denote the
push-forward measure on $\M(Y)$, defined by
\begin{equation}\label{eq:push_forw}
T_\sharp\mu(B):=\mu(T^{-1}(B))\quad
\text{for every Borel set $B\subset Y$}.
\end{equation}
We will often denote elements of $X\times X$ by $(x_0,x_1)$ and
the canonical projections by $\pi^i:(x_0,x_1)\to x_i$ for $i=0,1$.
A transport plan on $X$ is a measure $M_{01}\in \M(X{\times} X)$ with marginals
$\mu_i:=\pi^i_\sharp M_{01}$.

Given a couple of measures $\mu_0,\mu_1\in \mathcal{P}_2(X)$ with $
\mu_0(X)=\mu_1(X)$, its
(quadratic) Kantorovich-Wasserstein distance $\sfW_{\sfd_{X}}$ is defined by
\begin{equation}
\label{eq:11}
\sfW^{2}_{\sfd_{X}}\!(\mu_0,\mu_1)\!:=\!\min\!\bigg\{\!\!\iint 
\sfd_X^2(x_0,x_1)\,\d M_{01}(x_0,x_1)\,\Big|\, M_{01}\in \mathcal{P}_{2}(X{\times} X),\ \pi^i_\sharp M_{01}=\mu_i,\ i=0,1\bigg\}.
\end{equation}
We refer to \cite{AmGiSa05GFMS} for a survey on the Kantorovich--Wasserstein
distance and related topics.

\subsection{The logaritmic-entropy transport formulation}
\label{su:LogEntrF}

Here we first provide the definition of the $\HK_{\dd}(\mu_0,\mu_1)$
distance in terms of a minimization problem that balances a specific
transport problem of measures $\sigma_0\mu_0$ and $\sigma_1\mu_1$ with
the relative entropies of $\sigma_j\mu_j$ with respect to
$\mu_j$. From this, the fundamental scaling property
\eqref{eq:HK.scale} of $\HK_{\dd}$ will follow, see Theorem
\ref{th:ScalHK}.

For the characterization of the Hellinger--Kantorovich distance via
the static Logarithmic-Entropy Transport (LET) formulation, we 
define the logarithmic entropy density
$F:\left[0,\infty\right[\to[0,\infty[$ via $F(r) =r \log r - r+1 $ and
the cost function $L_{\dd}:\left[0,\infty\right[\to[0,\infty]$ via
$L_{\dd}(R) = -2 \log\left(\cos\left(R\dd\right)\right)$ for
$R\dd<\frac{\pi}{2}$ and $L_{\dd}\equiv+\infty$ otherwise. For
given measures $\mu_0,\mu_1$ the LET
functional $\LET_{\dd}(\,\cdot\,;\mu_0,\mu_1): \M(X\times
X)\to[0,\infty[$ reads
\begin{equation}\label{def:let}
\LET_{\dd}(\bfH_{01};\mu_0,\mu_1):=
\int_X F(\sigma_0)\d\mu_0+ \int_X F(\sigma_1)\d\mu_1+
\iint_{X\times X}  L_{\dd}(\sfd_X(x_0,x_1)) \d\bfH_{01} 
\end{equation}
with $\eta_i:=  (\pi_i)_\sharp
\bfH_{01}=\sigma_i\mu_i\ll\mu_i$.  With this, the equivalent
formulation of the Hellinger--Kantorovich distance as
entropy-transport problem reads as follows.

\begin{theorem}[{LET formulation, \cite[Sec.\,5]{LiMiSa14?OETP}}]
\label{thm:LET}
For every $\mu_0,\mu_1\in \M(X)$ we have
\begin{equation}\label{eq:HK.ET}
  \HK^{2}_{\dd}(\mu_0,\mu_1)=\min\big\{ 
  \LET_{\dd}(\bfH_{01};\mu_0,\mu_1)\,\big|\,
  \bfH_{01}\in\M(X\times X),\  (\pi_i)_\sharp
  \bfH_{01}  \ll\mu_i\big\}.
\end{equation}
\end{theorem}

An optimal transport plan $\bfH_{01}$, which always exists, gives the
effective transport of mass. Note, in particular, that only
$\eta_i\ll\mu_i$ is required and the cost of a deviation of $\eta_i$
from $\mu_i$ is given by the entropy functionals associated with
$F$. Moreover, the cost function $L_{\dd}$ is finite in the case $
 \dd \,\sfd_X(x_0 , x_1)  <\frac{\pi}{2}$, which highlights the
sharp threshold between transport and pure absorption-generation
mentioned earlier.

In general, optimal transport plans $\bfH_{01}\in\M(X\times X)$
are not unique. However, due to the strict convexity of $F$ its marginals
$\eta_i$ are unique such that the non-uniqueness of the plan $\bfH_{01}$
is solely a property of the optimal transport problem for 
the cost $L_{\dd}$.

\begin{theorem}[Optimality conditions {\cite[Thm.\,6.3]{LiMiSa14?OETP}}]
	\label{thm:optimality.cond} 
 For $\mu_0,\mu_1\in\M(X)$ let
 \begin{equation}
	\label{eq:33}
	A_i':=\left\{x\in X: \dd\dist(x,\supp(\mu_{1-i}))<\frac{\pi}{2}\right\},
	\quad
	A_i'':=X\setminus A_i',
 \end{equation}
 with the related decomposition
 \begin{equation}
	\label{eq:43}
	\mu_i:=\mu_i'+\mu_i'',\quad
	\mu_i':=\mu_i\res{A_i'},\quad
	\mu_i'':=\mu_i\res{A_i''}. 
 \end{equation}
	
 \begin{enumerate}[(i)]
  \item A plan $\bfH_{01}\in\M(X\times X)$ is optimal for the
   logarithmic entropy-transport problem in \eqref{eq:HK.ET} for
   $\mu_0$, $\mu_1\in\M(X)$ if and only if
   $\iint L_{\dd} \d\bfH_{01}<\infty$ and its marginals $\eta_i$ are
   absolutely continuous with respect to $\mu_i$ with densities
   $\sigma_i$, which satisfy (we adopt the convention $0\cdot
   \infty=1$ in \eqref{eq:34})
   \begin{subequations}
   \begin{align}
    \label{eq:32}
    \sigma_i=0	\quad
     &\text{on}\quad \supp(\mu_i'')\subset A_i''
   \\
     \sigma_i>0 \quad&\text{on}\quad X\setminus \supp(\mu_i''),
   \\
   \label{eq:34}
    \sigma_0(x_0)\sigma_1(x_1)\ge
    \cos_{\pi/2}^2\left(\dd \,\sfd_X(x_0,x_1)\right)\quad&\text{on}\quad
                        X\times X ,
   \\ \label{eq:ww}
     \sigma_0(x_0)\sigma_1(x_1)= \cos_{\pi/2}^2\left(\dd
       \,\sfd_X(x_0,x_1) \right)
    \quad&\bfH_{01}\text{-a.e.~on}\quad A_0'\times A_1'.
   \end{align}
  \end{subequations}
		
 \item  Moreover, we have that
  \begin{subequations}
   \begin{gather}
    \label{eq:48}
     \HK^{2}_{\dd}(\mu_0,\mu_1)= \HK^{2}_{\dd}(\mu_0',\mu_1') +
     \HK^{2}_{\dd}(\mu_0'',\mu_1''),
    \\
     \label{eq:60}
     \text{the couples $(\mu_0,\mu_1) $ and $(\mu'_0,\mu'_1)$ share
     the same optimal plans $\bfeta$, and}
    \\
     \label{eq:61}
     \HK^{2}_{\dd}(\mu_0'',\mu_1'')=\mu_0''(X)+\mu_1''(X)=
     \mu_0(X\setminus A_0')+\mu_1(X\setminus A_1').
   \end{gather}
  \end{subequations}
 \end{enumerate}
\end{theorem}

 We easily obtain upper bounds on $\HK^{2}_{\dd}$ by inserting
$H_{01}=0$ into the definition of $\LET_\dd$ in \eqref{eq:HK.ET}, viz,
for $\mu_0,\mu_1\in \M(X)$ and $\nu_0,\nu_1\in \cP(X)$ we have 
\begin{equation}
\label{massbound} 
\HK^{2}_{\dd}(\mu_{0},\mu_{1})\leq
\mu_{0}(X) +\mu_{1}(X)  \qquad
\text{and} \qquad \HK^{2}_{\dd}(\nu_{0},\nu_{1}) \leq 2. 
\end{equation}

\subsection{Scaling property of $ \HK_{\dd}$ and the definition of
  $(\cP(X),\SHK_\dd).$} 
\label{su:asa} 

Here we give the basic scaling property of the Hellinger--Kantorovich
distance that is  the basis of our interpretation of $(\M(X),\HK_\dd)$
as a cone space. 

\begin{theorem}[Scaling property of $\HK_{\dd}$]
\label{th:ScalHK}
For all $\mu_{0},\mu_{1}\in\M(X)$ and  $r_0,r_1 \geq 0$ we have
\begin{equation} 
  \label{eq:ScalHK}
   \HK_{\dd} ^{2}(r_0^2\mu_{0}, r_1^2\mu_{1})=r_0r_1 
  \HK_{\dd}^{2}(\mu_{0},\mu_{1}) +(r_0^2{-}r_0r_1)\mu_{0}(X) + (r_1^2
  {-} r_0r_1)\mu_{1}(X).
\end{equation}
Evenmore, if $\bfH_{01}$ is an optimal plan for the $\LET_\dd$
formulation of $ \HK_{\dd} (\mu_{0},\mu_{1}),$ then
$\bfH_{01}^{r_0r_1}= r_0r_1\bfH_{01}$ is an optimal plan for $ \HK_{\dd}
(r_0^2\mu_{0}, r_1^2\mu_{1}).$
\end{theorem}
\begin{proof}
  Let $\bfH$ be the minimizer in the definition of $\LET_{\dd}(
  \cdot;\mu_{0},\mu_{1}) $.  We now calculate the scale version
  $\LET_{\dd} (r_0r_1\bfH_{01}; r_0^2 \mu_{0}, r_1^2\mu_{1})$ as
  an upper estimate for $\inf \LET_{\dd} (\cdot;r_0^2\mu_{0},
  r_1^2\mu_{1}) =\HK_\dd(r_0^2\mu_{0}, r_1^2\mu_{1})^2 $. For the
  relative densities $\sigma^{r_0r_1}_{0}$ and $\sigma^{r_0r_1}_{1}$
  we calculate
\begin{equation*}
 \eta^{r_0r_1}_{0}= r_0r_1 \eta_{0}= r_0r_1
 \sigma_{0}\mu_{0}=\frac{r_1}{r_0}\sigma_{0}\, r_0^2 \mu_{0} 
 \quad\text{and}\quad 
 \eta^{r_0r_1}_{1} = r_0r_1 \eta_{1}=
 r_0r_1\sigma_{1}\mu_{1}=\frac{r_0}{r_1}\sigma_{1} \, r_1^2\mu_{1}, 
\end{equation*}
from which we obtain $\sigma_{0}^{r_0r_1}=\frac{r_1}{r_0}\sigma_{0} $
and $\sigma_{1}^{r_0r_1}= \frac{r_0}{r_1}\sigma_{1}$. 
To determine  $\LET_{\dd} (r_0r_1\bfH_{01}; r_0^2 \mu_{0},
r_1^2\mu_{1})$ we first calculate the relative entropy for
$\sigma_{0}^{r_0r_1}$: 
\begin{align*}
&\int_{X}F(\sigma_{0}^{r_0r_1}(x_{0}))r_0^2\mu_{0}(\d x_{0}) =
\int_{X}\big(\sigma^{r_0r_1}_{0}(x_{0})\log(\sigma^{r_0r_1}_{0}(x_{0}))-
  \sigma^{r_0r_1}_{0}(x_{0}) + 1\big) r_0^2\mu_{0}(\d x_{0})
\\&=
\int\left(\frac{r_1}{r_0}\sigma_{0}(x_{0})\log 
 \big(\frac{r_1}{r_0} \sigma_{0}(x_{0})\big)- \frac{r_1}{r_0}
\sigma_{0}(x_{0}) + 1\right) r_0^2\mu_{0}(\d x_{0})
\\&=  \int_{X} 
   \left( r_0r_1\big(\sigma_{0}(x_{0})\log\sigma_{0}(x_{0}) -\sigma_{0}(x_{0})
   +1\big)+r_0r_1\log\big(\frac{r_1}{r_0}\big)\sigma_{0}(x_{0})
   +(r_0^2{-}r_0r_1) \right)\mu_{0}(\d x_{0})
\\&= r_0r_1\int_{X}F(\sigma_{0}(x_{0}))\mu_{0}(\d x_{0}) + r_0r_1
\log\big(\frac{r_1}{r_0}\big) \eta_{0}(X) 
+(r_0^2{-}r_0r_1)\mu_{0}(X).
\end{align*}
Adding the corresponding term for $\sigma^{r_0r_1}_{1}$ we see that
the middle term cancels because we have $\eta_0(X)=\eta_1(X)$, and we
arrive at the following upper bound:  
\begin{align*}
  \HK_{\dd}^{2}(r_0^2\mu_{0}, r_1^2 \mu_{1}) &\leq \LET_{\dd}(r_0r_1
  \bfH_{01};r_0^2\mu_{0}, r_1^2\mu_{1})
  \\
  & = r_0r_1\left(\int_{X}F(\sigma_{0})\mu_{0}(dx_{0})
    +\int_{X}F(\sigma_{1})\mu_{1}(dx_{1})\right)
  +(r_0^2{-}r_0r_1)\mu_{0}(X)
 \\&\qquad + (r_1^2{-}r_0r_1) \mu_{1}(X)
  + \iint\nolimits_{X\times X}L_{\dd}(  \sfd_X(x_{0},x_{1})
   )r_0r_1\bfH_{01}(\d x_{0}\d x_{1} ) \\
& =r_0r_1
  \LET_{\dd}( \bfH_{01};\mu_{0},\mu_{1})+(r_0^2{-}r_0r_1)
  \mu_{0}(X) + (r_1^2{-}r_0r_1)\mu_{1}(X) \\
&= r_0r_1
  \HK^{2}_{\dd}(\mu_0,\mu_1)+(r_0^2{-}r_0r_1) \mu_{0}(X)
  +(r_1^2{-}r_0r_1)\mu_{1}(X),
\end{align*}
where in the last step we used that $\bfH_{01}$ is optimal. 

By replacing $r_j $ by $1/r_j$ and $\mu_j$ by $r_j^2\mu_j$ this
upper bound also yields
\begin{equation*}
	\HK^{2}_{\dd}(\mu_{0},\mu_{1}) \leq
        \frac{1}{r_0r_1}\HK^{2}_{\dd}(r_0^2 \mu_{0}, r_1^2\mu_{1}) + 
 \left(\frac{1}{r_0^2}-\frac{1}{r_0r_1}\right)r_0^2\mu_{0}(X) +
 \left(\frac{1}{r_1^2} -\frac{1}{r_0r_1}\right)r_1^2\mu_{1}(X).
\end{equation*}
Multiplying both sides with $r_0r_1$ and rearranging the terms, we obtain the
desired lower bound for $\HK^{2}_{\dd}(r_0^2\mu_{0}, r_1^2\mu_{1}),$ and the
scaling relation \eqref{eq:ScalHK} is proved. 
\end{proof}

 The above theory for the Hellinger-Kantorovich distance $\HK_\dd$ and
the abstract Theorem \ref{coneinverse} allows us now to introduce a
new metric distance on the probability measure $\cP(X)$ via
\begin{equation}
  \label{eq:Def.SHK}
  \SHK_\dd(\nu_0,\nu_1):= \arccos\Big( 1- \frac12
  \HK^{2}_\dd(\nu_0,\nu_1) \Big) \quad \text{for }\nu_0,\nu_1 \in \cP(X),
\end{equation}
where the mass bound \eqref{massbound} gives $\HK_\dd(\nu_0,\nu_1)\leq
\sqrt{2}$, which guarantees that the argument of ``arccos'' is in the
interval $[0,1]$, so that $\SHK_\dd$ takes values in $[0,\pi/2]$. 
The mapping $[\cdot,\cdot]: \cP(X)\ti [0,\infty)\to \M(X)$ is given
via

\[  \cP(X)\ti [0,\infty) \ni (\nu,r) \ \mapsto\ [\nu,r] \ \widehat= \ r\nu \in \M(X). \]
 
The general theory of Section \ref{se:Cones} shows that $\SHK_\dd$ is
indeed a metric and, even more, it is a geodesic metric if
$(X,\sfd_X)$ is a geodesic space. It is shown in \cite{LiMiSa14?OETP}
that $\HK_\dd$ is geodesic and hence our Theorem \ref{th:Project}
shows that $(\cP(X),\SHK_\dd)$ is a geodesic space. 
We summarize the result as follows. 

\begin{theorem}\label{thm:diam}
The Hellinger--Kantorovich space $(\M(X),\HK_{\dd})$ 
can be identified with the cone over the spherical space 
$(\mathcal{P}(X),\SHK_{\dd})$ in the above sense. Moreover, the latter
has diameter less or equal to $\frac{\pi}{2}.$
\end{theorem}

\subsection{Cone space formulation}
\label{su:conespace} 

Amongst the many charaqctierizations of $\HK_{\dd}$ discussed in
\cite{LiMiSa14?OETP} there is one that connects $\HK_{\dd}$ with 
the classic Kantorovich--Wasser\-stein distance on the cone $\mfC$
over the base space $(X,\dd \sfd_{X})$ with metric 
\begin{equation}
\label{eq:7}
\mathsf \sfd^{2}_{\mfC,\dd}(z_0,z_1):=
r_0^2+r_1^2-2r_0r_1\cos_{\pi}\left(\dd \sfd_{X}(x_{0},x_{1})\right),\quad
z_i=[x_i,r_i],
\end{equation} 
where  as above $\cos_{b}(a)=\cos(\min\{b,a\})$.
Measures in $\M(X)$ can be ``lifted'' to measures in $\M(\mfC)$, 
e.g.\ by considering the measure $\mu\otimes\delta_{1}$ for $\mu\in\M(X)$.
On the other, we can define the projection of measures in $\M_2(\mfC)$
onto measures in $\M(X)$ via
\[
\mfP:\left\{
\begin{array}{ccc}
\M_2(\mfC)&\to&\M(X),\\[0.2em]
\lambda&\mapsto&\int_{r=0}^\infty r^2\,\lambda(\cdot,\rmd r).\
\end{array}\right.
\]
For example, the lift $\lambda = m_0\delta_{\{\boldsymbol{0}\}} 
+ \mu\otimes\frac{1}{r(\cdot)^2}\delta_{r(\cdot)}$, with $m_0\geq 0$
and $r:\mathrm{supp}(\mu)\to\left]0,\infty\right[$ arbitrary,
gives $\mfP\lambda= \mu$. Now, the cone space formulation of the Hellinger--Kantorovich distance
of two measures $\mu_0$, $\mu_1\in\M(X)$ is given as follows.

\begin{theorem}[Optimal transport formulation on the cone]
\label{thm:OTcone}
For $\mu_0,\mu_1\in \M(\R^d)$ we have
\begin{subequations}
\label{eq:HK.ConeSpace}
\begin{align}    \label{eq:6}
\HK^2_{\dd}(\mu_0,\mu_1) &= 
  \min\Big\{\sfW_{\sfd_{\mfC,\dd}}^{2}(\lambda_0,\lambda_1) \, \Big|\, 
    \lambda_i\in \cP_2(\mfC),\ \mfP\lambda_i = \mu_i\Big\}\\    
\label{eq:70}
&=\min\Big\{\iint_{\mfC\times\mfC} \sfd_{\mfC, \dd}^{2}(z_0,z_1) 
  \rmd  \Lambda_{01} (z_0,z_1)\,\Big|\, \pi^i_{\sharp}\Lambda_{01}=
 \lambda_i,~\text{and}~\mfP\lambda_i=\mu_i\Big\}.
\end{align}
\end{subequations}
\end{theorem}

\begin{remark}\label{dtilde}
	By \cite[Lem.\,7.19]{LiMiSa14?OETP}, we also have
\begin{equation}	
\HK^{2}_{\dd}(\mu_0,\mu_1)=\min\Big\{\iint_{\mfC\times\mfC}\widetilde{\sfd}^{2}_{\mfC, \dd}(z_0,z_1)\rmd\Lambda_{01}(z_0,z_1)\,\Big|\,
\pi^i_{\sharp}\Lambda_{01}=\lambda_i~\text{and}~\mfP\lambda_i=\mu_i\Big\},	
\end{equation}	
where $\widetilde{\sfd}_{\mfC,
  \dd}^{2}([x_0,r_0],[x_1,r_1])=r_0^2+r_1^2-2r_0r_1\cos_{\pi/2}\left(\dd
  \sfd_{X}(x_{0},x_{1})\right)$  is defined  with the earlier cut-off at
$\pi/2$ instead of $\pi$ as in \eqref{eq:7}. 
\end{remark}

The cone space formulation is reminiscent of classical optimal
transport problems. Here, however, the marginals $\lambda_i$ of the
transport plan $\Lambda_{01}\in\M(\mfC\times\mfC)$ are not fixed, and
it is part of the problem to find an optimal pair of measures
$\lambda_i$ satisfying the constraints $\mfP\lambda_i=\mu_i$ and
having minimal Kantorovich--Wasserstein distance on the cone
$(\mfC,\sfd_\mfC)$.

The squared cone distance $\sf\sfd_\mfC$ has an important scaling
invariance: For  an arbitrary Borel function  $\theta:\mfC^2
\to \left]0,\infty\right[,$ we define the transformation
$\prd\theta:\mfC^2\to\mfC^2$ via
\[  
\prd\theta(z_0,z_1):=
([x_0,r_0/\theta(z_0,z_1)];[x_1,r_1/\theta(z_0,z_1)]), \text{ where }
z_i=[x_i,r_i] .
\]
Its dilation on measures $\Lambda_{01}\in \M(\mfC^2)$ is defined by
\begin{equation}
	\label{eq:40}
	\dil\theta(\Lambda_{01}):=(\prd\theta)_\sharp(\theta^2\Lambda_{01}),\quad
	\text{whenever }\theta\in \rmL^2(\mfC^2;\Lambda_{01}).
\end{equation}
Using the transformation rule, it is easy to see that 
\begin{equation}\label{eq:41}
\int_{\mfC^2} \sfd_{\mfC, \dd}^2(z_0,z_1)\,\d\Lambda_{01}=\int_{\mfC^2} \sfd_{\mfC, \dd}^2(z_0,z_1)\,\d\big(\!\dil\theta(\Lambda_{01})\big).
\end{equation}
\subsection{Characterization of geodesics in $(\cP(X),\SHK_\dd)$.}
\label{su:SHK}

For $X$ being a closed convex subset of $\R^d$ with the Euclidean
distance, we want to show that 
the goedesic curves can be characterized in terms of a generalized
continuity equation and a Hamilton--Jacobi equation. Thus,
$(\cP(X),\SHK_\dd)$ has pseudo-Riemannian structure that is in complete
analogy to that of $(\M(X),\HK_\dd)$ or that of the Wasserstein space
$(\cP(X), \mathsf{W}_2)$. 

Indeed, according to \cite[Eqn.\,(5.1)]{LiMiSa16OTCR} or
\cite[Thm.\,8.19]{LiMiSa14?OETP} all constant-speed geodesics for
$\HK_\dd$ are given as suitable solutions of the coupled system of
equations
\begin{align}
  \label{eq:HK.Geod}
\pl_t \mu + \frac1{\dd^2} \mathrm{div}\big( \mu\nabla \xi \big) = 4\xi
\mu,\quad \pl_t  \xi + \frac1{2\dd^2} |\nabla \xi|^2 + 2 \xi ^2 = 0. 
\end{align}
Here $\xi=\xi(t,x)$ is the dual potential, which satisfies the
Hamilton--Jacobi equation, while the measure $\mu(t)\in
\M(X)$ follows the generlized continuity equation with transport via
$V=\frac1{\dd^2}\nabla \xi$ and growth-decay according to $4\xi$.

We now want to derive the corresponding system for the spherical space
$(\cP(X),\SHK_\dd)$ by applying Theorem \ref{th:Project}, which tells
us that any geodesic $s\mapsto \nu(s)\in \cP(X)$ is a rescaling of 
the geodesic for $\HK_\dd$ connecting $\nu_0$ and $\nu_1$. 

\begin{theorem}[Equation for geodesics in$(\cP(X),\SHK_\dd)$] \label{thm:SHK.Geod} 
The geodesic curves $s \mapsto \nu(s)$ lying in space $(\cP(X),\SHK_\dd)$ are given by 
\begin{equation}\label{eq:SHK.geod}
\pl_s \nu + \frac{1}{\dd^2} \mathrm{\div}\big(\nu\nabla\zeta \big) =4\big( \zeta {-}{\textstyle\int_X\zeta \,\d\nu}\big) \nu 
,\quad \pl_s  \zeta + \frac1{2\dd^2} |\nabla \zeta|^2 + 2 \big(\zeta 
 {-} {\textstyle\int_X\zeta \,\d\nu} \big)^2 = 0, 
\end{equation}
where the equations have to be understood in the sense as described in \cite[Sec.\,8.6]{LiMiSa14?OETP}. 
\end{theorem}

\begin{proof} We simply use the result in \cite[Thm.\,8.19]{LiMiSa14?OETP} and transform it as given the
abstract projection from the cone $(\M(X),\HK_\dd)$ to the spherical space $(\cP(X),\SHK_\dd)$, namely by a 
renormalizing of the mass and a rescaling of the arclength parameter. For this, we use the ansatz
\[\nu(s) = n(s) \mu(\tau(s)) \quad \text{ and } \quad 
\zeta(s,x) = a(s) \xi(\tau(s),x) + b(s),\]
where the functions $n,\;\tau,\;a$, and $b$ have to be chosen
suitably as functions of $s$, but will be independent of $x\in X$.  
In particular, we have 
\begin{equation}
  \label{eq:NuMu.b}
\int_X \zeta(s,\cdot) \,\d\nu(s)= b(s) + a(s) \int_X
\xi(\tau(s),\cdot) \,\d \nu(s)=  b(s) + \frac{a(s)}{n(s)} \int_X
\xi(\tau(s),\cdot) \,\d \mu(s).   
\end{equation}

Using that $(\mu,\xi)$ solves \eqref{eq:HK.Geod}, we obtain the
relations
\[
\pl_s \nu+ \frac{\dot\tau}{a\dd^2} \mathrm{div}\big( \nu\nabla
\zeta\big) = \Big( \frac{4\dot\tau}{a} (\zeta{-} b) + \frac{\dot
  n}n\big) \nu, \quad 
\pl_s\zeta +\frac{\dot\tau}{a\dd^2} |\nabla\zeta|^2
+\frac{2\dot\tau}{a}(\zeta{-}b)^2= \frac{\dot a}a(\zeta{-}b) + \dot
b.
\]
To keep the transport terms, which involve the spatial derivatives,
correct we choose $\tau$ such that $\dot\tau=a$ from now on. As
$\nu(s) \in \cP(X)$, the term on the right-hand side of the continuity
equation must have average $0$, hence we impose 
\begin{equation}
  \label{eq:dot.n}
  \textstyle 4 \int_X \zeta \,\d \nu= 4 b + \dot n/n.
\end{equation}
With this, we can rewrite the Hamilton--Jacobi equation for $\zeta$ in
the form 
\[
\pl_s\zeta +\frac{1}{\dd^2} |\nabla\zeta|^2
+2 \big(\zeta{-}{\textstyle\int_X\zeta \,\d\nu}\big)^2= \big(
\frac{\dot a}a - \frac{\dot n}n\big)\zeta + \dot b -\frac{\dot a}a b -
2 b^2 +2\big( {\textstyle\int_X\zeta \,\d\nu}\big)^2 .
\]
Choosing  further $a=n$ the right-hand side simplifies further,
because the term linear in $\zeta$ vanishes and the remaining
term is $\dot b + 2(b{-} {\textstyle\int_X\zeta \,\d\nu})^2$. 

Now, we show starting from a solution $(\nu,\zeta)$ of
\eqref{eq:SHK.geod} we can find a solution $(\mu,\xi)$ of
\eqref{eq:HK.Geod}. We first solve $\dot b + 2(b{-}
{\textstyle\int_X\zeta \,\d\nu})^2=0$ with $b(s_0)$ such that
\eqref{eq:NuMu.b} holds at initial time $s_0$. Then, $a=n$ is determined
from \eqref{eq:dot.n} with $n(s_0)=1$. Finally, the reparametrization
$t=\tau(s)$ is obtained from $\dot\tau(s)=a(s)$ and $\tau(s_0)=t_0$.
The inverse direction from a solution $(\mu,\xi)$  of
\eqref{eq:HK.Geod} to a solution $(\nu,\zeta)$ of
\eqref{eq:SHK.geod} works similarly. 
\end{proof} 

The dual dissipation potential $\mathcal R^*$ and the
associated Onsager operator $\mathbb K$, as described in 
\cite{Miel11GSRD,LieMie13GSGC,LMPR17MOGG} for  $(\cP(X),\SHK_\dd)$ are
given formally as 
\begin{align*}
\calR^*_\dd(\nu,\zeta) &= \int_X \Big( \frac1{2\dd^2}|\nabla \zeta|^2 +
2\big( \zeta {-}{\textstyle \int_X\zeta\,\d\nu} \big)^2\Big)\,\d\nu \
\text{ and}
\\
\mathbb K_\dd(\hat\nu)\zeta &=- \frac{1}{\dd^2} \mathrm{div} \big(
\hat\nu \nabla \zeta\big) + 4 \hat\nu \,\big( \zeta {-}{\textstyle
  \int_X\zeta\,\d\nu} \big), 
\end{align*}
where in the latter case $\nu$ is assumed to have the density
$\hat\nu$ with respect to the Lebesgue measure. Note that 
$\calR^*_\dd(\nu,\zeta)$ is no longer affine in $\nu$, but it is still
concave, which reflects the fact that the set of geodesic curves
connecting two  measures $\nu_0$ and $\nu_1\in \cP(X)$ is still
convex, a fact which is inherited from $(\M(X),\HK_\dd)$.

Thus, a gradient flow for a density $\mathcal E(\nu) = \int_X E(\hat
\nu)\,\d x$ would formally take the form 
\[
\pl_t \hat \nu = - \mathbb K_\dd(\hat\nu) \mathrm D \mathcal
E(\hat\nu) = \frac1{\dd^2} \mathrm{div} \big(
 \hat\nu \nabla(E'(\hat\nu)) \big) -  4 \hat\nu \Big(E'(\hat\nu) -
{\textstyle \int_X E'(\hat\nu)\,\hat\nu \!\;\d x} \Big).
\]

Existence results for such gradient-flow equations will be
studied in a forthcoming paper. The next section provides first steps
into this direction.

\section{Finer geometric properties of the Hellinger--Kantorovich and the Spherical Hellinger--Kantorovich spaces}
\label{se:PropHK}

In this section we are going to prove that the metric space
$(X,\sfd_{X})$ satisfies $m$-LAC (cf.\ Definition
\ref{def:m-LAC}), if and only if both $(\M(X),\HK_{\dd})$ and
$(\mathcal{P}(X),\SHK_{\dd})$ satisfy $m$-LAC. This result is
surprising since the cone $(\mfC,\sfd_\mfC)$, which is intrinsically
linked to $(\M(X),\HK_{\dd})$, does not share this equivalence;
however the disturbing role of the apex $\bm\mfo \in \mfC$ is
irrelevant for $\HK_\dd$.

We are also going to prove that under the extra assumption that the metric space $(X,\sfd_{X})$ satisfies $K$-semiconcavity on every ball $B\left(x,\frac{\pi}{2\dd}\right),$ then $(\M(X),\HK_{\dd})$ and $(\mathcal{P}(X),\SHK_{\dd})$ satisfy $K'$-semiconcavity on some sets  $\overline{\M}^{\mathcal{L}}_{\delta}(X), \overline{\mathcal{P}}^{\mathcal{L}}_{\delta}(X)$ respectively, where
$K'$ depends on $\delta,\dd.$ We would like to remark that every space $(X,\sfd_{X}),$ with curvature not less than $\kc,$ for some $\kc\in\mathbb{R},$ satisfies such a property \cite[Lemma 3.3]{Ohta2009a} . As it is was mentioned in Section \ref{su:CurvLAC} (see \cite[Part\,1, Ch.\,6]{OlPaVi14OTTA},
\cite{Savare2007}), when these two properties hold in a space, and a
functional $F$ defined on that space is $\lambda$-convex, then for
every point in the space there exists a unique gradient flow with respect
to $\calF$ starting on that point. In some parallel work, we are aiming
to extend that result to cover cases where $K$-seminconcavity holds
only on suitable collections of subsets, as long as the functionals
$\calF$ have the property that starting from any point that belongs in a
set in the collection, then any minimizer in the JKO scheme, belongs
in an another suitable subset in the class. This way, we are going to
provide several examples of gradient flows in $(\M(X),\HK_{\dd}), (\mathcal{P}(X),\SHK_{\dd}).$

\subsection{Stability of $m$-LAC between 
  $(X,\sfd_X)$, $(\M(X),\HK_{\dd}(X))$,  and $(\cP(X),\SHK_{\dd}(X))$}
\label{se:Angles}

We will start by proving that the metric space $(X,\sfd_{X})$
satisfies $m$-LAC, if and only if both $(\M(X),\HK_{\dd})$ and
$(\mathcal{P}(X),\SHK_{\dd})$ satisfy it too. The proof of the first
is a modification of the proof that if a metric space $(X,\sfd_{X})$
satisfies $m$-LAC, then the Wasserstein space $(\cP_2(X),\sfW_{2})$
over $(X,\sfd_{X})$ also satisfies it, which was kindly communicated
to us by Giuseppe Savar\'e (personal communication, May 2017). Because the cone $(\mfC,\sfd_{\mfC})$ over $(X,\sfd_{X})$ does not
necessarily satisfy $m$-LAC due to the degeneracy at the apex
 (see Theorem \ref{Cones and LAC}), one cannot use the
argument verbatim.  We will show the
desired equivalence by exploiting that the minimizing plans  satisfy the 
optimality conditions. 

\begin{proposition}\label{a.e. LAC}
  Consider $\mu_{0}\in (\M(X),\HK_{\dd})$ such that
  $(X,\sfd_X)$ satifies $m$-LAC for $\mu_{0}$-a.e.\ $x_{0}\in
  X$. Then, $(\M(X),\HK_{\dd}(X))$ satisfies $m$-LAC at $\mu_{0}.$
\end{proposition}
\begin{proof}
  For the proof, we are going to utilize the cone representation
  introduced in Section \ref{su:conespace}. Let
  $\mm_{01},\dots,\mm_{0m}$ be geodesics connecting $\mu_{0}\in\M(X),$
  with $\mu_{i}\in\M(X), i=\{1,\dots,m\}.$ By an application of
  \cite[Thm.\,8.4]{LiMiSa14?OETP}, we can find geodesics
  $\lala_{01},\dots,\lala_{0m}$ in $\mathcal{P}(\mfC),$ such that
  $\mfP\lala_{0i}(t)=\mm_{0i}(t)$ (the fact that we can have
  $\lala_{0i}(0)$ to be equal to some fixed $\lambda_{0}$ for
  $i=1,\dots,m$ is given by \cite[Lemma 7.10]{LiMiSa14?OETP}). By
  \cite[Thm.\,6]{Lisini2006a} we can find optimal geodesic plans
  $\bfLambda_{0\rightarrow i}\in\cP(C[0,1];\mfC)$ in the sense that
  $(e_{t})_{\sharp}\bfLambda_{0\rightarrow i}=\bflambda_{0i}(t).$ By a
  refined version of the glueing lemma we can find a plan
  $\bm{\Lambda}\in\cP((C([0,1];\mfC)^{m}),$ such that
  $\pi^{0\rightarrow i}_{\sharp} \bfLambda=\bfLambda_{0\rightarrow
    i}$. For $\bfLambda$-a.e.\ $\zz=( \zz_{01}, \dots,\zz_{0m})$ we
  have that $\zz_{01},\dots,\zz_{0m}$ are geodesics and
  $\zz_{01}(0)=\dots=\zz_{0m}(0)$. We split the measure $\bfLambda$ in
  two parts $\bfLambda^{\{\mathfrak{\boldsymbol{0}}\}}$ and
  $\bfLambda^{\mfC \setminus\{\mathfrak{\boldsymbol{0}}\}},$ such that
  $\bfLambda^{\{\mathfrak{\boldsymbol{0}}\}}(\zz_{0i}(0) =
  \{\mathfrak{\boldsymbol{0}}\})= \bfLambda(\zz_{0i}(0) =
  \{\mathfrak{\boldsymbol{0}}\})$ and
  $\bfLambda^{\mfC\setminus\{\mathfrak{\boldsymbol{0}}\}}
  (\zz_{0i}(0)\neq \{\mathfrak{\boldsymbol{0}}\}) =
  \bfLambda(\zz_{0i}(0)\neq\{\mathfrak{\boldsymbol{0}}\})$.  For
  $\bfLambda^{\mfC\setminus\{\mathfrak{\boldsymbol{0}}\}}$ let us set
  $\theta_{ij}(\zz)=\varangle_{\rmu\rmp}(\zz_{0i},\zz_{0j}).$ Since
  $m$-LAC is satisfied for $\mu_{0}$-a.e.\ $x_{0}$ in $(X,\sfd_X),$ by
  an application of Theorem \ref{Cones and LAC pointwise}, we have
  that $m$-LAC is satisfied for $(e_{t})_{\sharp}\pi^{0\rightarrow
    i}_{\sharp}\bfLambda^{\mfC\setminus\{\mathfrak{\boldsymbol{0}}\}}$-a.e.\ $z_{0}$
  in $(\mfC,\sfd_\mfC),$ and therefore for
  $\bfLambda^{\mfC\setminus\{\mathfrak{\boldsymbol{0}}\}}$-a.e.\ $\zz=(
  \zz_{01}, \dots,\zz_{0m}).$ We will assume without any loss of
  generality that all geodesics have length equal to $a.$ By applying
  Remark \ref{dtilde},  where we introduced
  $\widetilde\sfd_{\calC,\dd}$ with the cut-off $\pi/2$ instead of
  $\pi$ as in $\sfd_{\calC,\dd}$, we obtain
\begin{align*}
&a^{2} \cos\varangle_{\rmu\rmp}(\mm_{0i},\mm_{0j}) =
\liminf_{s,t\downarrow 0} \frac{1}{2st}\left(\HK^{2}
_{\dd}(\mu_{0},\mm_{0i}(t)){+}\HK^{2}
_{\dd}(\mu_{0},\mm_{0j}(s)){-}\HK^{2}_{\dd}(\mm_{0i}(t),\mm_{0j}(s))\right)
\\ 
&\geq
\liminf_{s,t\downarrow 0}\frac{1}{2st}\left(W^{2}_{\sfd_{\mfC,\dd}} 
(\lambda_{0},\lala_{0i}(t)){+}W^{2}_{
	\sfd_{\mfC,\dd}}(\lambda_{0},\lala_{0j}(s)){-}W^{2}_{
	\widetilde{\sfd}_{\mfC,\dd}}(\lala_{0i}(t),\lala_{0j}(s))\right)\\
& \geq \liminf_{s,t\downarrow
	0} \frac{1}{2st}\int\left( 	\sfd_{\mfC,
	\dd}^{2}(z_{0},\zz_{0i}(t)){+} 	\sfd_{\mfC,
	\dd}^{2}(z_{0},\zz_{0j}(s)){-}\widetilde{	\sfd}_{\mfC,
	\dd}^{2}(\zz_{0i}(t),\zz_{0j}(s))\right) \d  \bfLambda\\ 
&\geq \liminf_{s,t\downarrow 0} 
\frac{1}{2st} \int\left( \sfd_{\mfC,
	\dd}^{2}(\mathfrak{\boldsymbol{0}},\zz_{0i}(t)){+} 	\sfd_{\mfC,
	\dd}^{2}(\mathfrak{\boldsymbol{0}},\zz_{0j}(s)){-}\widetilde{
	\sfd}_{\mfC,\dd}^{2}(\zz_{0i}(t),\zz_{0j}(s))\right)
 \d  \bfLambda^{\{\mathfrak{\boldsymbol{0}}\}}\\   
&\quad +\liminf_{s,t\downarrow 0} \frac{1}{2st} \int\left(\sfd_{\mfC,
	\dd}^{2}(z_{0},\zz_{0i}(t)){+} \sfd_{\mfC, \dd}^{2}
(z_{0},\zz_{0j}(s)){-}\widetilde{\sfd}_{\mfC,
	\dd}^{2}(\zz_{0i}(t),\zz_{0j}(s))\right) \d \bfLambda^{\mfC
	\setminus\{\mathfrak{\boldsymbol{0}}\}}. 
\end{align*} 
The first term in the last sum is strictly positive. For the second term, we 
are able to use $\widetilde{\sfd}_{\mfC, \dd}^{2}(\zz_{0i}(t),\zz_{0j}(s))\leq \sfd_{\mfC, \dd}^{2}(\zz_{0i}(t),\zz_{0j}(s))$.
Therefore, by applying Fatou's lemma we have 
\begin{equation*}
\begin{split}
& a^2  \cos\varangle_{\rmu\rmp}(\mm_{0i},\mm_{0j})\\
&\geq \!\int\liminf_{s, t \downarrow 0} \frac{1}{2st}\left(\sfd_{\mfC,
	\dd}^{2}(z_{0},\zz_{0i}(t)){+}\sfd_{\mfC,
	\dd}^{2}(z_{0},\zz_{0j}(s)) {-} \sfd_{\mfC,
	\dd}^{2}(\zz_{0i}(t),\zz_{0j}(s))\right)  \d \bfLambda^{\mfC
	\setminus\{\mathfrak{\boldsymbol{0}}\}}\\
&\geq
\int \cos(\theta_{ij}(\zz )) \,\d  
 \bfLambda^{\mfC \setminus\{\mathfrak{\boldsymbol{0}}\}}.
\end{split}
\end{equation*}
Thus, applying part (b) of Theorem \ref{Cones and LAC}  for
every choice of positive $b_{i}$ ($i=1,\dots,m$) we find 
\begin{equation*}
\sum_{i,j=1}^{m}\cos(\mm_{0i},\mm_{0j})b_{i}b_{j}\geq\frac{1}{a^2}\int
\left(\sum_{i,j=1}^{m}\cos(\theta_{ij}(\zz))b_{i}b_{j}\right) \, \d
 \bfLambda^{\mfC \setminus\{\mathfrak{\boldsymbol{0}}\}}\geq 0,
\end{equation*}
which is the desired result for $\mu_{0}$. 
\end{proof}

We conclude this subsection with the following main result.

\begin{theorem} The space 
 $(X,\sfd_{X})$ satisfies $m$-LAC, if and only if
  the space $(\M(X),\HK_{\dd})$ satisfies $m$-LAC, if and
  only if the space  $(\mathcal{P}(X),\SHK_{\dd})$ satisfies $m$-LAC.
\end{theorem}
\begin{proof} We simple collect the results from above. 
\\[0.4em]
\underline{$((X,\sfd_{X})\Rightarrow (\M(X),\HK_{\dd}))$:} It is a
straightforward application of Proposition \ref{a.e. LAC}.
\\[0.4em]
\underline{$((\M(X),\HK_{\dd}) \Rightarrow (X,\sfd_{X}))$:}  We just
use Dirac measures, and the fact that geodesics stay within the set of
Dirac measures.
\\[0.4em]
\underline{$((\M(X),\HK_{\dd}) \Leftrightarrow (\mathcal{P}(X),\SHK_{\dd}))$:} 
The proof is a straightforward  application of Theorem \ref{Cones and LAC} part (d), using that $(\mathcal{P}(X),\SHK_{\dd}))$ has diameter less than $\pi/2$ (see Theorem \ref{thm:diam}.)
\end{proof}

\subsection{$K$-semiconcavity on sets of measures with doubling
  properties}\label{su:KSemiconcave}

Here we are going to provide results related to $K$-semiconcavity. We
will start with a general lemma that gives an estimate for the total
mass of the minimizer in $\LET(\cdot;\mu_0,\mu_1)$ (see Theorem \ref{thm:LET}).
By $\calB(X)$ we denote the collection of all Borel sets in
$(X,\sfd_X)$. 
\begin{lemma}
\label{totaleta}
Let $\mu_{0},\,\mu_{1}\in \M(X),$ and let $\bfH_{01}$ be a minimizer
for $\LET_{\dd}(\cdot;\mu_0,\mu_1),$ then 
\begin{equation}
 \label{eq:H.XX}
\bfH_{01}(X\times X) \leq \sqrt{\mu'_{0}(X)\mu'_{1}(X)}\leq
  \sqrt{\mu_{0}(X)\mu_{1}(X)},
\end{equation} 
where $(\mu'_{0},\mu'_{1})$ is the reduced couple of
$(\mu_{0},\mu_{1}).$ Furthermore,  we
have 
\begin{equation}
  \label{eq:H.AX}
  \bfH_{01}(A\times X)\leq 
\sqrt{\mu'_{0}(A)\mu'_{1}\left(A_{\frac{\pi}{2\dd}}\right)} \ 
\text{ for all } A\in \calB(X) ,
\end{equation} 
where $A_{b}=\set{y \in X}{\forall\, x\in A:\;\sfd_{X}(x,y)\leq b}$. Finally, if
$X\subset\mathbb{R}^{d}$ and $\mu_{0},\,\mu_{1} \ll \mathcal{L}$,  and
$T:X \to X$ is a function whose graph is the support of $\bfH_{01} $ (such a function exists by \cite[Theorem 6.6]{LiMiSa14?OETP}),   then
\begin{equation}
\label{secondpart}
\bfH_{01}(A\times T(A))\leq\sqrt{\mu'_{0}(A)\mu'_{1}(T(A))} \ {
\text{ for all } A\in \calB(X).}
\end{equation}
\end{lemma}
\begin{proof}
By \eqref{eq:61}, $(\mu_{0},\mu_{1})$ and $(\mu'_{0},\mu'_{1}),$ share
the same optimal plans. Let $\sigma_{i}$ be the optimal densities
$\frac{d\bfeta_{i}}{d\mu'_{i}}$. Then, the optimality condition
\eqref{eq:ww}, which is valid in the support of $H_{01}$, gives    
\begin{align*}
&\bfH^{2}_{01}(X\times X) =\bigg( \int_{ A_{0}'\times A_{1}'} 1
\d\bfH_{01} \bigg)^2 \overset{\eqref{eq:ww}}{=}\bigg(
\int_{A_{0}'\times A_{1}' } \frac{\cos\big( \dd \sfd_X(x_0,x_1)\big)}
{\sqrt{\sigma_0(x_0) \sigma_1(x_1) }} 
\d\bfH_{01} \bigg)^2\\
&\overset{\cos \leq 1}{\leq}
\bigg( \int_{A_{0}'\times A_{1}' }
\frac{ 1 }{\sqrt{\sigma_0(x_0) \sigma_1(x_1) }}
\d \bfH_{01} \bigg)^2
\overset{\text{C-S}}{\leq}
\bigg(\int_{A_{0}'\times A_{1}'} \frac1{\sigma_0(x_0)} \d \bfH_{01} \bigg)
\bigg(\int_{A_{0}'\times A_{1}'} \frac1{\sigma_1(x_1)} \d \bfH_{01} \bigg)\\
&=\int_{A'_{0}} \frac{1}{\sigma_0} \d \eta_0 \;
\int_{A'_{1}} \frac{1}{\sigma_1} \d \eta_1
= \int_{A_{0}'} \d \mu_0 \;\int_{A'_{1}} \d \mu_1 \ = \ \mu'_0(X)\,\mu'_1(X).
\end{align*}
For showing \eqref{eq:H.AX} we define 
\[\sigma_{1,A}=\frac{\d\bfH_{01}(A\times\cdot)}{\d\mu'_{1}}\hspace{8pt}  \text{and}\hspace{8pt} 
\sigma_{1}=\frac{d\bfH_{01} (X\times \cdot)}{\d\mu'_{1}}.\]
 such that $0\leq \sigma_{1,A} \leq \sigma_{1}$.   We define two
measures $\widetilde{\mu}'_{1}$ and $\overline{\mu}'_{1}$ via 
\begin{equation}\label{resmeasures}
\widetilde{\mu}'_{1}(B)=\int_{B}\frac{\sigma_{1,A}(x_{1})}{\sigma_{1}(x_{1})}\mu'_{1}(\d
x_{1}) \hspace{16pt} \text{and}
\hspace{16pt}\overline{\mu}'_{1}(B)=\mu'_{1}(B)-\widetilde{\mu}'_{1}(B)
\ \text{ for all } B\in \calB(X).
\end{equation}
We have that $(\bfH_{01}) \res{(A{\times} X)},$ is a plan 
between $(\mu'_0) \res{{A}}$ and $\widetilde{\mu}'_{1}$. In a similar way we
see that $(\bfH_{01}) \res{((X{\setminus} A){\times} X)}$ is a plan
between $(\mu'_0) \res{(X{\setminus} A)}$ and $\overline{\mu}'_{1}.$
Also it is straightforward to see that the sum of the cost of the two
plans is equal to the cost of $\bfH_{01},$ therefore these plans must
be both optimal. Now applying the first part,  i.e.\
\eqref{eq:H.XX}, w we have
\begin{align*}
\bfH_{01}(A\times X)&=\big(\bfH_{01} \res{(A {\times} X)} \big)(X\times
        X)\leq\sqrt{ \mu'_0(A)\widetilde{\mu}'_1(X)} \\
&\leq\sqrt{
          \mu'_0(A)\widetilde{\mu}'_1\left(A_{\frac{\pi}{2\dd}}\right)}\leq\sqrt{
          \mu'_0(A)\mu'_1\left(A_{\frac{\pi}{2\dd}}\right)},
\end{align*}
which is the desired result \eqref{eq:H.AX}.
	
Finally, if $\bfH_{01}$ is an optimal plan for $\mu'_0,\mu'_1$, and
$T:X \to X$ is a function whose graph is the support of $\bfH_{01} ,$  then 
$\bfH_{01} \res{(A {\times} T(A))}=\bfH_{01} \res{(A {\times} X)}$ is an optimal plan between
$ \mu'_0 \res{A}$ and $\tilde{\mu}'_1 \res{T(A)}=\tilde{\mu}'_1 ,$ where $\tilde{\mu}'_{1}$ is defined as in \eqref{resmeasures}. Now by applying the same argument as before, we have
\begin{align*}
\bfH_{01}(A\times T(A))&=\big(\bfH_{01} \res{(A {\times} T(A))} \big)(X\times
X)=\big(\bfH_{01} \res{(A {\times} X)} \big)(X\times
X)\leq\sqrt{ \mu'_0(A)\widetilde{\mu}'_1(X)} \\
&\leq\sqrt{ \mu'_0(A)(\widetilde{\mu}'_1\res{T(A)})(X)}\leq \sqrt{ \mu'_0(A)\widetilde{\mu}'_1(T(A))} \leq\sqrt{
	\mu'_0(A)\mu'_1\left(T(A)\right)},
\end{align*}
\end{proof}
\vspace{3pt}

Before we proceed with the main result of this subsection, we are
going to provide some definitions and extra notation. In the
following we use the notation $B(x,r)$ for metric balls in
$(X,\sfd_X)$ and possibly in over metric spaces.

\begin{definition}[Doubling metric space]\label{dbspace}
  A metric space $(X,\sfd_{X})$ is called \emph{doubling}, if for every
  $\mathfrak{D}_{2}\geq \mathfrak{D}_{1}>0,$ there exists a constant
  $C(\mathfrak{D}_{2}/\mathfrak{D}_{1})\geq 1,$ that depends only on
  the ratio, such that every ball of radius $\mathfrak{D}_{2}$ can be covered
  by $C(\mathfrak{D}_{2}/\mathfrak{D}_{1})$ balls of radius
  $\mathfrak{D}_{1}.$ 
\end{definition}

\begin{definition}[Doubling measure on metric space\!]\label{dbmeasure}
  In  $(X,\sfd_{X}),$ a Borel measure $\mathcal{L}$ is
  called \emph{doubling} if for every $\mathfrak{D}_{2}\geq
  \mathfrak{D}_{1}>0,$ it exists a constant
  $\bar{C}(\mathfrak{D}_{2}/\mathfrak{D}_{1})\geq 1,$ that depends
  only on the ratio, that for every $x\in X,$ we have $\mathcal{L}
  (B(x,\mathfrak{D}_{2}))\leq \bar{C}
  (\mathfrak{D}_{2}/\mathfrak{D}_{1})
  \mathcal{L}(B(x,\mathfrak{D}_{1})).$
\end{definition}

In \cite{HKST15SSMM,Hein01LAMS} one can find more information on
doubling spaces and measures. The existence of a doubling measure in
every complete doubling metric space is provided in
\cite[Thm.\,13.3]{Hein01LAMS}. We are mostly interested in
$X=\mathbb{R}^{d}$ or $X=\Omega,$ where $\Omega$ is a compact subset
of $\mathbb{R}^{d}$ with Lipschitz boundary, in which case the
Lebesgue measure is doubling. We are also interested in manifolds of
finite dimension  with  lower bounds on the Ricci curvature,
where the volume measure is doubling, see \cite{Sturm2006,Sturm2006a}.

\begin{definition}[Locally doubling measure]\label{locally doubling}
  In a metric space $(X,\sfd_{X}),$ a Borel measure $\mathcal{L}$ is
  called \emph{locally doubling}, if for every $M>0$ and $M\geq
  \mathfrak{D}_{2}\geq \mathfrak{D}_{1}>0$ there exists a constant
  $\bar{C}_{M}(\mathfrak{D}_{2}/\mathfrak{D}_{1})\geq 1$ that depends
  only on the ratio $ \mathfrak{D}_{2}/\mathfrak{D}_{1}$ and on the
  upper bound $M$ such that for every $x\in X$ we have
  $\mathcal{L}(B(x,\mathfrak{D}_{2}))\leq
  \bar{C}_{M}(\mathfrak{D}_{2}/\mathfrak{D}_{1})
  \mathcal{L}(B(x,\mathfrak{D}_{1})).$
\end{definition}

Since for our result it is easier to work with finite reference
measures, we provide the following useful lemma, where we exchange the
global doubling property with finiteness of the reference measure.

\begin{lemma}
  For every doubling measure $\widetilde{\mathcal{L}}$ we can find a
  finite locally doubling measure $\mathcal{L}$ that is equivalent to
  $\widetilde{\mathcal{L}}$
  (i.e. $\widetilde{\mathcal{L}}\ll\mathcal{L}$ and
  $\mathcal{L}\ll\widetilde{\mathcal{L}}$ ).
\end{lemma}
\begin{proof}
  For some point $x_{\mathfrak{a}}\in X,$ we define
  $\mathcal{L}(dx)=\frac{1}{(1+\bar{C}(2))^{2\sfd(x_{\mathfrak{a}},x)}}
  \widetilde{\mathcal{L}}(dx).$ For the finiteness of $\mathcal{L},$
  we observe that
\begin{equation}
\begin{split}
 \mathcal{L}(X)\!&=\!\sum_{i=0}^{\infty}\!
   \int_{\overline{B(x_{\mathfrak{a}},i+1)}\setminus
   B(x_{\mathfrak{a}},i)}\!\!\frac{1}{(1 {+}
          \bar{C}(2))^{2\sfd(x_{\mathfrak{a}},x)}}
     \widetilde{\mathcal{L}}(\d x)\!\leq\!
        \sum_{i=0}^{\infty}\!\int_{\overline{B(x_{\mathfrak{a}},i{+}1)}}\!\frac{1}{(1{+}\bar{C}(2))^{2i}}\widetilde{\mathcal{L}}(\d
        x)\\
&\leq\sum_{i=0}^{\infty}\frac{\mathcal{L}(\overline{B(x_{\mathfrak{a}},i{+}1)})}
{(1{+}\bar{C}(2))^{2i}} \leq \mathcal{L}(B(x_{\mathfrak{a}},1)) 
 \sum_{i=0}^{\infty}\frac{\bar{C}(2)^{i+2}}
 {(1{+}\bar{C}(2))^{2i}}<\infty,  
\end{split}
\end{equation}
where $\bar{C}(2)$ is the doubling constant for $\mathcal{L}.$
We also have
\[
\frac{\mathcal{L}B(x,\mathfrak{D}_{2})}{\mathcal{L}B(x,\mathfrak{D}_{1})}\leq\frac{\tilde{\mathcal{L}}B(x,\mathfrak{D}_{2})}{\tilde{\mathcal{L}}B(x,\mathfrak{D}_{1})}\frac{(1+
  \bar{C}(2))^{2(\sfd(x_{\mathfrak{a}},x)+\mathfrak{D}_{1})}}{(1+
  \bar{C}(2))^{2(\sfd(x_{\mathfrak{a}},x)-\mathfrak{D}_{2})}}\leq
\bar{C}(\mathfrak{D}_{2}/\mathfrak{D}_{1})(1+\bar{C}(2))^{2(\mathfrak{D}_{1}+\mathfrak{D}_{2})}. 
\]
Therefore for $M>0,$  we conclude that $\mathcal{L}$ is locally
doubling with
constant $\bar{C}_{M}(\mathfrak{D}_{2}/\mathfrak{D}_{1}):=
\bar{C}(\mathfrak{D}_{2}/\mathfrak{D}_{1})(1+\bar{C}(2))^{6M}$, which
proves the result. 
\end{proof}

For a finite, locally doubling measure $\mathcal{L}$  and
$\delta \in (0,1)$ we define the set
\begin{equation}
\label{eqdef:Mdelta}
\overline{\M}^{\mathcal{L}}_{\delta}(X)=\bigg\{\mu\in\M(X) : \mu\ll\mathcal{L}, \  \delta\leq\frac{d\mu}{d\mathcal{L}}(x)\leq\frac{1}{\delta}, \text{ for } \mathcal{L}\text{-a.e.\ } x\in X \bigg\}.
\end{equation}
For positive numbers $\sfd_{1},\sfd_{2},$ we also define
\begin{equation}
\label{eqdef:Md}
	\widetilde{\M}^{\mathcal{L}}_{\sfd_{1},\sfd_{2}}(X)=\bigg\{\mu\in\M(X)
          : \forall x\in X:\ \sfd_{2}\leq
	\frac{\mu\left(B\left(x,\sfd_{1}\right)\right)} 
         {\mathcal{L}(B\left(x,\sfd_{1}\right))} 
	\leq\frac{1}{\sfd_{2}} \bigg\}.
\end{equation} 
It is straightforward to see that
$\overline{\M}^{\mathcal{L}}_{\delta} (X) \subset
\widetilde{\M}^{\mathcal{L}}_{\sfd_{1},\delta}(X).$ Furthermore all
elements in $\overline{\M}^{\mathcal{L}}_{\delta}(X)$ have total mass
bounded by $\frac{1}{\delta}\mathcal{L}(X)$. The reason that we are
using these two sets instead of just of one of them is that neither is
geodesically closed in $(\M(X),\HK_\dd)$. However, as will be proved
later, for each $\delta>0$ we can find $\tilde{d}_{1},\tilde{d}_{2}>0$
such that for every 
$\mu_{0},\mu_{1}\in\overline{\M}^{\mathcal{L}}_{\delta}(X)$ we have
$\mm_{01}(t)\in \widetilde{\M}^{\mathcal{L}}_{\sfd_{1},\sfd_{2}}(X)$ for all
$t\in [0,1].$

\begin{theorem}[$K$-semiconcavity for $(\M(X),\HK_{\dd}) $]\label{thm:K.Semi} 
  Let $(X,\sfd_{X})$ be a doubling metric space. We also assume that $(X,\sfd_{X})$ satisfies $K$-semiconcavity on every ball $B\left(x,\frac{\pi}{2\dd}\right).$ Finally, let $\mathcal{L}$ be a finite, locally
  doubling measure, and $\overline{\M}^{\mathcal{L}}_{\delta}(X)$ as
  in \eqref{eqdef:Mdelta}. Then, there exists $K'\in\mathbb{R},$ that depends only on $K,\delta$ such that $(\M(X),\HK_{\dd})$ is $K'$-semiconcave on  $\overline{\M}^{\mathcal{L}}_{\delta}(X)$ 
\end{theorem}

The result is based on two facts. The first one is Corollary
\ref{Ksemicor}, i.e.\ that for $R_{1},R_{2}>0$ and
$0<\hat{\mathfrak{D}}<\frac{\pi}{2}$ it exists a $K'\in\mathbb{R}$
that depends only on $R_{1},R_{2},\hat{\mathfrak{D}},K,\dd$ such that
for every $x \in \YX$ the space $(\mfC,\sf\sfd_{\mfC,\dd})$
satisfies $K'$-semiconcavity on $B_{\dd\sf\sfd_{\YX}}\left(x,
  \mathfrak{D}\right)\times [R_{1},R_{2}]$. The second is that when
two measures, are ``uniform'' enough, and have bounded densities with
respect to each other, then the transport happens in distances less
than $\frac{\mathfrak{D}}{\dd},$ for some $\mathfrak{D}$ with
$\mathfrak{D}<\pi/{2},$ and also the densities with respect to the
optimal plan are bounded. The result is established via of
several intermediate results.

\begin{lemma}\label{uniform bounds of density} Let $(X,\sfd_{X})$ be
  doubling, $\mathcal{L}$ a finite locally doubling measure, and
  $\widetilde{\M}^{\mathcal{L}}_{\sfd_{1},\sfd_{2}}(X),$ as in
  \eqref{eqdef:Md} for $0< \sfd_{1}<\frac{\pi}{2\dd}$ and
  $\sfd_{2}>0$. Then, there 
  exists $0<C_{\min}\leq C_{\max},$ such that for every
  $\mu_{0},\mu_{1}\in \widetilde{\M}^{\mathcal{L}}_{\sfd_{1},\sfd_{2}}(X)$
  and any optimal plan $\bfH_{01}$ for
  $\LET_{\dd}(\cdot;\mu_0,\mu_1)$ we have
\begin{equation}
\label{bounds} 
C_{\min}\leq\sigma_{i}(x_{i})\leq C_{\max},\hspace{10pt}
\eta_{i}\text{-a.e.\ } 
\end{equation}
where $\eta_i= \pi^i_\# H_{01}= \sigma_i \mu_i$ for $i=0,1$.
 Furthermore, any transportation happens in distances strictly less than some $\frac{\pi}{2\dd},$ i.e. it exists $\mathfrak{D}<\frac{\pi}{2}$ that depends only on $\sfd_{1},\sfd_{2},$ such that $\dd\sfd_{X}(x_{0},x_{1})\leq \mathfrak{D}$  for $H_{01}$ almost every $(x_{0},x_{1}).$ 
\end{lemma}
\begin{proof}
  By the optimality conditions, we know that there exist sets
  $A_{0},A_{1}$ with 
  $\mu_{0} (X{\setminus}A_{0}) =\eta_{0}(X{\setminus} A_{0}) = 
 \mu_{1}(X {\setminus}A_{1})=\eta_{1}(X{\setminus}A_{1})=0,$ 
  such that
\begin{equation}
\label{a}\sigma_{0}(x_{0})\sigma_{1}(x_{1})\geq
          \cos_{\frac{\pi}{2}}^{2}\left(\dd\sfd_{X}(x_{0},x_{1})\right)
          \hspace{8pt}\text{in} \hspace{8pt}A_{0}\times A_{1}. 
\end{equation}
By dividing with $\sigma_{1}(x_{1})$ and integrating with respect to
$\mu_{0}$ on $B(x_{1},\sfd_{1}),$ we obtain
\begin{equation}
 \label{b}
 \eta_{0}(B(x_{1},\sfd_{1}))\geq
 \frac{\cos_{\frac{\pi}{2}}^{2}\left(\dd
     \sfd_{1}\right)}{\sigma_{1}(x_{1})} 
 \mu_{0}(B(x_{1},\sfd_{1}))\geq
 \frac{\cos_{\frac{\pi}{2}}^{2}\left(\dd \sfd_{1}\right)}
 {\sigma_{1}(x_{1})} \sfd_{2}\mathcal{L}(B(x_{1},\sfd_{1})), 
\end{equation}
for every $x_{1}\in A_{1}.$ Using Lemma \ref{totaleta} we
find  
\begin{equation}
\label{bb}\begin{split}
\eta_{0}(B(x_{1},\sfd_{1}))
& \leq
\sqrt{\mu_{0}(B(x_{1},\sfd_{1}))\mu_{1}(B(x_{1},\sfd_{1})_{\frac{\pi}{2\dd}})}\\
& \leq
\sqrt{\mu_{0}(B(x_{1},\sfd_{1}))}\sqrt{C\left( 
  \left(\frac{\pi}{2\dd}{+}\sfd_{1}\right)/\sfd_{1}\right)\sup_{y\in
    B_{\frac{\pi}{2\dd}}(x_{1},\sfd_{1})}\mu_{1}(B(y,\sfd_{1}))}\\
&  \leq \frac{1}{\sfd_{2}} \sqrt{\mathcal{L}(B(x_{1},\sfd_{1}))} 
\sqrt{C\left(\left(\frac{\pi}{2\dd} {+} \sfd_{1}\right)/\sfd_{1}\right) 
 \sup_{y\in
   B_{\frac{\pi}{2\dd}}(x_{1},\sfd_{1})}\mathcal{L}(B(y,\sfd_{1}))}\\
 &  \leq \frac{1}{\sfd_{2}} \sqrt{\mathcal{L}(B(x_{1},\sfd_{1}))} 
 \sqrt{C\left(\left(\frac{\pi}{2\dd} {+} \sfd_{1}\right)/\sfd_{1}\right)
  \mathcal{L}(B(x_{1},\sfd_{1})_{\frac{\pi}{2\dd}})}\\
 & \leq \frac{1}{\sfd_{2}} \mathcal{L}(B(x_{1},\sfd_{1})) 
 \sqrt{C\left(\left(\frac{\pi}{2\dd} {+} \sfd_{1}\right) 
 /\sfd_{1}\right)\bar{C}_{\frac{\pi}{\dd}}\left( 
\left(\frac{\pi}{2\dd}+\sfd_{1}\right)/\sfd_{1}\right)},
\end{split}
\end{equation} 
where the constant
$C\left(\left(\frac{\pi}{2\dd} {+} \sfd_{1}\right)/\sfd_{1}\right)$ is as in the
definition of doubling metric spaces to cover a set of radius
$\frac{\pi}{2\dd}+\sfd_{1}$ by balls of radius $\sfd_{1},$ and
$\bar{C}_{\frac{\pi}{\dd}}\left(\left(\frac{\pi}{2\dd} {+} \sfd_{1}\right)/\sfd_{1}\right)$
is the doubling measure constant for radius less than
$\frac{\pi}{\dd}.$ We set $\widetilde{C}= \sqrt{C\left(\left(
      \frac{\pi}{2\dd} {+} \sfd_{1}\right)/\sfd_{1}\right)
  \bar{C}_{\frac{\pi}{\dd}} \left(\left(\frac{\pi}{2\dd} {+} \sfd_{1}\right)
    /\sfd_{1}\right)},$ and by combining \eqref{b} and \eqref{bb}, we
 derive the lower bound 
\begin{equation}
\sigma_{1}(x_{1}) \geq  {\cos_{\frac{\pi}{2}}^{2}\left(\dd
                          \sfd_{1}\right)d^{2}_{2}}/\widetilde{C} 
\hspace{8pt}\text{in} \hspace{8pt} A_{1}.
\end{equation}
Now, by the second optimality condition we have 
\begin{equation}
\sigma_{0}(x_{0})= \frac{\cos_{\frac{\pi}{2}}^{2}
  (\dd\sfd_{X}(x_{0},x_{1}))} {\sigma_{1}(x_{1})} \leq
\frac{\widetilde{C}}{\cos_{\frac{\pi}{2}}^{2} \left(\dd
    \sfd_{1}\right)d^{2}_{2}}, \hspace{8pt}  \bfH_{01}\text{-a.e.\ in}
\hspace{8pt} A_{0}\times A_{1}.
\end{equation} 
	
By interchaning the roles of $\sigma_0$ and $\sigma_1$ and combining all the inequalities we arrive at 
\begin{equation*}
 C_\mathrm{min}:=  \frac{\cos_{\frac{\pi}{2}}^{2} \left(\dd
            \sfd_{1}\right)d^{2}_{2}}{\widetilde{C}}\leq
        \sigma_{i}(x_{i}) \leq
        \frac{\widetilde{C}}{\cos_{\frac{\pi}{2}}^{2}\left(\dd
            \sfd_{1}\right)d^{2}_{2}} =: C_\mathrm{max}, \hspace{8pt}
        \bfeta_{i}\text{-a.e.\ in} \hspace{8pt} A_{i},
\end{equation*}	
 which is the desired result. 

Now by visiting the second optimality condition one more time, we get that $\cos_{\frac{\pi}{2}}^{2}
(\dd\sfd_{X}(x_{0},x_{1}))$ is bounded from below by a positive constant that depends only on the bounds of 
$\sigma_{i},$ for $ \bfH_{01}$-a.e. $(x_{0},x_{1}).$ Therefore by continuity of the cosine, we get that  it exists
$\mathfrak{D}<\frac{\pi}{2}$ such that  for every $\mu_{0}, \mu_{1}\in \widetilde{\M}^{\mathcal{L}}_{\sfd_{1},\sfd_{2}}(X),$ we have $\dd\sfd_{X}(x_{0},x_{1})\leq \mathfrak{D},$  for $ \bfH_{01}$-a.e. $(x_{0},x_{1}).$ 

\end{proof}

 The next result shows that the geodesic closure of $
\overline{\M}^{\mathcal{L}}_{\delta}(X)$ is contained in
$\widetilde{\M}^{\mathcal{L}}_{\sfd_{1},\sfd_{2}}(X)$ for suitably chosen
$\sfd_1,\sfd_2$.

\begin{lemma}\label{MdeltatoMd1d2}
  Let $(X,\sfd_{X})$ be doubling, $\mathcal{L}$ be a finite locally doubling
  measure, and $\overline{\M}^{\mathcal{L}}_{\delta}(X)$ be as in
  \eqref{eqdef:Mdelta}. Then, for each $\delta>0$ there exist
  $\sfd_{1} \in (0,\frac{\pi}{2\dd})$ and $\sfd_2>0$ such that any
  constant-speed geodesic $\mm_{01}$ connecting $\mu_{0}$ to $\mu_{1}$
  with $\mu_0,\mu_1\in\overline{\M}^{\mathcal{L}}_{\delta}(X)$
  satisfies $\mm_{01}(t)\in \widetilde{\M}
  ^{\mathcal{L}}_{\sfd_{1},\sfd_{2}}(X)$ for all $t\in [0,1].$
\end{lemma}
\begin{proof}
  It is straightforward to see that
  $\overline{\M}^{\mathcal{L}}_{\delta}(X)$ is a subset of some
  $\widetilde{\M}^{\mathcal{L}}_{\min\{\frac{\pi}{4\dd},\frac{1}{2}\},\delta}(X).$
  Therefore, by Lemma \ref{uniform bounds of density}, we find
  $\tilde{d} \in {]0,\pi/2[}$, which depends 
  only on $\delta$, such that $\dd\sfd_{X}(x_{0},x_{1})\leq
  \tilde{d}<\frac{\pi}{2}$ holds for $\bfH_{01}$-a.a. $(x_0,x_1)$. 
  Let $\Lambda_{01}$ be the optimal plan in the cone definition, and $\bm{\Lambda}_{0\rightarrow 1}$ the
occurring plan on the geodesics. For $x_{0}\in X,$ we have
\begin{equation} \label{aaa}
\begin{split}
&\mm_{01}\left(t;B\left(x_{0}, \frac{\pi+ 2\tilde{d}}{4\dd}\right)\right)\geq
\mfP\left[(e_{t})_{\sharp}\left(\bm{\Lambda}_{0\rightarrow 1}\right) \res{\left\{\xx_{01}(0)\in B\left(x_{0},\frac{\pi- 2\tilde{d}}{4\dd}\right)\right\}}\right](X) ,
\end{split}
\end{equation}
since all points in
$B\left(x_{0},\frac{\pi- 2\tilde{d}}{4\dd}\right),$ will be transfered
at most distance $\frac{\tilde{d}}{\dd}$. Therefore will remain in a ball of radius
$B\left(x_{0},\frac{\pi+ 2\tilde{d}}{4\dd}\right).$ Now
$\tilde{\mm}_{01}(t)=\mfP\left[(e_{t})_{\sharp}\left(\bm{\Lambda}_{0\rightarrow
      1}\right) \res{\left\{\xx_{01}(0)\in B\left(x_{0},\frac{\pi-
          2\tilde{d}}{4\dd}\right)\right\}}\right]$ is a geodesic
starting from $\mu_{0}\res{B\left(x_{0},\frac{\pi-
      2\tilde{d}}{4\dd}\right)}.$ Let
$\tilde{m}(t)=(\tilde{\mm}_{01}(t))(X).$ By \eqref{rr prop3} and
recalling \eqref{cone structure of HK} we get  $$\tilde{m}(t)\geq
(1{-}t) ^{2}\tilde{m}(0) +t^{2}\tilde{m}(1),$$ which in turn for
$t\in\left[0,\frac{1}{2}\right]$, gives  
\begin{equation}\label{bbb}
\begin{split}
&\mfP\left[(e_{t})_{\sharp}\left(\bm{\Lambda}_{0\rightarrow 1}\right)\res{\left\{\xx_{01}(0)\in B\left(x_{0},\frac{\pi- 2\tilde{d}}{4\dd}\right)}\right\}\right](X)\\&\geq
 (1{-}t) ^{2} \mfP\left[(e_{0})_{\sharp}\left(\bm{\Lambda}_{0\rightarrow 1}\right)\res{\left\{\xx_{01}(0)\in B\left(x_{0},\frac{\pi- 2\tilde{d}}{4\dd}\right)\right\}}\right] (X) 
\\&\geq (1{-}t) ^{2}\mu_{0}\left(B\left(x_{0},\frac{\pi-
      2\tilde{d}}{4\dd}\right)\right)\geq\frac{1}{4}\mu_{0}\left(B\left(x_{0},\frac{\pi-
      2\tilde{d}}{4\dd}\right)\right)\geq\frac{\delta}{4}\mathcal{L}\left(B\left(x_{0},\frac{\pi-
      2\tilde{d}}{4\dd}\right)\right)\\&\geq \frac{
  \delta}{4\widetilde{C}_{M}\left(\left(\frac{\pi+
        2\tilde{d}}{4\dd}\right)/\left(\frac{\pi-
        2\tilde{d}}{4\dd}\right)\right)}\mathcal{L}\left(B\left(x_{0},\frac{\pi+
      2\tilde{d}}{4\dd}\right)\right). \end{split} 
\end{equation}
Combining \ref{aaa} and \ref{bbb}, we get that 

\begin{equation}
\frac{\mm_{01}\left(t;B\left(x_{0}, \frac{\pi + 2\tilde{d}}{4\dd}\right)\right)}{\mathcal{L}\left(B\left(x_{0},\frac{\pi + 2\tilde{d}}{4\dd}\right)\right)}\geq \frac{\delta}{4\widetilde{C}_{M}\left(\left(\frac{\pi+ 2\tilde{d}}{4\dd}\right)/\left(\frac{\pi- 2\tilde{d}}{4\dd}\right)\right)}
\end{equation}
We work in the same manner with the roles of $\mu_{0}$ and $\mu_{1}$
being reversed to recover the same estimate on the interval $[1/2,1],$
and this way we retrieve the lower bound with $\sfd_{1}=\frac{\pi +
  2\tilde{d}}{4\dd}$ and $\sfd_{2}=
\frac{\delta}{4\widetilde{C}_{M}\left(\left(\frac{\pi+
        2\tilde{d}}{4\dd}\right)/\left(\frac{\pi-
        2\tilde{d}}{4\dd}\right)\right)}$. 

In a similar manner by
utilizing \eqref{rr prop2}  instead of \eqref{rr prop3}, we
obtain a corresponding upper bound.  
\end{proof}

\begin{lemma}
\label{exist Lambdas}
Let $(X,\sfd_{X})$ be doubling, $\mathcal{L}$ a finite, locally
doubling measure, and let
$\widetilde{\M}^{\mathcal{L}}_{\sfd_{1},\sfd_{2}}(X)$ be as in
\eqref{eqdef:Md}. Then, there exist $R_{\min},R_{\max}>0$ that depend
on $\sfd_{1},\sfd_{2},$ such that for $\mu_{0},\mu_{1}$ with
$\mm_{01}(t)\in\widetilde{\M}^{\mathcal{L}}_{\sfd_{1},\sfd_{2}}(X)$
and $\mu_{2}\in\widetilde{\M}^{\mathcal{L}}_{\sfd_{1},\sfd_{2}}(X)$
we can find measures $\lambda_{0},\lambda_{1},\lambda_{2},\lambda_{t}
\in \mathcal{P}_{2}(\mfC [R_{\min},R_{\max}]) $ with
\[
\mathfrak{P}\lambda_{i}=\mu_{i},\hspace{8pt}\mathfrak{P}\lambda_{t} 
=\mm_{01}(t),\hspace{8pt}
\sfW_{\sfd_{\mfC,\dd }} (\lambda_{i},\lambda_{t}) =
\HK_{\dd}(\mu_{i},\mm_{01}(t)) \quad \text{ for }
i=0,1,2.
\]
\end{lemma}
\begin{proof}
   For $i=0,1,2,$ let
  $\bfH_{ti}$ be the optimal plan in the definition of
  $\LET_{\dd}(\cdot;\mu_{i},\mm(t)),$ and
  $\sigma_{i}^{ti},\sigma_{t}^{ti}$ the densities of
  $\eta^{ti}_{i},\eta^{ti}_{t}$ with respect to $\mu_{i},\mu_{t}.$
  Let now the plans
\[
\Lambda_{ti}(\d z_{i}, \d z_{t}) = \delta_{\sqrt{\sigma^{ti}_{i}(x_{i})}}(\d r_{i})
\delta_{\sqrt{\sigma^{ti}_{t}(x_{t})}} (\d r_{i})\bfH_{ti}(\d x_{i}, \d x_{t}).
\]
For $i=0,1,2,$ we take $\theta^{ti}([z_{t},z_{i}])
=\sqrt{\frac{\sigma_{t}^{ti}(x_{t})} {\sigma_{t}^{t0}(x_{t})}},$ and
we define $\tilde{\Lambda}_{ti}=\dil{\theta^{ti}}(\Lambda^{ti}).$
Finally we set $\lambda_{i}=\pi^i_{\sharp}\tilde{\Lambda}_{ti}$ for
$i=0,1,2.$ It is straightforward to see that
$r_{i}=\sqrt{\frac{\sigma_{t}^{ti}(x_{t})}{\sigma_{t}^{t0}(x_{t})}}\sqrt{\sigma_{i}^{ti}(x_{i})}$
for $\lambda_{i}$-a.e. $z_{i}=[x_{i},r_{i}],$ with $i=0,1,2.$ By Lemma
\ref{uniform bounds of density}, we now obtain
\[
 R_{\min} :=  
 \frac{C_{\min}}{\sqrt{C_{\max}}}\leq r_{i} \leq
 \frac{C_{\max}}{\sqrt{C_{\min}}}  =: R_{\max} 
 \hspace{8pt}\text{for}\hspace{8pt}
 \lambda_{i}\text{-a.e.}\hspace{8pt} z_{i}=[x_{i},r_{i}],\hspace{8pt}
 \text{for}\hspace{8pt} i=0,1,2.
\] 
This proves the the claim that all $\lambda_i$ are supported in
$\mfC[R_{\min},R_{\max} ]$. 
\end{proof}	

Now we are able to conclude the proof of the main result of this
section.\medskip

\noindent
\begin{proof}[Proof of Theorem \ref{thm:K.Semi}] 
	By Lemma \ref{MdeltatoMd1d2} there exists
	$0<\sfd_{1}<\frac{\pi}{2\dd}$ and $0<\sfd_{2}$ such that every geodesic
	$\mm_{01}$ connecting
	$\mu_{0},\mu_{1}\in\overline{\M}^{\mathcal{L}}_{\delta}(X)$
	satisfies $\mm_{01}(t)\in
	\widetilde{\M}^{\mathcal{L}}_{\sfd_{1},\sfd_{2}}(X)$ for all $ t\in [0,1].$
	We also have $\mu_{2}\in
	\widetilde{\M}^{\mathcal{L}}_{\sfd_{1},\sfd_{2}}(X) \supset
	\overline{\M}^{\mathcal{L}}_{\delta}(X)$.
	We would like to utilize the equivalent definition of $K-$semiconcavity given in \eqref{equivksemi}, therefore we will just take $\tilde{\mu}_{0}=\mm_{01}(t_{1}), \tilde{\mu}_{1}=\mm_{01}(t_{2}),$ for $t_{1},t_{2}\in[0,1],$ and $\tilde{\mm}_{01}(t)=\mm_{01}(t(t_{2}-t_{1})+t_{1}).$  
	By Lemma \ref{exist Lambdas}, there exists $R_{\min},R_{\max}$ that depend on
  $\sfd_{1},\sfd_{2},$ and therefore on $\delta,$ such that for every
  $\tilde{\mu}_{0},\tilde{\mu}_{1},\tilde{\mu}_{2}\in\widetilde{\M}^{\mathcal{L}}_{\sfd_{1},\sfd_{2}}(X)$
  and $0<t<1$ we can find measures
  $\lambda_{0},\lambda_{1},\lambda_{2},\lambda_{t} \in
  \mathcal{P}_{2}(\mfC [R_{\min},R_{\max}]) $ with 
  \begin{equation}
    \label{eq:La.mu}
      \mathfrak{P}\lambda_{i}=\tilde{\mu}_{i},\hspace{8pt}\mathfrak{P}\lambda_{t} 
   =\tilde{\mm}_{01}(t),\hspace{4pt} \text{and} \hspace{4pt} 
  \sfW_{\sfd_{\mfC,\dd}}(\lambda_{i},\lambda_{t}) 
  =\HK_{\dd}(\tilde{\mu}_{i},\tilde{\mm}_{01}(t)),\hspace{4pt} i=0,1,2.
  \end{equation}
Using the geodesic property of $\tilde{\mm}_{01}$ yields 
\begin{align*}
\sfW_{  \sfd_{\mfC,\dd}}(\lambda_{0},\lambda_{t}) + \sfW_{
    \sfd_{\mfC,\dd}}(\lambda_{1},\lambda_{t}) &=
  \HK_{\dd}(\mu_{0},\tilde{\mm}_{01}(t))+\HK_{\dd}(\mu_{1},\tilde{\mm}_{01}(t))\\
&= \HK_{\dd}(\tilde{\mu}_{0},\tilde{\mu}_{1})\!\leq\!
  \sfW_{ \sfd_{\mfC,\dd}}(\lambda_{0},\lambda_{1}).
\end{align*} 
Hence, it is
  straightforward to see that there exists a geodesic $\lala_{01}$ connecting
  $\lambda_{0},\lambda_{1},$ such that $\lala_{01}(t)=\lambda_{t}.$
  Furthermore, by \cite[Thm.\,6]{Lisini2006a} there is a plan
  $\bfLambda_{0\rightarrow 1}$ on the geodesics such that
  $\Lambda_{ts}:=(e_{t},e_{s})_{\sharp}\bfLambda_{0\rightarrow1}$
  is an optimal plan between $\lala(t)$ and $\lala(s).$ Now, by using
  a gluing lemma, we can find a plan $\bfLambda^{0 \to
    1}_{2t}$ in $\mathcal{P}((C[0,1];\mfC)\times\mfC),$ such
  that $\Lambda_{01}=(e_{0},e_{1})_{\sharp}\left(\pi^{0\rightarrow
      1}_{\sharp}\bfLambda^{0 \to 1}_{2t}\right),$
  and $(e_{t}(\pi^{0\rightarrow 1})\times
  I)_{\sharp}\bfLambda^{0 \to 1}_{2t}$ is an optimal plan
  for $ \sfW_{\sfd_{\mfC ,\dd}}(\lambda_{2},\lala_{01}(t)).$ Finally by applying the last part
  of Lemma \ref{uniform bounds of density}, we get the existence of a $\mathfrak{D}<\frac{\pi}{2}$ such that $\dd\sfd_{X}(x_{2},x_{t})<\mathfrak{D}$ for $(e_{t}(\pi^{0\rightarrow 1})\times  I)_{\sharp}\bfLambda^{0 \to 1}_{2t}$ almost every $(z_{2},z_{t}),$ similarly $\dd\sfd_{X}(x_{0},x_{1})<\mathfrak{D}$ for $\Lambda_{01}$ almost every $[z_{0},z_{1}].$ Therefore, for $\bfLambda^{0 \to
  	1}_{2t}$ almost every $(z_{2}, \zz(\cdot,z_{0},z_{1})),$ where $\zz(\cdot,z_{0},z_{1})$ is a geodesic connecting $z_{0},z_{1},$ we have $x_{0},x_{1},x_{2}, \bar{\xx}(t,z_{0},z_{1})\in B\left(\bar{\xx}(t,z_{0},z_{1}),d\right).$ By Lemma \ref{mainKseminconcavitylemma} we get a $K'$ such that 

\begin{equation}
\label{eq:C.Kconcave}
  \sfd^{2}_{\mfC ,\dd}(z_{2},\zz(t,z_{0},z_{1})) +
   K't(1{-}t)\sfd^{2}_{\mfC ,\dd}(z_{0},z_{1})\geq
   (1{-}t) \sfd^{2}_{\mfC ,\dd}(z_{2},z_{0})+ t\, \sfd^{2}_{\mfC ,\dd}(z_{2},z_{1}),
\end{equation}
for $\bfLambda^{0 \to 1}_{2t}$ almost every $(z_{2}, \zz(\cdot,z_{0},z_{1})).$ By integrating with respect to $\bfLambda^{0 \to 1}_{2t},$ we
find
\begin{equation}
    \sfW_{\sfd_{\mfC ,\dd}}^{ 2}(\lambda_{2},\lala_{01}(t)) +
    K't (1{-}t) \sfW_{\sfd_{\mfC ,\dd}}^{ 2}
    (\lambda_{0},\lambda_{1})\geq   
   (1{-}t) \sfW_{\sfd_{\mfC,\dd}}^{ 2}(\lambda_{2},\lambda_{0})+ 
   t \, \sfW_{\sfd_{\mfC,\dd}}^{ 2} (\lambda_{2},\lambda_{1}).
\end{equation}
Using \eqref{eq:La.mu} we find the desired semiconcavity, and Theorem
\ref{thm:K.Semi} is proved. 
\end{proof}

To obtain a similar result for the Spherical
Hellinger--Kantorovich distance $\SHK_{\dd} $ we define 
\[
\mathcal{P}^{\mathcal{L}}_{\delta}(X) := 
\left\{\nu\in\mathcal{P}(X) \; : \; \nu=\frac{\mu}{\mu(X)},\   \mu \in
  \overline{\M}^{\mathcal{L}}_{\delta}(X)\right\}\supset\mathcal{P}(X)\cap\overline{\M}^{\mathcal{L}}_{\delta}(X) 
\]
as analog of $\overline{\M}^{\mathcal{L}}_{\delta}(X)$, see
\eqref{eqdef:Mdelta}.  
Now for the Spherical Hellinger-Kantorovich space 
$(\mathcal{P}(X),\SHK_{\dd}) $  satisfies the following analog of
Theorem \ref{thm:K.Semi} for $(\M(X),\HK_{\dd}) $.  

\begin{theorem}[ $K$-semiconcavity for $(\mathcal{P}(X),\SHK_{\dd})$] 
\label{thm:SK.Semi} 
Let $(X,\sfd_{X})$ be a doubling metric space and assume that
$(X,\sfd_{X})$ satisfies $K$-semiconcavity on every ball
$B\left(x,\frac{\pi}{2\dd}\right).$ Furthermore, let $\mathcal{L}$ be
a finite, locally doubling measure, and
$\overline{\M}^{\mathcal{L}}_{\delta}(X)$ as in
\eqref{eqdef:Mdelta}.  Then, there exists $K'\in\mathbb{R}$,
which depends only on $K,\delta,\ell$, such that
$(\mathcal{P}(X),\SHK_{\dd})$ is $K'$-semiconcave on
$\mathcal{P}^{\mathcal{L}}_{\delta}(X).$
\end{theorem}
\begin{proof}
For  $\mu\in\overline{\M}^{\mathcal{L}}_{\delta}(X),$ we have  $\delta\mathcal{L}(X)\leq \mu(X)\leq \frac{1}{\delta}\mathcal{L}(X),$ therefore for $\nu=\frac{\mu}{\mu(X)},$ we have ${\delta^{2}}\leq\frac{d\nu}{d\mathcal{L}}(x)\leq\frac{1}{\delta^{2}}$. We get that $\mathcal{P}^{\mathcal{L}}_{\delta}(X)\subset \mathcal{P}(X)\cap\overline{\M}^{\mathcal{L}}_{\delta^{2}}(X).$ It is also trivial to see that it exists a $\mathcal{D}<\frac{\pi}{2},$ such that $\mathcal{P}(X)\cap\overline{\M}^{\mathcal{L}}_{\delta^{2}}(X)\subset B(\nu_{0},\mathcal{D}),$ for some $\nu_{0}\in \mathcal{P}(X)\cap\overline{\M}^{\mathcal{L}}_{\delta^{2}}(X)\subset B(\nu_{0},\mathcal{D}).$ Now we apply Corollary \ref{Ksemicor} with combination with Theorem \ref{thm:K.Semi}, and get the result.
\end{proof}

\paragraph*{Acknowledgment}
We would like to thank Anton Petrunin, Giuseppe Savar\'e, and Marios
Stamatakis for useful communication in different stages of this
project. We would like to especially thank Marios Stamatakis for
providing us a first proof of Proposition \ref{ThKsSZ}.

\newcommand{\etalchar}[1]{$^{#1}$}
\def\cprime{$'$}

\end{document}